\documentclass[final]{amsart}
\usepackage{caption}
\usepackage{subcaption}

\usepackage{mathabx,cite}
\usepackage[utf8]{inputenc} 
\usepackage[english]{babel}
\usepackage{amstext}
\usepackage{euscript}
\usepackage{bbm}
\usepackage{mathrsfs}
\usepackage{enumerate}
\usepackage{graphicx}
\usepackage{caption}
\usepackage{amscd}
\usepackage{verbatim}
\usepackage{amssymb}
\usepackage{amsthm}
\usepackage{amsmath}
\usepackage{amstext}
\usepackage{amsfonts}
\usepackage{latexsym}
\usepackage{mathtools} 
\usepackage{enumitem} 
\usepackage{esint}
\usepackage{accents} 
\usepackage[foot]{amsaddr} 

\usepackage[usenames,dvipsnames]{xcolor} 
\definecolor{dkblue}{RGB}{1,31,91} 
\usepackage[colorlinks=true, pdfstartview=FitV, linkcolor=dkblue, citecolor=dkblue, urlcolor=dkblue,hyperindex,breaklinks]{hyperref}   

\usepackage{todonotes}

\usepackage[color]{showkeys}

\newcommand{\eqdef }{\overset{\mbox{\tiny{def}}}{=}}
\newcommand{\eps}{\varepsilon}

\newcommand{\rth}{{\mathbb{R}^3}}

\newcommand{\stw}{{\mathbb{S}^2}}
\newcommand{\vphi}{\varphi}

\newcommand{\tI}{\tilde{I}}
\newcommand{\bnd}{\partial\Omega}
\newcommand{\Div}{\text{div}}

\theoremstyle{definition}
\newtheorem{theorem}{Theorem}

\newtheorem{corollary}[theorem]{Corollary}

\newtheorem{lemma}[theorem]{Lemma}
\newtheorem{proposition}[theorem]{Proposition}

\newtheorem{remark}[theorem]{Remark}

\numberwithin{equation}{section}
\numberwithin{theorem}{section}
\numberwithin{definition}{section}

\allowdisplaybreaks

\begin{document}


\title[Compactness for a general class of radiative transfer equations]{Compactness and existence theory for a general class of stationary radiative transfer equations}

\author[E. Dematt\`e]{Elena Dematt\`e}
\address{Institute for Applied Mathematics, University of Bonn, 53115 Bonn, Germany. \href{mailto:	dematte@iam.uni-bonn.de}{dematte@iam.uni-bonn.de} }

\author[J. W. Jang]{Jin Woo Jang}
\address{Department of Mathematics, Pohang University of Science and Technology (POSTECH), Pohang, South Korea (37673). \href{mailto:jangjw@postech.ac.kr}{jangjw@postech.ac.kr} }

\author[J. J. L. Vel\'azquez]{Juan J. L. Vel\'azquez}
\address{Institute for Applied Mathematics, University of Bonn, 53115 Bonn, Germany. \href{mailto:velazquez@iam.uni-bonn.de}{velazquez@iam.uni-bonn.de} }

\begin{abstract}In this paper, we study the steady-states of a large class of stationary radiative transfer equations in a $C^1$ convex bounded domain. Namely, we consider the case in which both absorption-emission and scattering coefficients depend on the local temperature $T$ and the radiation frequency $\nu.$ The radiative transfer equation determines the temperature of the material at each point. The main difficulty in proving existence of solutions is to obtain compactness of the sequence of integrals along lines that appear in several exponential terms. We prove a new compactness result suitable to deal with such a non-local operator containing integrals on a line segment. 
On the other hand, to obtain the existence theory of the full equation with both absorption and scattering terms we combine the compactness result with the construction of suitable Green functions for a class of non-local equations.
\end{abstract}
\thispagestyle{empty}

\maketitle 
\tableofcontents

\section{Introduction}
In this paper, we study the \textit{stationary radiative transfer equation} for a radiation intensity function $I_\nu: \Omega\times \rth \to [0,\infty)$ on a $C^1$ convex and bounded domain $\Omega \subset \rth,$ which takes the form
\begin{equation}\label{General rad eq}  
n\cdot \nabla_x I_\nu = \alpha^a_\nu(T) (B_\nu(T)-I_\nu )+ \alpha^s_\nu(T)\left[\int_{\mathbb{S}^2} K(n',n)I_\nu(x,n')dn'-I_\nu\right], 
\end{equation}
where it factors in both the excitation and de-excitation processes of gas molecules alongside photon scattering. This formulation is underpinned by the presumption of local thermodynamic equilibrium (LTE) of gas molecules. Specifically, the components $\alpha^a_\nu I_\nu $, $ \alpha^s_\nu I_\nu,\alpha^a_\nu B_\nu, \alpha^s_\nu \int_{\mathbb{S}^2} K(n',n)I_\nu(x,n')dn'$ represent absorption, scattering loss, emission from gas deexcitation, and scattering gain, respectively. Herein,
 \begin{equation}
       \label{blackbody}
      B_\nu= B_\nu(T)\eqdef\frac{2h\nu^3}{c^2}\frac{1}{e^{\frac{h\nu}{kT}}-1}
   \end{equation}
   symbolizes the Planck emission from a black body, while $I_\nu=I_\nu(x,n)$ signifies the radiation intensity at frequency $\nu$, located at position $x\in\Omega \subset \rth$ and oriented in direction $n\in \mathbb{S}^2$.

    Note that $B_\nu(T)$ is monotonically increasing in $T$ for each $\nu$ and $B_\nu(T)=0$ is equivalent to $T=0.$ 
   By making a change of variables $\nu\mapsto \zeta \eqdef\frac{h\nu}{kT} $, we obtain that 
\begin{multline}\label{integration of blackbody}
       \int_0^\infty B_\nu(T)d\nu = \int_0^\infty\frac{2h\nu^3}{c^2}\frac{1}{e^{\frac{h\nu}{kT}}-1}d\nu=\int_0^\infty\frac{2hk^3T^3}{h^3c^2}\frac{\zeta^3}{e^{\zeta}-1}\frac{kT}{h}d\zeta\\=\frac{\pi^4}{15}\frac{2hk^3T^3}{h^3c^2}\frac{kT}{h}=\sigma T^4,
   \end{multline}
   where we define $
       \sigma\eqdef \frac{2\pi^4k^4}{15h^3c^2}.
 $
 
The radiation energy flux at frequency $\nu$ can be articulated as:
$$\mathcal{F}_\nu=\mathcal{F}_\nu(x)=\int_{\mathbb{S}^2}nI_\nu(x,n)dn \in \rth.$$
The scattering kernel of the ``non-local" gain term of scattering has the property \begin{equation}\label{scattering assume}\int_{\mathbb{S}^2} K(n',n)dn=1.\end{equation} If the scattering is isotropic then it becomes simply $\alpha^s_\nu(T)I_\nu$ in \eqref{General rad eq}. 
The class of models is for the LTE situation. The temperature $T$ is well-defined at each point, and each coefficient $\alpha_\nu=\alpha_\nu(T)$ depends on the frequency of radiation $\nu$ and the local temperature $T$. The coefficient $\alpha_\nu$ can be considered as the spectral lines for each $\nu$ or, more generally, the averages of these processes.

Throughout the paper, we will study the existence theory of the general model \eqref{General rad eq} with  \eqref{scattering assume}. The assumption \eqref{scattering assume} implies that the scattering does not modify the frequency. We will consider the general case where the absorption-emission and the scattering coefficients can depend not just on the radiation frequency $\nu>0$ but also on the local temperature $T(x)$.   Another main assumption in this model above is that the non-elastic mechanisms yielding LTE in the gas molecules' distributions are extremely fast, and therefore the scattering cannot modify much the Boltzmann ratio between the different energy levels at each point. For more details, see \cite{Dematte2023,2109.10071}.
We write the total flux of radiation energy with frequency $\nu$ at $x$ as $$\mathcal{F}=\mathcal{F}(x)\eqdef \int_0^\infty d\nu \int_{\mathbb{S}^2}nI_\nu(x,n)dn.$$
At the Planck equilibrium, we have 
\begin{equation}\label{planck} I_\nu(x,n)=B_\nu(T)=\frac{2h\nu^3}{c^2}\frac{1}{e^{\frac{h\nu}{kT}}-1}.\end{equation} Then, the stationarity of the temperature at each point requires that the divergence of the total flux of radiation energy vanishes (cf. \cite{mihalas2013foundations,oxenius}); i.e.,
\begin{equation}\label{heat equation}\nabla_x\cdot \mathcal{F}(x)=0,\text{ at any }x\in\Omega.\end{equation}Hence, we will examine throughout this paper whether the temperature $T$ at each point can be determined uniquely by \eqref{General rad eq}  
and a suitable boundary condition for the radiation at $\partial \Omega$, if we impose the divergence-free total-flux condition \eqref{heat equation}; see also \cite{2109.10071} for the conservation law. 
The simplest boundary condition that one can impose is
\begin{equation}\label{incoming boundary}I_\nu(x,n)= g_\nu(n)\ge 0\ \text{ for } x\in\partial\Omega\text{ and }n\cdot n_x<0,\end{equation} where $n_x$ is the outward normal vector at the boundary point $x\in \partial \Omega$ and $\Omega $ is  with smooth boundary $\partial\Omega$.

 Throughout the paper, we assume $\Omega$ to be a convex domain with $C^1$-boundary and strictly positive curvature.    

The problem \eqref{General rad eq}, \eqref{heat equation}, and \eqref{incoming boundary} is considered in \cite{jang2023temperature} in the case where $\alpha_\nu^a$ and $\alpha_\nu^s$ are independent of the temperature. The main novelty of this paper is that we were able to extend the proof of the existence of solutions to \eqref{General rad eq}, \eqref{heat equation}, and \eqref{incoming boundary} for a general class of coefficients $\alpha_\nu^a$ and $\alpha_\nu^s$ which also depend on the temperature $T$. To obtain this result we will derive a compactness result for a large class of non-local operators including terms with the form  \begin{equation}\label{nlo}
    T(\cdot)\mapsto \exp{\left(-\int_{[x,\eta]}\beta(T(\xi))ds(\xi)\right)},\;x,\eta\in\Omega,
\end{equation}
where the integral is along the segment connecting $x$ to $\eta$. The compactness result that we obtain in this paper to prove the existence has some analogies with the classical averaging lemmas that have been extensively used in kinetic theory \cite{GOLSE1988110,refId0,Westdickenberg,DIPERNA1991271,DeVorePetrova}. Nevertheless, the currently available averaging lemmas including the treatment of line integrals in \cite{NouriArkeryd} do not seem to provide the compactness that we require. For this reason, we prove a new compactness result more suitable to deal with the non-local operator on a line segment with the form \eqref{nlo}.

We now introduce some related works in the literature.
\subsection{Summary of literature}The study of the distribution of the temperature within a body where the transfer of heat by means of radiation plays an important role has been extensively studied. Seminal works by Compton and Milne \cite{Compton,Milne} laid the foundation for understanding the interaction between radiation and gases. Subsequent papers by Holstein and Kenty provided further insights \cite{Holstein,PhysRev.42.823}. Specifically, Holstein highlighted the necessity to approach heat transfer by radiation as a non-local issue. The study of the evolution of temperature over time in a bar where the heat transfer is strictly due to radiation was considered by Spiegel \cite{spiegel1957smoothing}. Detailed reviews on the physics of radiative transfer can be found in works by Mihalas, Oxenius and Rutten\cite{oxenius, rutten1995radiative, mihalas2013foundations}.

In recent times the mathematical properties of the radiative transfer equation have been examined in \cite{Golse-Bardos,B1,B2,B3}. In several of these papers the authors studied the well-posedness of the time-dependent problem, usually using semi-group theory or the theory of $m$-accretive operators.

Another question that has been considered by several authors is the so-called Milne problem (cf. \cite{clouet2009milne,golse1987milne,MR0906921}). The Milne problem consists of describing the distribution of temperature in half space, a question which is motivated by the study of the distribution of the temperature near the boundary in the diffusion approximation limit. In this setting, the equation reduces to a one-dimensional problem.

Equations describing the distribution of temperature for bodies where the heat is transmitted by means of radiation and conduction have been considered by numerous authors with different boundary conditions, see for instance \cite{Nouri2,MR2478911,MR1866555,golse2008radiative,MR1234995,mercier1987application,MR1608072,MR2076780,MR3797032,MR3412332,MR1471600,ghattassi2023diffusive1, ghattassi2023diffusive2,ghattassi2023stability,ghattassi2022diffusive}. We refer to \cite{jang2023temperature} for more details. Moreover, other papers such as \cite{Nouri1,Dematte2023,RSM,2109.10071} consider the radiative transfer equation coupled to the Boltzmann equation.

Equations similar to \eqref{General rad eq} with the absorption-emission coefficient $\alpha^a_\nu=0$, focusing solely on scattering, are widely examined in mathematical studies, especially about neutron diffusion as seen in references \cite{MR533346,MR3324149,MR3686005}. Similar equations appear also in the theory of Lorentz gases, cf. \cite{MR3356368,MR3324151,BBS,G,LS,Spohn1978}. 

A recent paper by Arkeryd and Nouri \cite{NouriArkeryd} that considers the existence of mild solutions to normal discrete velocity Boltzmann equations with given incoming boundary values also requires a compactness theorem for line integrals having some analogies with the one derived in this paper. 

We want to emphasize that although the time-dependent problem has been considered in various papers, the existence of a time-dependent solution, even a globally bounded one, does not imply the existence of a solution of the stationary problem.

\subsection{Main theorems} 
In this paper, we consider the boundary value problem given by the system of equations \eqref{General rad eq}, \eqref{heat equation} and \eqref{incoming boundary}. We will consider two types of absorption coefficients and scattering coefficients. In the first case, the coefficients satisfy the so-called Grey approximation where $\alpha_\nu^a(T)=\alpha^a(T)$ and $\alpha_\nu^s(T)=\alpha^s(T)$ are independent of the frequency $\nu$. We will also consider a particular choice of $\alpha$'s in the non-Grey case, namely where $\alpha_\nu^a$ and $\alpha_\nu^s$ can be written as the product of functions in $\nu$ and $T$, separately. We denote this case from now on as ``pseudo Grey". A similar choice can be found in \cite{golse1987milne}. First, we study the case of pure emission and absorption where $\alpha^s_\nu=0$ and we will show the existence of a solution to this problem as stated in the following theorem. In the following theorem and throughout the rest of the paper, we denote $I_\nu \in L^\infty(\Omega,L^\infty(\mathbb{S}^2,L^1(\mathbb{R}_+)))$ when 
$$\sup_{x\in\Omega}\sup_{n\in\mathbb{S}^2}\int_0^\infty I_\nu(x,n)d\nu<\infty.$$
\begin{theorem}\label{main theorem pseudo grey}
Let $\Omega\subset \rth$ be bounded and open with $C^1$-boundary and strictly positive curvature. Suppose that the incoming boundary profile $g_\nu$ satisfies the bound $$\sup_{n\in\stw}\int_0^\infty d\nu\ g_\nu(n)<\infty,$$ and that $\alpha_\nu^a(T(x))=Q(\nu)\alpha(T(x))$ is bounded, strictly positive and $C^1$ in $T$, where $Q:\mathbb{R}_+\to\mathbb{R}_+$ and $\alpha:\mathbb{R}_+\to\mathbb{R}_+$. 
Then there exists a solution $ (T,I_\nu) \in L^\infty(\Omega)\times L^\infty\left(\Omega, L^\infty\left(\stw,L^1(\mathbb{R}_+)\right)\right)$,  which solves the boundary value problem given by \eqref{General rad eq}, \eqref{heat equation} and \eqref{incoming boundary} for $\alpha^s_\nu=0$, namely
\begin{equation}\label{rtenoscattering}
    \begin{cases}
        &n\cdot \nabla_x I_\nu(x,n) =Q(\nu)\alpha(T(x))\left( B_\nu(T(x) -I_\nu(x,n)\right),  \text{ for }x\in\Omega,\ n\in\stw,\\
        &\nabla_x\cdot \mathcal{F}(x)=0,\text{ at any }x\in\Omega,\\
        &I_\nu(x,n)= g_\nu(n)\ge 0,\ \text{ for } x\in\partial\Omega\text{ and }n\cdot n_x<0.
    \end{cases}
\end{equation}
Here, $I_\nu$ is a solution to \eqref{General rad eq} in the sense of distribution.
\end{theorem}We will prove Theorem \ref{main theorem pseudo grey} using a fixed-point argument. As we will see, the main difficulty that arises in our proof is to show the compactness of the terms involving the exponential function of a line integral. As first step we will regularize the line integral in order to obtain a problem where it is possible to prove existence using the Schauder fixed-point theorem. We will then show the compactness of the solutions of such regularized problems uniformly in the regularizing parameter. To this end, we provide the following type of general $L^2$ compactness result for sequences of non-linear operators of line integrals based on the study of auxiliary measures on $\stw$.
\begin{proposition}[Compactness result for line integrals]\label{compact}
    Let $\Pi^3=[-L,L]^3$ and $ (\vphi_j)_{j\in\mathbb{N}}\in 
    L^\infty\left(\Pi^3\right) $
     be a sequence of periodic functions with $ \Arrowvert \vphi_j\Arrowvert_\infty\leq M $. For $ n\in\stw $ and $m\in\mathbb{N}$ we define the operators $ L_n $ and $T_m$ by
  \begin{equation*}
  L_n[\vphi](x)\eqdef \int_{-L}^{L}d\lambda\ \vphi(x-\lambda n) \;\;\;\text{and}\;\;\;T_m[\vphi](x)\eqdef \int_\stw dn \left(L_n[\vphi](x)\right)^m.
  \end{equation*}
  Then for every $ m\in\mathbb{N} $ the sequence $ \left(T_m[\vphi_j]\right)_j $ is compact in $ L^2\left(\Pi^3\right) $. More precisely, the sequence $T_m[\vphi_j]$ satisfies the following equi-integrability condition: For any $\eps>0$ there exists some $h_0>0$ such that \begin{multline}\label{actual.equiintegrability.torus}
    \int_{\Pi^3}dx\;\left|T_m[\vphi_j](x)-T_m[\vphi_j](x+h)\right|^2 \\\leq C_m \int_{\Pi^3}dx \int_\stw dn\;\left| L_n[\vphi_j](x)-L_n[\vphi_j](x+h)\right|^2<\eps 
  \end{multline}
  for all $j\in\mathbb{N}$ and all $|h|<h_0$. The constant $C_m>0$ depends only on $m\in\mathbb{N}$, $M$ and $L$.
\end{proposition}The proposition above will provide the compactness theory required to conclude the proof of the existence of solutions to the original problem \eqref{rtenoscattering}. 
Finally, we study the existence of solutions for the full equation with both scattering and absorption-emission. In this case, we obtain the following existence theorem via the construction of suitable Green functions associated with the system.
\begin{theorem}[Full equations in the pseudo Grey case]\label{thm.full.pseudo.grey}
Let $\Omega\subset \rth$ be bounded and open with $C^1$-boundary and strictly positive curvature. Let 
$ \alpha_\nu^a(T(x))=Q_a(\nu)\alpha^a(T(x)) $ and $ \alpha_\nu^s(T(x))=Q_s(\nu)\alpha^s(T(x)) $ be bounded and strictly positive. Assume that $ Q_\ell\in C^1\left(\mathbb{R}_+\right) $ and $ \alpha^\ell\in C^1\left(\mathbb{R}_+\right)$ for $ \ell=a,s $. Assume $ K\in C^1\left(\stw\times\stw\right) $ be non-negative, rotationally symmetric, and independent of the frequency with the property \eqref{scattering assume}. 
Then there exists a solution $ (T,I_\nu) \in L^\infty(\Omega)\times L^\infty\left(\Omega, L^\infty\left(\stw,L^1(\mathbb{R}_+)\right)\right)$ to the equation \eqref{General rad eq} coupled with \eqref{heat equation} satisfying the boundary condition \eqref{incoming boundary}, where the $I_\nu$ is a solution to \eqref{General rad eq} in the sense of distribution.
\end{theorem}
\subsection{Strategy of the proof and main estimates}
 In the case of pure absorption and emission, our main strategy for the proof of the existence to the boundary-value problem is to reduce the stationary radiation equation and the ``divergence-free-radiation-flow" equation to a non-local non-linear elliptic equation for an explicit function $u$ of the temperature in the presence of an external source. This will allow to reformulate the problem \eqref{General rad eq} (with $\alpha^s=0$), \eqref{heat equation} and \eqref{incoming boundary} as
 \begin{equation}\label{non-local equation}
     u(x)-\int_0^\infty \int_\Omega \;\frac{F\left(\nu,u(\eta)\right)}{|x-\eta|^2}\exp{\left(-\int_{[x,\eta]}\alpha^a_\nu(T(\xi))ds(\xi)\right)}d\eta \:d\nu=S(x),
 \end{equation}
 where $[x,\eta]$ indicates the segment from $x\in\Omega$ to $\eta\in\Omega$ and the exact form of $u$, $F$ and $S$ are given in Section \ref{Sec.preliminary}. To prove the existence of a solution to the system given by equations \eqref{General rad eq}, \eqref{heat equation} and \eqref{incoming boundary} is equivalent to showing the existence of a solution to the non-local integral equation \eqref{non-local equation}. When the emission-absorption coefficient $\alpha_\nu^a$ depends on the local temperature and has the form $\alpha_\nu(x)=Q(\nu)\alpha(T(x))$, where $Q(\nu)$ is a function of the frequency which can be also constant, the strategy we use to prove the existence of a solution is the following. We consider first a regularized version of \eqref{non-local equation}, for which the existence of a solution is guaranteed by the Schauder fixed-point theorem. With the compactness result of Proposition \ref{compact}, which is based on the study of some auxiliary measures defined on the sphere $\stw$, it turns out that the sequence of regularized solutions is compact in $L^2$ and hence a subsequence converges pointwise almost everywhere to the solution of the original problem. We remark that obtaining $L^\infty$-estimates for the function $u$ solving \eqref{non-local equation} (or a regularized version of it) is not difficult using the structure of the integral operator. However, the main difficulty remains getting compactness. 
In the case of the problem including the scattering term, we use a similar strategy that however becomes more involved. 
To find a reformulation of \eqref{General rad eq}, \eqref{heat equation} and \eqref{incoming boundary} analogous to \eqref{non-local equation} we construct suitable fundamental solutions for a problem that includes absorption and scattering. These fundamental solutions satisfy recursive equations that allow us to obtain information about their properties using Duhamel series which contain terms involving exponentials of some integrals along straight lines as in \eqref{non-local equation}. Due to this we will regularize again the problem for which then solutions exist applying the Schauder fixed-point theorem. With the previous compactness result in Proposition \ref{compact} applied to finitely many terms of the Duhamel expansion we will obtain the compactness of the sequence of regularized solution in $L^2(\Omega)$ and thus the convergence pointwise almost everywhere to the solution of the full problem.

 \subsection{Outline of the paper}The rest of the paper is organized as follows. In Section \ref{sec.deriv.reg}, we provide the derivation of the non-local integral equation and the regularization of the equation which will be crucially used in the proof of the existence of solutions (Theorem \ref{main theorem pseudo grey}). Section \ref{Sec.emission.absorption} is devoted to the study of the existence of a solution to \eqref{General rad eq} in the absence of scattering. In Subsection \ref{sec.reg.exist} we prove the existence of solutions to the regularized problem. In Subsection \ref{sec.compactness.defect}, we provide a $L^2$ compactness theory of non-linear operators of line integrals based on the study of some auxiliary measures defined on the sphere $\stw$. This compactness theory will be used to obtain the compactness of the solution sequences of the regularized problem and this allows us to show the existence of the original problem stated in Theorem \ref{main theorem pseudo grey} in Subsection \ref{sec.grey proof}.  
In Section \ref{sec.full} we show the existence of solutions to the full equation \eqref{General rad eq} taking into account also the scattering term, in particular we will prove Theorem \ref{thm.full.pseudo.grey}. This will be made in several steps starting from the study of the Grey case deriving a non-local equation for the temperature (Subsection \ref{Sec.full.grey}) and constructing suitable Green functions which encode the effect of the scattering (Subsection \ref{Sec.full.green}). Subsections \ref{Sec.full.reg} to \ref{Sec.full.cmpt} are devoted to the proof of  existence of solutions to the equation \eqref{General rad eq} in the Grey case. There, a regularized problem is solved by means of Schauder's fixed-point theorem and a weak maximum principle, while the compactness result of Subsection \ref{sec.compactness.defect} is used to conclude the existence of a solution for the Grey approximation. Finally, in Subsection \ref{Sec.full.pseudo} we provide the proof of Theorem \ref{thm.full.pseudo.grey}.

\section{Derivation of a non-local integral equation for the temperature and the regularization of the equation}\label{sec.deriv.reg}
\subsection{Derivation of a non-local integral equation}\label{Sec.preliminary} In this section, we will first derive a non-local integral equation that is satisfied by the temperature. This equation is associated to the stationary equation
\begin{equation}\label{eq. absorption only}
    n\cdot \nabla_x I_\nu =\alpha_\nu^a (T)(B_\nu(T)-I_\nu).
\end{equation}
   Without loss of generality, we can assume $\sigma=1$ by rescaling variables. We define for every $(x,n)\in\Omega\times\stw$ a new coordinate system with variables $(y,s)=(y(x,n),s(x,n))\in\partial\Omega\times \mathbb{R}_{\geq0} $. These variables are defined in the following way. We consider for every $ x\in\Omega$ and $ n\in\stw$ the backward trajectory starting from $x$ in direction $-n$. Then $y(x,n)\in\partial\Omega$ is the boundary point that intersects with this straight line and $s(x,n)$ is its length, i.e. $s(x,n)=|x-y(x,n)|$ and $x=y+sn.$
   Therefore, using this notation, solving by characteristics $I_\nu$ and integrating \eqref{eq. absorption only}, we obtain that the flow
   $ \mathcal{F}=\int_0^\infty\int_{\mathbb{S}^2} nI_\nu(x,n)dnd\nu$ satisfies \begin{multline}\label{eq.flow.wellposed}
       \mathcal{F}
       = \int_0^\infty d\nu\int_{\mathbb{S}^2}dn \ ng_\nu(n)\exp\left(-\int_0^{s(x,n)}\alpha_\nu(T(y(x,n)+\zeta n))d\zeta\right) 
      \\ +\left[\int_0^\infty d\nu \int_{\mathbb{S}^2}dn \ \int_0^{s(x,n)}d\xi \exp\left(-\int^{s(x,n)}_{\xi}\alpha_\nu(T(y(x,n)+\zeta n))d\zeta)\right) n\right.\\\times \alpha_\nu\left(T(y(x,n)+\xi n)\right)B_\nu\left(T(y(x,n)+\xi n)\right)\bigg]
       =: \mathcal{F}_1+\mathcal{F}_2.
   \end{multline}
   Now we recall the \textit{conservation of energy} \eqref{heat equation} that yields $\nabla_x\cdot \mathcal{F}=0.$ In order to use this condition, we take the divergence of \eqref{eq.flow.wellposed}. 
   
   We first compute $\nabla_x\cdot \mathcal{F}_2$. We define new variables $\hat{\xi}\eqdef s-\xi$ and  $\hat{\zeta}\eqdef s-\zeta$ and make a change of variables $\xi\mapsto\hat{\xi}$in the integral, then
   \begin{multline*}
     \mathcal{F}_2=  \int_0^\infty d\nu\int_{\mathbb{S}^2}dn \ \int_0^{s(x,n)}d\xi \exp\left(-\int^{s(x,n)}_{\xi}\alpha_\nu(T(y(x,n)+\zeta n))d\zeta)\right) n\alpha_\nu(T)B_\nu(T)\\
    = \int_0^\infty d\nu\int_{\mathbb{S}^2}dn \ \int_0^{s(x,n)}d\hat{\xi}\exp\left(-\int^{\hat{\xi}}_{0}\alpha_\nu(T(x-\hat{\zeta}n))d\hat{\zeta}\right) n\alpha_\nu(T(x-\hat{\xi}n))B_\nu\left(T(x-\hat{\xi}n)\right)\\
    =\int_\Omega d\eta \int_0^\infty d\nu \exp\left(-\int^{|x-\eta|}_{0}\alpha_\nu\bigg(T\left(x-\hat{\zeta}\frac{x-\eta}{|x-\eta|}\right)\bigg)d\hat{\zeta}\right) \frac{x-\eta}{|x-\eta|}\alpha_\nu(T(\eta ))\frac{B_\nu(T(\eta ))}{|x-\eta|^2},
   \end{multline*}since the Jacobian gives $\frac{\partial (\hat{\xi},n)}{\partial \eta}=\frac{1}{|x-\eta|^2}$
   where $\eta\eqdef x-\hat{\xi}n$ and $\hat{\xi}=(x-\eta)\cdot n = |x-\eta|$. Also note that $n=\frac{x-\eta}{|x-\eta|}.$
   Therefore, we have
   \begin{equation}\label{nablaF2 ori}\nabla_x\cdot \mathcal{F}_2 =\int_0^\infty d\nu\int_\Omega d\eta \ \alpha_\nu(T(\eta ))B_\nu(T(\eta )) \nabla_x\cdot (\varphi v),\end{equation} where we define 
   \begin{equation*}
       \varphi(x,\eta)\eqdef \exp\left(-\int^{|x-\eta|}_{0}\alpha_\nu\bigg(T\left(x-\hat{\zeta}\frac{x-\eta}{|x-\eta|}\right)\bigg)d\hat{\zeta}\right) \text{ and }v\eqdef \frac{x-\eta}{|x-\eta|^3}.
   \end{equation*}
  We now use that 
   $$\nabla_x\cdot (\varphi v)=\nabla\varphi\cdot v +\varphi\nabla\cdot v,$$ where $\Div(v)=4\pi\delta(x-\eta),$ and
   \begin{multline*}
       \nabla \varphi = -\varphi \alpha_\nu(T(\eta)) \frac{x-\eta}{|x-\eta|}\\-\varphi\int^{|x-\eta|}_{0}\frac{d\alpha_\nu}{dT}\left(T\left(x-\hat{\zeta}\frac{x-\eta}{|x-\eta|}\right)\right)\left(\nabla T\left(x-\hat{\zeta}\frac{x-\eta}{|x-\eta|}\right)\right)\cdot D_x\left(x-\hat{\zeta}\frac{x-\eta}{|x-\eta|} \right)d\hat{\zeta}.
   \end{multline*}Note that
   \begin{multline}\label{Dx}
 \bigg(   \frac{x-\eta}{|x-\eta|}\cdot  D_x\bigg(\frac{x-\eta}{|x-\eta|}\bigg)\bigg)_l=\sum_{j=1}^3 \frac{(x-\eta)_j}{|x-\eta|}\bigg(\frac{\delta_{jl}}{|x-\eta|}-\frac{(x_l-\eta_l)}{|x-\eta|^3}(x_j-\eta_j)\bigg)\\= \frac{(x-\eta)_l}{|x-\eta|}\frac{1}{|x-\eta|}-\frac{|x-\eta|(x_l-\eta_l)}{|x-\eta|^3}=0.
 \end{multline}
   Using $\varphi(\eta,\eta)=1,$ we have
   \begin{multline*}
       \nabla_x\cdot (\varphi v) = 4\pi \delta(x-\eta)-\varphi \alpha_\nu(T(\eta)) \frac{1}{|x-\eta|^2}\\-\frac{\varphi}{|x-\eta|^2}\int^{|x-\eta|}_{0}\frac{d\alpha_\nu}{dT}(T\left(x-\hat{\zeta}\frac{x-\eta}{|x-\eta|}\right))\frac{x-\eta}{|x-\eta|}\cdot \nabla T\left(x-\hat{\zeta}\frac{x-\eta}{|x-\eta|}\right)d\hat{\zeta}\\
       =4\pi \delta(x-\eta)-\varphi \alpha_\nu(T(\eta)) \frac{1}{|x-\eta|^2}\\+\frac{\varphi}{|x-\eta|^2}\int^{|x-\eta|}_{0}\frac{d\alpha}{dT}(T\left(x-\hat{\zeta}\frac{x-\eta}{|x-\eta|}\right))\frac{d}{d\hat{\zeta}} (T\left(x-\hat{\zeta}\frac{x-\eta}{|x-\eta|}\right))d\hat{\zeta}\\
       =4\pi \delta(x-\eta)-\varphi \alpha_\nu(T(\eta)) \frac{1}{|x-\eta|^2}+\frac{\varphi}{|x-\eta|^2}\left(\alpha_\nu(T(\eta)-\alpha_\nu(T(x))\right)\\
       =4\pi \delta(x-\eta)-\frac{\varphi}{|x-\eta|^2}\alpha_\nu(T(x)).
   \end{multline*}
   Therefore, by \eqref{nablaF2 ori} we have
   \begin{multline}\label{nablaF2}
       \nabla_x\cdot \mathcal{F}_2 =4\pi\int_0^\infty \alpha_\nu(T(x))B_\nu (T(x)) d\nu\\-\int_0^\infty d\nu \alpha_\nu(T(x))\int_\Omega d\eta \ \alpha_\nu(T(\eta ))B_\nu(T(\eta )) \frac{\exp\left(-\int^{|x-\eta|}_{0}\alpha_\nu(T\left(x-\hat{\zeta}\frac{x-\eta}{|x-\eta|}\right))d\hat{\zeta}\right)}{|x-\eta|^2}.
   \end{multline} We will see that the integral operator in \eqref{nablaF2} is a contractive operator.
   
   Now we compute $\nabla_x\cdot \mathcal{F}_1.$ Using \eqref{eq.flow.wellposed}, we have
   \begin{multline*}
      \nabla_x\cdot \mathcal{F}_1=  \int_0^\infty \int_{\mathbb{S}^2}dn \ \exp\left(-\int_0^{s(x,n)}\alpha_\nu(T(y(x,n)+\zeta n))d\zeta\right)
     \\\times n\cdot\bigg[
    - g_\nu(n)\nabla_xs\alpha_\nu(T(x))-g_\nu(n)\int_0^{s(x,n)}\nabla_x((\alpha_\nu \circ T)(y(x,n)+\zeta n))d\zeta\bigg].
   \end{multline*}
Here we observe that 
\begin{multline*}n\cdot \nabla_x((\alpha_\nu \circ T)(y(x,n)+\zeta n))=\frac{d}{dt}((\alpha_\nu \circ T)(y(x+tn,n)+\zeta n))|_{t=0}\\=\frac{d}{dt}((\alpha_\nu \circ T)(y(x,n)+\zeta n))|_{t=0}=0,\end{multline*} since $y(x+tn,n)=y(x,n).$
Also note that 
that $n\cdot\nabla_x s=1$. This holds by the following observation. For any $x_0\in \Omega,$ we have  
$$y(x_0+\zeta n,n)+s(x_0+\zeta n,n)n = x_0+\zeta n, $$ and hence $$y(x_0,n)+s(x_0+\zeta n,n)n = x_0+\zeta n.$$ We can differentiate it with respect to $\zeta$ and we obtain 
$$\frac{d}{d\zeta}(s(x_0+\zeta n,n)n)=(\nabla_x s(x_0+\zeta n,n)\cdot n)n=n.$$ Thus we have
\begin{equation}\label{nablaF1 final}
      \nabla_x\cdot \mathcal{F}_1=  -\int_0^\infty d\nu \int_{\mathbb{S}^2}dn \ \exp\left(-\int_0^{s(x,n)}\alpha_\nu(T(y(x,n)+\zeta n))d\zeta\right)\alpha_\nu(T(x))g_\nu(n),
\end{equation}and note that $\nabla_x \cdot \mathcal{F}_1 \le 0$ and $|\nabla_x \cdot \mathcal{F}_1|$ is bounded from above in $L^\infty$, since $\alpha_\nu(\cdot)$ is bounded and $G=\int_0^\infty d\nu\ g_\nu(n)\in L^\infty(\mathbb{S}^2)$.

Combining \eqref{nablaF2} and \eqref{nablaF1 final} we finally obtain
    \begin{multline}\label{nablaF}
       \nabla_x\cdot \mathcal{F} =4\pi\int_0^\infty \alpha_\nu(T(x))B_\nu(T(x))d\nu \\-\int_0^\infty d\nu\alpha_\nu(T(x))\int_\Omega d\eta \ \alpha_\nu(T(\eta ))B_\nu(T(\eta )) \frac{\exp\left(-\int^{|x-\eta|}_{0}\alpha_\nu(T\left(x-\hat{\zeta}\frac{x-\eta}{|x-\eta|}\right))d\hat{\zeta}\right)}{|x-\eta|^2}\\-\int_0^\infty d\nu\int_{\mathbb{S}^2}dn \ \exp\left(-\int_0^{s(x,n)}\alpha_\nu(T(y(x,n)+\zeta n))d\zeta\right)\alpha_\nu(T(x))g_\nu(n) = 0.
   \end{multline} 
   In the pseudo Grey case as in Theorem \ref{main theorem pseudo grey} the absorption coefficient has the form $\alpha_\nu(T(x))=Q(\nu)\alpha(T(x))$ and it is strictly positive and bounded. Hence dividing by $\alpha(T(x)) $ equation \eqref{nablaF} reads
   \begin{multline}\label{nablaFpseudogrey}
      4\pi\int_0^\infty Q(\nu)B_\nu(T(x))d\nu \\=\int_0^\infty d\nu\ Q^2(\nu)\int_\Omega d\eta \ \alpha(T(\eta ))B_\nu(T(\eta )) \frac{\exp\left(-Q(\nu)\int^{|x-\eta|}_{0}\alpha(T\left(x-\hat{\zeta}\frac{x-\eta}{|x-\eta|}\right))d\hat{\zeta}\right)}{|x-\eta|^2}\\+\int_0^\infty d\nu\int_{\mathbb{S}^2}dn \ \exp\left(-Q(\nu)\int_0^{s(x,n)}\alpha(T(y(x,n)+\zeta n))d\zeta\right)Q(\nu)g_\nu(n).
   \end{multline} 
   Assuming now the Grey approximation, i.e. assuming that the absorption coefficient is strictly positive and independent of $\nu$ (i.e. $\alpha_\nu(T(x))=\alpha(T(x))$), and using the Stefan Law \eqref{integration of blackbody} we obtain dividing \eqref{nablaF} by $ \alpha(T(x))$ and defining $G(n)=\int_0^\infty g_\nu(n)d\nu$,\begin{multline}\label{nablaFreduced}
   4\pi(T(x))^4 =\int_\Omega d\eta \ \alpha(T(\eta ))(T(\eta ))^4 \frac{\exp\left(-\int^{|x-\eta|}_{0}\alpha(T\left(x-\hat{\zeta}\frac{x-\eta}{|x-\eta|}\right))d\hat{\zeta}\right)}{|x-\eta|^2}\\+\int_{\mathbb{S}^2}dn \ \exp\left(-\int_0^{s(x,n)}\alpha(T(y(x,n)+\zeta n))d\zeta\right)G(n) \ge 0.
   \end{multline}

   \subsection{Non-local integral equation in the Grey case}
   We will focus next on the Grey approximation and we will prove the following theorem.
\begin{theorem}\label{Grey thm}
Let $\Omega\subset \rth$ be bounded and open with $C^1$-boundary and strictly positive curvature.
Suppose that the incoming boundary profile satisfies $\|G\|_{L^\infty(\mathbb{S}^2)}<\infty$, where $G(n)=\int_0^\infty g_\nu(n)d\nu$. In addition, suppose that the absorption coefficient $\alpha(\cdot)$ is bounded and strictly positive and assume $\alpha_s=0$. Then there exists a solution $(T,I_\nu)\in L^\infty(\Omega)\times L^\infty\left(\Omega,L^\infty\left(\stw,L^1\left(\mathbb{R}_+\right)\right)\right)$ which solves the boundary-value problem \eqref{General rad eq}-\eqref{incoming boundary} coupled with the \textit{conservation of energy} \eqref{heat equation}. $I_\nu$ is a solution to \eqref{General rad eq} in the sense of distribution.
\end{theorem}
In order to prove Theorem \ref{Grey thm} we aim to use a fixed-point argument. To this end we first see that $u$ satisfies an $L^\infty$-estimate. Indeed, observe that $\Omega$ is bounded in $\mathbb{R}^3$ and
   \begin{multline}\label{u upper bnd}
       \int_\Omega d\eta \ \alpha(T(\eta ))(T(\eta ))^4 \frac{\exp\left(-\int^{|x-\eta|}_{0}\alpha(T\left(x-\hat{\zeta}\frac{x-\eta}{|x-\eta|}\right))d\hat{\zeta}\right)}{|x-\eta|^2}\\
       =\int_{\mathbb{S}^2}dn \int_0^{s(x,n)}r^2 dr  \ \alpha(T(\eta ))(T(\eta ))^4 \frac{\exp\left(-\int^{r}_{0}\alpha(T(x-\hat{\zeta}n))d\hat{\zeta}\right)}{r^2}\\
       = \int_{\mathbb{S}^2}dn \int_0^{s(x,n)}  dr \ (T(\eta ))^4 \left(-\frac{d}{dr}\exp\left(-\int^{r}_{0}\alpha(T(x-\hat{\zeta}n))d\hat{\zeta}\right)\right)\\
       \le \|T\|_{L^\infty(\Omega)}^4 \int_{\mathbb{S}^2}dn\left(1-\exp\left(-\int^{s(x,n)} _{0}\alpha(T(x-\hat{\zeta}n))d\hat{\zeta}\right)\right)\\
       \le 4\pi(1-\delta)\|T\|_{L^\infty(\Omega)}^4,
   \end{multline}where $$\delta =  \exp\left(-\|\alpha\|_{L^\infty}\max_{x\in\Omega, \ n\in\mathbb{S}^2}s(x,n)\right)>0.$$
   In addition, we observe
  \begin{multline}\label{G upper bound}
       \left|\int_{\mathbb{S}^2}dn \ \exp\left(-\int_0^{s(x,n)}\alpha(T(y(x,n)+\zeta n))d\zeta\right)G(n)\right|\\\le \int_{\mathbb{S}^2}G(n)dn=\|G\|_{L^1(\mathbb{S}^2)}\leq 4\pi \|G\|_{L^\infty(\mathbb{S}^2)}<\infty.
  \end{multline}
  
  Let us now define $ u(x)=4\pi\sigma T^4(x) $. Hence, we write $ \gamma(u(x))=\alpha\left(\sqrt[4]{\frac{u(x)}{4\pi\sigma}}\right) $. In order to simplify the notation we also denote by $$\int_{[x,\eta]} f(\xi)d\xi=\int_0^{|x-\eta|}f\left(x+t\frac{\eta-x}{|\eta-x|}\right)\;dt .$$ Then we obtain
  \begin{multline}\label{nablaFreducedu}
  u(x) =\int_\Omega d\eta \ \gamma(u(\eta))u(\eta ) \frac{\exp\left(-\int_{[x,\eta]}\gamma(u\left(\zeta\right))d\zeta\right)}{4\pi|x-\eta|^2}\\+\int_{\mathbb{S}^2}dn \ \exp\left(-\int_{[x,y(x,n)]}\gamma(u(\zeta))d\zeta\right)G(n).
  \end{multline} This completes the derivation of the non-local integral equation. In the next Subsection, we consider the regularization of the line integral in the non-local equation. 
 \subsection{Regularization of the non-local equation}\label{sec.regular}
  In order to prove the existence of a function $ u $ solving the non-local equation \eqref{nablaFreducedu}, we will consider a regularization of the line integral. For the the regularized problem we will apply Schauder's fixed-point theorem and show the existence of a solution. We obtain in this way a sequence of solutions to the regularized problem. We will hence show that the sequence of integral operators acting on that regularized solutions is compact in $ L^2 $. This implies the existence of a subsequence convergent pointwise almost everywhere to a function $ u $. After an application of the dominated convergence theorem we will show that this limit function is a solution to the original problem \eqref{nablaFreducedu}.
  
  Let $ \phi_\eps\in C_c^\infty(\rth) $ be a standard positive and radially symmetric mollifier. Given a segment $ \Gamma $ we define $ \int_\rth \delta_\Gamma(y)\varphi(y)dy=\int_\Gamma \varphi(\xi)d\xi$. Hence, for $ x,\eta\in\Omega $
  \begin{multline}\label{convolution}
  \int_\rth F(\xi)\delta_{[x,\eta]}*\phi_\eps (\xi)d\xi=\int_\rth F(\xi)\int_0^{|\eta-x|}\phi_\eps\left(\xi-x-\lambda \frac{\eta-x}{|\eta-x|}\right)d\lambda d\xi\\
  =\int_\rth \int_0^{|\eta-x|}F\left(\xi+x+\lambda \frac{\eta-x}{|\eta-x|}\right)\phi_\eps (\xi)d\lambda d\xi\\\overset{\eps\to 0}{\longrightarrow}\int_0^{|\eta-x|}F\left(x+\lambda \frac{\eta-x}{|\eta-x|}\right)d\lambda.
  \end{multline}
  In order to have also the same type of $ L^\infty $ estimate we consider
  \begin{multline}\label{regular}
  u(x)=\mathcal{B}^\eps(u)(x) \eqdef \int_\Omega d\eta \ \left(\gamma(u)*\phi_\eps\right)(\eta)u(\eta ) \frac{\exp\left(-\int_\rth\gamma(u\left(\xi\right))\delta_{[x,\eta]}*\phi_\eps (\xi)d\xi\right)}{4\pi|x-\eta|^2}\\+\int_{\mathbb{S}^2}dn \ \exp\left(-\int_\rth\gamma(u(\xi))\delta_{[x,y(x,n)]}*\phi_\eps (\xi)d\xi\right)G(n).
  \end{multline}
  We remark that by the smoothness of $\gamma$ and the continuity of the exponential function the integral operator $\mathcal{B}^\eps$ is continuous.
The interesting part of this regularization is that we can get the same type of $L^\infty$-estimate as for the original problem. Indeed, using the symmetry of $ \phi_\eps $ and again the change of variables $ \eta=x-rn $ we see that
  \begin{multline}\label{derivative}
  \frac{d}{dr}\exp\left(-\int_\rth \gamma(u(\xi))\int_0^r\phi_\eps(\xi-x+\lambda n)d\lambda d\xi\right)\\
  =-\exp\left(-\int_\rth \gamma(u(\xi))\int_0^r\phi_\eps(\xi-x+\lambda n)d\lambda d\xi\right)\int_\rth \phi_\eps(x-rn-\xi)\gamma(u(\xi))d\xi\\=-\exp\left(-\int_\rth \gamma(u(\xi))\int_0^r\phi_\eps(\xi-x+\lambda n)d\lambda d\xi\right)\left(\gamma(u)*\phi_\eps\right)(x-rn).
  \end{multline}
  We can then argue as in \eqref{u upper bnd}  and \eqref{G upper bound} that $\Arrowvert \mathcal{B}^\eps(u)\Arrowvert_\infty\leq (1-\delta) \Arrowvert u\Arrowvert_\infty+\Arrowvert G\Arrowvert_{L^1}$. We have hence obtained a suitable regularization of equation \eqref{nablaFreducedu}. In the next Subsection, we will prove the existence of a solution to the regularized problem.
  \section{Existence theory for the pure emission-absorption case}\label{Sec.emission.absorption}
  \subsection{Existence of solutions to the regularized problem in the Grey case}\label{sec.reg.exist}
  We are now ready to prove the existence of a solution to the regularized problem \eqref{regular} for the Grey approximation. We start with the $ L^\infty $-estimate and we proceed exactly as before. Hence, for $ D=\text{diam}\left(\Omega\right) $, passing to spherical coordinates and using \eqref{derivative} we obtain
  \begin{equation*}
  \Arrowvert \mathcal{B}^\eps (u)\Arrowvert_\infty\leq \Arrowvert u \Arrowvert_\infty \left(1-e^{-D \Arrowvert\gamma\Arrowvert_\infty}\right)+\Arrowvert G\Arrowvert_{L^1}.
  \end{equation*} 
  Thus, for $ K>\Arrowvert G\Arrowvert_{L^1}e^{D \Arrowvert\gamma\Arrowvert_\infty} $ we see that the operator $ \mathcal{B}^\eps$ maps continuously the set $$ \left\{u\in L^\infty(\Omega): u\geq 0,\; \Arrowvert u\Arrowvert_\infty\leq K\right\} $$ to itself.
  Actually, it is a compact operator mapping the non-negative continuous functions bounded by $ K $ to the H\"older continuous functions. This is relevant because it allows us to apply the Schauder fixed-point theorem (cf. \cite{evans}).
  
  To this end we assume now $ u\in C(\Omega) $ and $ u\geq 0 $. By definition, we can extend it continuously up to the boundary $\bnd$. Moreover, we extend by zero both functions $u$ and $\gamma(u)$ outside $\bar{\Omega}$ such that the convolution $\gamma(u)*\phi_\eps$ is smooth and well-defined. Let $x\in\Omega$ and $h\in\rth$ with $x+h\in\Omega$. We estimate
\begin{multline}\label{schauder1}
  \left|\mathcal{B}^\eps(u)(x)-\mathcal{B}^\eps(u)(x+h)\right|\\
  \leq \left|\int_\Omega d\eta \ \left(\gamma(u)*\phi_\eps\right)(\eta)u(\eta ) \frac{\exp\left(-\int_\rth\gamma(u\left(\xi\right))\delta_{[x,\eta]}*\phi_\eps (\xi)d\xi\right)}{4\pi|x-\eta|^2}\right.\\\left.\;\;\;\;\;\;\;\;\;\;\;\;\;\;\;-\int_\Omega d\eta \ \left(\gamma(u)*\phi_\eps\right)(\eta)u(\eta ) \frac{\exp\left(-\int_\rth\gamma(u\left(\xi\right))\delta_{[x+h,\eta]}*\phi_\eps (\xi)d\xi\right)}{4\pi|x+h-\eta|^2}\right|\\
  +\left|\int_{\mathbb{S}^2}dn \exp\left(-\int_\rth\gamma(u(\xi))\delta_{[x,y(x,n)]}*\phi_\eps (\xi)d\xi\right)G(n)\right.\\\left.-\int_{\mathbb{S}^2}dn \exp\left(-\int_\rth\gamma(u(\xi))\delta_{[x+h,y(x+h,n)]}*\phi_\eps (\xi)d\xi\right)G(n)\right|\\
  \leq\frac{1}{4\pi} \int_\Omega d\eta \ \left(\gamma(u)*\phi_\eps\right)(\eta)u(\eta ) \exp\left(-\int_\rth\gamma(u\left(\xi\right))\delta_{[x,\eta]}*\phi_\eps (\xi)d\xi\right)\\\;\;\;\;\;\;\;\;\;\;\;\;\;\;\;\;\;\;\;\;\;\;\;\;\;\;\;\;\;\;\times \left|\frac{1}{|x-\eta|^2}-\frac{1}{|x+h-\eta|^2}\right|\\
  +\int_\Omega d\eta \ \frac{\left(\gamma(u)*\phi_\eps\right)(\eta)u(\eta )}{4\pi|x+h-\eta|^2} \left|\exp\left(-\int_\rth\gamma(u\left(\xi\right))\delta_{[x,\eta]}*\phi_\eps (\xi)d\xi\right)\right.\\\left.\;\;\;\;\;\;\;\;\;\;\;\;\;\;\;\;\;\;\;\;\;\;-\exp\left(-\int_\rth\gamma(u\left(\xi\right))\delta_{[x+h,\eta]}*\phi_\eps (\xi)d\xi\right)\right|\\+\int_{\mathbb{S}^2}dn\ G(n)\left|\exp\left(-\int_\rth\gamma(u(\xi))\delta_{[x,y(x,n)]}*\phi_\eps (\xi)d\xi\right)\right.\\\left.\;\;\;\;\;\;\;\;\;\;\;\;\;\;\;\;\;\;\;\;\;\;-\exp\left(-\int_\rth\gamma(u(\xi))\delta_{[x+h,y(x+h,n)]}*\phi_\eps (\xi)d\xi\right)\right|\\=:I+II+III.
  \end{multline} 
  
  In order to estimate the integral term $ I $, we proceed splitting it in two integrals
\begin{multline*}
  	I\leq C(\Arrowvert\gamma\Arrowvert_\infty,K)\int_\Omega d\eta \left|\frac{1}{|x-\eta|^2}-\frac{1}{|x+h-\eta|^2}\right|\\\leq C\int_ {\Omega\cap\{|x-\eta|\leq 2|h|\}} d\eta \left|\frac{1}{|x-\eta|^2}-\frac{1}{|x+h-\eta|^2}\right|\\+C\int_ {\Omega\cap\{|x-\eta|> 2|h|\}} d\eta \left|\frac{1}{|x-\eta|^2}-\frac{1}{|x+h-\eta|^2}\right|.
  \end{multline*}
  If $ |x-\eta|\leq 2|h| $ then also $ |x+h-\eta|\leq 3|h| $ and hence 
  \begin{equation*}
  \int_ {\Omega\cap\{|x-\eta|\leq 2|h|\}} d\eta \left|\frac{1}{|x-\eta|^2}-\frac{1}{|x+h-\eta|^2}\right|\leq 2\int_{B_{3|h|}(0)}dy\frac{1}{|y|^2}=24 \pi|h|. \end{equation*}
  On the other hand, if $ |x-\eta|>2|h| $, then for $ 0<s<1 $,
  \begin{equation*}
  \left|\frac{1}{|x-\eta|^2}-\frac{1}{|x+h-\eta|^2}\right|\leq \frac{|h|^2+2|h||x-\eta|}{|x-\eta|^2|x+h-\eta|^2}\leq 2^{s-2}|h|^s\frac{1}{|x-\eta|^s|x+h-\eta|^2}.
 \end{equation*}
 Choosing now $ s=\frac{1}{2} $ we see that $ \frac{1}{|x-\cdot|^s}\in L^4(\Omega) $ and $ \frac{1}{|x+h-\cdot|^2}\in L^{\frac{4}{3}}(\Omega) $. Hence 
 \begin{equation*}
 \int_ {\Omega\cap\{|x-\eta|> 2|h|\}} d\eta \left|\frac{1}{|x-\eta|^2}-\frac{1}{|x+h-\eta|^2}\right|\leq C(\Omega)|h|^{\frac{1}{2}}.
 \end{equation*}
  Summarizing we get
   for a sufficiently small $|h|<1,$\begin{equation}\label{I}
  I\leq C(\Omega, \Arrowvert\gamma\Arrowvert_\infty,K)|h|^{\frac{1}{2}}.
\end{equation}

  For the second term $ II $ we use the following three estimates which are the consequence of the smoothness of $ \phi_\eps $
  \begin{equation}\label{path1}
  \int_0^{|\eta-x|}d\lambda \left|\phi_\eps\left(z-x-\lambda\frac{\eta-x}{|\eta-x|}\right)-\phi_\eps\left(z-x-h-\lambda\frac{\eta-x}{|\eta-x|}\right)\right|\leq C(\phi_\eps)|h| |\eta-x|;
  \end{equation}
  \begin{multline}\label{path2}
  \int_0^{|\eta-x|}d\lambda \left|\phi_\eps\left(z-x-h-\lambda\frac{\eta-x}{|\eta-x|}\right)-\phi_\eps\left(z-x-h-\lambda\frac{\eta-x-h}{|\eta-x-h|}\right)\right|\\\leq C(\phi_\eps) \frac{|x-\eta|}{2}\frac{\big|(\eta-x\pm h)|\eta-x-h|-(\eta-x-h)|\eta-x|\big|}{|\eta-x-h|}\\\leq C(\phi_\eps)\frac{|x-\eta|}{2} \left(\big||\eta-x-h|-|x-h|\big|+|h|\right)\leq C(\phi_\eps)|h||\eta-x|;
  \end{multline}and
  \begin{equation}\label{path3}
  \left|\int_{|\eta-x|}^{|\eta-x-h|}d\lambda\ \phi_\eps\left(z-x-h-\lambda\frac{\eta-x-h}{|\eta-x-h|}\right)\right|\leq C(\phi_\eps)|h|.
  \end{equation}
  Now, using the well-known inequality $ |e^{-a}-e^{-b}|\leq |a-b| $ for $a,b\geq 0$ and the definition of the line integrals as in \eqref{convolution} we see
  \begin{multline}\label{II}
  II\leq C(\Arrowvert\gamma\Arrowvert_\infty,\phi_\eps,K)\int_\Omega d\eta\  \frac{1}{|\eta-x-h|^2}
  \int_\rth dz\ \gamma(u)(z)\\\times\left|\int_0^{|\eta-x|}d\lambda\ \phi_\eps\left(z-x-\lambda\frac{\eta-x}{|\eta-x|}\right)-\phi_\eps\left(z-x-h-\lambda\frac{\eta-x-h}{|\eta-x-h|}\right)\right|\\\leq C(\Arrowvert\gamma\Arrowvert_\infty,\phi_\eps,K,\Omega)\Arrowvert \gamma\Arrowvert_\infty |h|,
  \end{multline}
  where in the last step we used all three estimates \eqref{path1}, \eqref{path2} and \eqref{path3}.  
  
  The last integral term $ III $ is estimated in a similar way as we did for $ II $. Since we assumed that $\bnd$ is $C^1$ and has positive curvature, we notice that there exists a constant $ C(\Omega) $ depending on the curvature of the domain, such that if $ |h|<1 $ is sufficiently small then \begin{equation}\label{curvature}
       \left|s(x,n)-s(x+h,n)\right|\leq C(\Omega)|h|^{\frac{1}{2}},
  \end{equation} for all $ n\in \mathbb{S}^2 $. 
Estimate \eqref{curvature} is the result of a geometrical argument considering the worst case scenario when $n$ is close to tangent to the boundary at the point $x-s(x,n)n$ or $x+h-s(x+h,n)n$ taking into account that the curvature of $\bnd$ is strictly positive.
  Hence,
   \begin{multline*}
\left|\int_0^{s(x,n)}d\lambda\ \phi_\eps(z-x+\lambda n)-\int_0^{s(x+h,n)}d\lambda\ \phi_\eps(z-x-h+\lambda n)\right|\\
 \leq \int_0^{\min\left(s(x,n),s(x+h,n)\right)}d\lambda \left|\phi_\eps(z-x+\lambda n)-\phi_\eps(z-x-h+\lambda n)\right|\\
 + \int_{\min\left(s(x,n),s(x+h,n)\right)}^{\max\left(s(x,n),s(x+h,n)\right)}d\lambda \left|\phi_\eps\right|
 \leq C(\phi_\eps, \Omega) |h|^{\frac{1}{2}}.
  \end{multline*}
  Hence, we conclude
  \begin{equation}\label{III}
  III\leq C(\Arrowvert G\Arrowvert_\infty,\Omega,\phi_\eps)\Arrowvert \gamma \Arrowvert_\infty |h|^{\frac{1}{2}}.
  \end{equation}
  
  Estimates \eqref{I}, \eqref{II} and \eqref{III} together imply the estimate $$ \left|\mathcal{B}^\eps(u)(x)-\mathcal{B}^\eps(u)(x+h)\right|\leq C\left(G,\gamma, K,\phi_\eps,\Omega\right) |h|^{\frac{1}{2}}, $$ for all $ x\in\Omega $ and $ |h|<1 $ sufficiently small. We have just proved then that $ \mathcal{B}^\eps $ maps continuous functions to H\"older continuous functions. 
  It is therefore a compact operator. As we have already noticed it is also a continuous operator. Then Schauder's fixed-point theorem implies the existence of a fixed-point $ u_\eps\in C\left(\Omega\right) $ with $ 0\leq u_\eps\leq K $ such that $ u_\eps=\mathcal{B}^\eps(u_\eps) $. This concludes the proof of the existence of a solution $ u_\eps $ for the regularized problem \eqref{regular}.
  In the next section, we will provide a general $L^2$ compactness theory based on some auxiliary meausres defined on $\stw$ to prove the existence of the original problem.

  \subsection{Compactness theory for operators defined by means of some line integrals}\label{sec.compactness.defect}
  We prove now Proposition \ref{compact}. 
  \begin{proof}[Proof of Proposition \ref{compact}] 
  Without loss of generality we can assume $L=1$ and $M=1$. We start writing $ \vphi_j $ in its Fourier series form as $$ \vphi_j=\sum_{k\in\pi\mathbb{Z}^3}a_k^j e^{ik\cdot x} .$$ We denote by $ \mu_j $ the measure associated to $ \vphi_j $ and defined by $$ \mu_j=\sum_{k\in\pi\mathbb{Z}^3} \left|a_k^j\right|^2 \delta_{\frac{k}{|k|}}\in \mathcal{M}_+\left(\stw\right) .$$ We will work with the auxiliary measures defined on $\stw$ given for $ R>0 $ by $$ \mu_j^R=\sum_{\substack{k\in\pi\mathbb{Z}^3\\ |k|>R}} \left|a_k^j\right|^2 \delta_{\frac{k}{|k|}}\in \mathcal{M}_+\left(\stw\right) .$$ Moreover we see that $$ \mu_j^R(\stw)\leq \mu_j(\stw)=\Arrowvert\vphi_j\Arrowvert_{L^2}\leq 8 ,$$ where we used that $ \left|\Pi^3\right|=8 $.	
 	
 	We can now rewrite using the absolute convergence of the series and computing the integrals
 	\begin{multline*}
 	L_n[\vphi_j](x)= \int_{-1}^1 d\lambda \sum_{k\in\pi\mathbb{Z}^3}a_k^j e^{ik\cdot (x-\lambda n)}\\=\sum_{k\in\pi\mathbb{Z}^3}a_k^j e^{ik\cdot x}\int_{-1}^1 d\lambda e^{-ik\cdot \lambda n}=\sum_{k\in\pi\mathbb{Z}^3}2a_k^j\frac{\sin(k\cdot n)}{k\cdot n} e^{ik\cdot x}\\
 	=\sum_{\substack{k\in \pi\mathbb{Z}^3\\|k|\leq R}}2a_k^j\frac{\sin(k\cdot n)}{k\cdot n} e^{ik\cdot x}+\sum_{\substack{k\in \pi\mathbb{Z}^3\\|k|> R}}2a_k^j\frac{\sin(k\cdot n)}{k\cdot n} e^{ik\cdot x}.
 	\end{multline*}
 	Since the first sum is finite, the first term on the right hand side is compact for every fixed $R>0$. We now consider the contribution due to the second term. We define the auxiliary measure associated to $ L_n[\vphi_j] $ that will be denoted by $\nu_{n,j}^R$. More precisely we define it by means of
 	\begin{equation}\label{nu1}
 	\nu_{n,j}^R(\omega)\eqdef \sum_{\substack{k\in \pi\mathbb{Z}^3\\|k|> R}}4\left|a_k^j\right|^2\left|\frac{\sin(k\cdot n)}{k\cdot n}\right|^2 \delta_{\frac{k}{|k|}}(\omega).
 	\end{equation}
 	Again we see $  \nu_{n,j}^R\in \mathcal{M}_+\left(\stw\right)  $ with $ \nu_{n,j}^R\left(\stw\right)\leq 32 $. Notice that $ \nu_{n,j}^R\leq 4\mu_j^R $ for all $ n\in\stw $ as a measure. 
  We notice also that by definition 
 	$ \nu_{n,j}^R\left.\right|_{\{\omega\cdot n=0\}}=4\mu_j^R\left.\right|_{\{\omega\cdot n=0\}} $ since $ \lim\limits_{x\to 0}\frac{\sin(x)}{x}=1 $. Moreover, we can write
 	\begin{equation}\label{decomposition1} 
 	\nu_{n,j}^{R}=\nu_{n,j}^{R}\left.\right|_{\{0\leq|\omega\cdot n|<\kappa\}}+\sum_{\substack{k\in \pi\mathbb{Z}^3\\|k|> {R}}}4\left|a_k^{j}\right|^2\left|\frac{\sin(k\cdot n)}{k\cdot n}\right|^2 \delta_{\frac{k}{|k|}}(\omega)\left.\right|_{\{|\omega\cdot n|\geq\kappa\}}.
 	\end{equation}  
 	On one hand we have $$ \nu_{n,j}^{R}\left.\right|_{\{0\leq|\omega\cdot n|<\kappa\}}\leq 4\mu_{j}^{R}\left.\right|_{\{0<|\omega\cdot n|<\kappa\}} $$ and also defining $ f_\kappa(\omega,n)= \chi_{\{\omega:0\leq|\omega\cdot n|<\kappa\}}(\omega)$ we have $$f_\kappa(\omega,n)=\chi_{\{(\omega,n):0\leq|\omega\cdot n|<\kappa\}}(\omega,n)= \chi_{\{n:0\leq|\omega\cdot n|<\kappa\}}(n)$$  hence we compute 
 	\begin{multline}\label{measure11}
 	\int_\stw dn\int_\stw d\nu_{n,j}^{R}(\omega)\left.\right|_{\{0<|\omega\cdot n|<\kappa\}}\leq 4\int_\stw dn \int_\stw d\mu_{j}^{R}(\omega)\chi_{\{(\omega,n):0<|\omega\cdot n|<\kappa\}}(\omega,n)\\= 4\int_\stw d\mu_{j}^{R}(\omega)  \int_\stw dn \chi_{\{n:0\leq\left|\omega\cdot n\right|<\kappa\}}(n)\leq 128 \pi \kappa\to 0
 	\end{multline}
 	uniformly in $j\in\mathbb{N}$ and $ R\in \pi\mathbb{N} $. For the first inequality we used that $ \nu_{n,j}^{R}\leq 4 \mu_{j}^{R} $, after that we changed the order of integration using the boundedness of the measures and we concluded using $  \mu_{j}^{R}\left(\stw\right)\leq 8 $ as well as $$ \int_\stw dn \chi_{\{0<\left|\omega\cdot n\right|<\kappa\}}(n)<4\pi\kappa.$$
 	
 	On the other hand we have for fixed $ \kappa>0 $\begin{multline}\label{measure21}
 	\int_\stw dn \int_\stw\sum_{\substack{k\in \pi\mathbb{Z}^3\\|k|> {R}}}4\left|a_k^{j}\right|^2\left|\frac{\sin(k\cdot n)}{k\cdot n}\right|^2 \delta_{\frac{k}{|k|}}(\omega)\left.\right|_{\{|\omega\cdot n|\geq\kappa\}}\\\leq\int_\stw dn \int_\stw\frac{4}{{R}^2\kappa^2}\sum_{\substack{k\in \pi\mathbb{Z}^3\\|k|> {R}}}\left|a_k^{j_l}\right|^2 \delta_{\frac{k}{|k|}}(\omega)\left.\right|_{\{|\omega\cdot n|\geq\kappa\}}\\\leq\int_\stw dn\int_\stw\frac{4}{R^2\kappa^2}d\mu_j(\omega)\left.\right|_{\{|\omega\cdot n|\geq\kappa\}}\leq \frac{128 \pi}{R^2\kappa^2} \underset{R\to\infty}{\longrightarrow}0.
 	\end{multline}
 	uniformly in $ j\in\mathbb{N} $. We used indeed that if $ |k|\omega=k $, $ |k|>R $ and $ |\omega\cdot n|>\kappa $, then $ |k\cdot n|>R\kappa $. Moreover, we can always bound the measure $ \mu_j^R\leq\mu_j $ and $ \mu_j(\stw)\leq 8 $.
 	
 	Hence, we conclude $ \int_\stw dn\int_\stw d\nu_{n,j_l}^{R}(\omega)\to 0 $ as $ R\to 0 $. Indeed, let $ \eps>0 $. We chose $ 0<\kappa_0< \frac{1}{256 \pi \eps} $. Then testing according to \eqref{measure21} we define $ R_0(\eps)> \frac{16\sqrt{\pi}}{\kappa_0\sqrt{\eps}} $ such that for all $ R\geq R_0 $ we have 
 	\begin{equation}\label{measure31}
 	\int_\stw dn\int_\stw d\nu_{n,j}^{R}\left.\right|_{\{|\omega\cdot n|\geq\kappa_0\}}(\omega)<\frac{\eps}{2} .
 	\end{equation}
 	Combining \eqref{measure11} for $ \kappa_0 $ and \eqref{measure31} we obtain
 	\begin{equation}\label{finalmeasure1}
 	0\leq\int_\stw dn\int_\stw d\nu_{n,j}^{R}(\omega)<\eps
 	\end{equation}
 	for all $ R\geq R_0 $ and most importantly for all $ j\in\mathbb{N} $.

  We are now ready to show the compactness of the sequence $ T_m[\vphi_{j}] $ in $ L^2\left(\Pi^3\right) $. Since $ T_m $ is a bounded operator, $\Pi^3 $ is a compact subset of $ \rth $, we only have to show the equi-integrability condition (cf. \cite{MR2759829}). We recall that $ \Arrowvert L_n[\vphi]\Arrowvert_\infty\leq 2 $. Hence, let $ x,h\in \Pi^3 $ using Jensen's inequality we compute
  \begin{multline}\label{lemmacompactness1}
  \left|T_m[\vphi_{j}](x)-T_m[\vphi_{j}](x+h)\right|^2=\left|\int_\stw dn\left[\left(L_n[\vphi_{j}](x)\right)^m-\left(L_n[\vphi_{j}](x+h)\right)^m\right]\right|^2\\
  \leq \left(m2^{m-1}\right)^2\left(\int_\stw dn\left|L_n[\vphi_{j}](x)-L_n[\vphi_{j}](x+h)\right|\right)^2\\
  \leq \left(4\pi m2^{m-1}\right)^2 \int_\stw dn\left|\sum_{k\in\pi\mathbb{Z}^3} 2a_k^{j}\frac{\sin(k\cdot n)}{k\cdot n}e^{ik\cdot x}\left(1-e^{ik\cdot h}\right)\right|^2.
  \end{multline}
 Since $ \left\{\frac{1}{8}e^{i k\cdot x}\right\}_{k\in \pi\mathbb{Z}^3} $ is an orthonormal basis of $ \Pi^3 $ we obtain denoting by $ C_m=8^24\pi m2^{m-1} $,
 \begin{multline}\label{lemmacompactness2}
 \int_{\Pi^3}dx\left|T_m[\vphi_{j}](x)-T_m[\vphi_{j}](x+h)\right|^2\\\leq \frac{C_m^2}{64} \int_{\Pi^3}dx\int_\stw dn\;\left|L_n[\vphi_{j}](x)-L_n[\vphi_{j}](x+h)\right|^2\\= C_m^2 \int_\stw dn\sum_{k\in\pi\mathbb{Z}^3} 4\left|a_k^{j}\right|^2\left|\frac{\sin(k\cdot n)}{k\cdot n}\right|^2\left|\left(1-e^{ik\cdot h}\right)\right|^2\\
 \leq 32\pi C_m^2 \sum_{\substack{k\in \pi\mathbb{Z}^3\\|k|\leq R}} 4 \left|a_k^{j}\right|^2 |k|^2 |h|^2
 + 32C_m^2 \int_\stw dn \int_\stw d\nu_{n,j}^R(\omega).
 \end{multline}
 Let $ \eps>0 $. We have shown that there exists $ R_0>0 $ such that $$ \int_\stw dn \int_\stw d\nu_{n,j}^R(\omega)<\frac{\eps}{64 C_m^2} ,$$ for all $ R\geq R_0 $ and for all $ j\in\mathbb{N} $. Taking in \eqref{lemmacompactness2} $ R=R_0 $ and $ h_0=\frac{\sqrt{\eps}}{C_m 32\sqrt{{2}} R_0 }$ we obtain the desired equi-integrability condition
 \begin{multline*}
 \int_{\Pi^3}dx\left|T_m[\vphi_{j}](x)-T_m[\vphi_{j}](x+h)\right|^2\\\leq C_m \int_{\Pi^3}dx\int_\stw dn\;\left|L_n[\vphi_{j}](x)-L_n[\vphi_{j}](x+h)\right|^2<\eps
 \end{multline*}
 for all $ |h|<h_0 $. 
 This concludes the proof of Proposition \ref{compact}.
\end{proof}
We can get a stronger result for the compactness of line integrals of functions depending also on the direction $n\in\stw$. We will use it in the proof of existence of solution to the equation containing also the scattering term.
\begin{corollary}\label{corollary.cpt1}
 Let $\Pi^3=[-L,L]^3$ and $ (\vphi(x,n)_j)_{j\in\mathbb{N}}\in C\left(\stw,L^2\left(\Pi^3\right)\cap L^\infty\left(\Pi^3\right)\right) $ be a sequence of periodic functions with $ \sup\limits_{n\in\stw}\Arrowvert \vphi_j(\cdot,n)\Arrowvert_{L^\infty(\Pi^3)}\leq M $. Assume also $ \Arrowvert \vphi_j(\cdot,n_1)-\vphi_j(\cdot,n_2)\Arrowvert_{L^\infty(\Pi^3)}\leq \sigma(d(n_1,n_2))\to 0$ uniformly in $ j\in\mathbb{N} $ if $d(n_1,n_2)\to 0 $, where $ d $ is the metric on the sphere and $\sigma\in C\left(\mathbb{R}_+,\mathbb{R}_+\right)$ with $\sigma(0)=0$ is a uniform modulus of continuity. For $ n\in\stw $ and $m\in\mathbb{N}$ we define the operators $ L_n $ and $T_m$ by
\begin{equation*}
L_n[\vphi](x,\omega)\eqdef \int_{-1}^{1}d\lambda\ \vphi(x-\lambda n,\omega) \;\;\;\text{and}\;\;\;T_m[\vphi](x)\eqdef \int_\stw dn \left(L_n[\vphi](x,n)\right)^m
\end{equation*}
Then for every $ m\in\mathbb{N} $ the sequence $ \left(T_m[\vphi_j]\right)_j $ is compact in $ L^2\left(\Pi^3\right) $.
\begin{proof} Without loss of generality we can assume again $L=1$ and $M=1$.
This statement is a corollary to Proposition \ref{compact} and the Besicovitch covering Lemma. Since $ \stw $ with the geodesic metric is a Riemannian Manifold of class greater than 2, it is also a directionally (1,C)-limited metric space for a fixed constant $ C>0 $. See \cite{federer}  for further reference. This implies that the Federer-Besicovitch covering Lemma (a generalization of the well-known Lemma in $ \mathbb{R}^n $) applies. Hence, for any family $ \mathcal{F}_\delta=\left\{B_\delta(n)\right\}_{n\in\stw} $ of balls with radius $ \delta<1 $ there exists subfamilies $ \mathcal{G}_k\subset \mathcal{F}_\delta $ for $ 1\leq k\leq 2C+1 $ consisting of disjoint balls such that $$ \stw\subset \bigcup_{k=1}^{2C+1}\bigsqcup_{B\in\mathcal{G}_k} B, $$ where $\bigsqcup$ denotes the disjoint union. Since $ \stw $ is compact there exists also a finite cover, i.e., the subfamilies $ \mathcal{G}_k $ are finite. Hence,  $$ \stw\subset \bigcup_{k=1}^{2C+1}\bigsqcup_{1\leq i\leq N(k,\delta)} B_\delta(n_{k,i}). $$ Let now $ \eps>0 $ and $ h\in\rth $. Similarly as in equation \eqref{lemmacompactness1} we estimate using first Jensen's inequality
\begin{multline}
 \int_{\Pi^3}dx\left|T_m[\vphi_{j}](x)-T_m[\vphi_{j}](x+h)\right|^2\\= \int_{\Pi^3}dx\left|\int_\stw dn\left[\left( \int_{-1}^{1}d\lambda\ \vphi_j(x-\lambda n,n)\right)^m-\left( \int_{-1}^{1}d\lambda\ \vphi_j(x+h-\lambda n,n)\right)^m\right]\right|^2\\
 \leq \frac{C_m}{4\pi}  \int_{\Pi^3}dx\left[\int_\stw dn \left|\int_{-1}^{1}d\lambda \left[\vphi_j(x-\lambda n,n)- \vphi_j(x+h-\lambda n,n)\right]\right|\right]^2\\
 \leq C_m  \int_{\Pi^3}dx\int_\stw dn \left|\int_{-1}^{1}d\lambda \left[\vphi_j(x-\lambda n,n)- \vphi_j(x+h-\lambda n,n)\right]\right|^2\\
 = C_m  \int_{\Pi^3}dx\int_{\bigcup_{k=1}^{2C+1}\bigsqcup_{i=1}^{N(k,\delta)} B_\delta(n_{k,i})} dn\left| \int_{-1}^{1}d\lambda \left[\vphi_j(x-\lambda n,n)- \vphi_j(x+h-\lambda n,n)\right]\right|^2\\
 \leq C_m \sum_{k=1}^{2C+1}  \int_{\Pi^3}dx\int_{\bigsqcup_{i=1}^{N(k,\delta)} B_\delta(n_{k,i})} dn \left|\int_{-1}^{1}d\lambda \left[\vphi_j(x-\lambda n,n)- \vphi_j(x+h-\lambda n,n)\right]\right|^2\\
 \leq C_m \sum_{k=1}^{2C+1}  \int_{\Pi^3}dx\sum_{i=1}^{N(k,\delta)}\int_{ B_\delta(n_{k,i})} dn \left|\int_{-1}^{1}d\lambda \left[\vphi_j(x-\lambda n,n_{k,i})- \vphi_j(x+h-\lambda n,n_{k,i})\right]\right|^2\\
 + C_m \sum_{k=1}^{2C+1}  \int_{\Pi^3}dx\sum_{i=1}^{N(k,\delta)}\int_{ B_\delta(n_{k,i})} dn 4 \sigma(\delta)^2\\
 \leq C_m N(k,\delta)\sum_{k=1}^{2C+1}\sum_{i=1}^{N(k,\delta)}\int_{\Pi^3}dx\int_{\stw} dn \left|\int_{-1}^{1}d\lambda \left[\vphi_j(x-\lambda n,n_{k,i})- \vphi_j(x+h-\lambda n,n_{k,i})\right]\right|^2\\
 + 4C_m (2C+1) \int_{\Pi^3}dx\int_\stw dn \sigma(\delta)^2,
\end{multline}
where in the last inequality we used that the balls $ \left\{B_\delta(n_{k,i})\right\}_{1\leq i\leq N(k,\delta)} $ are disjoint. We choose thus $ \delta_0>0 $ such that $ 4C_m (2C+1) \sigma(\delta_0)<\frac{\eps}{64} $. Lemma \ref{lemmacompactness1} with equation \eqref{actual.equiintegrability.torus} implies for any $ (k,i) $ with $ 1\leq k\leq 2C+1 $ and $ 1\leq i\leq N(k,\delta_0) $ the existence of some $ h_0(k,i) $ such that 
\begin{multline*}
\int_{\Pi^3}dx\int_{\stw} dn \left|\int_{-1}^{1}d\lambda \left[\vphi_j(x-\lambda n,n_{k,i})- \vphi_j(x+h-\lambda n,n_{k,i})\right]\right|^2\\<\frac{\eps}{2 C_m(2C+1)}\frac{1}{N(k,\delta_0)^2}
\end{multline*}
for all $ |h|<h_0(i,k) $ and for all $ j\in\mathbb{N} $.
Hence, choosing $ h_0=\min\limits_{\substack{1\leq k\leq 2C+1\\1\leq i\leq N(k,\delta_0)}} \left\{h_0(k,i)\right\} $ we conclude
\begin{equation*}
\int_{\Pi^3}dx\left|T_m[\vphi_{j}](x)-T_m[\vphi_{j}](x+h)\right|^2<\eps
\end{equation*}
for all $ |h|<h_0 $ and all $ j\in\mathbb{N} $. Hence, the sequence $ \left(T_m[\vphi_j]\right)_j $ is compact in $ L^2\left(\Pi^3\right) $.
\end{proof}
\end{corollary}
We extend now Proposition \ref{compact} to other more general type of operators involving line integrals. To this end we define for $\vphi\in L^\infty(\Pi^3)$ and $0\leq s<t\leq \frac{L}{2}$, $x\in\Pi^3$ and $n\in\stw$ the line integral
\begin{equation}\label{line integral}
L_{n,t-s}[\vphi](x)=\int_s^t d\lambda \vphi(x-\lambda n).
\end{equation}
Then the following lemma holds.
\begin{lemma}\label{general compact}
    Under the notation above let $ (\vphi_j)_{j\in\mathbb{N}}\in 
    L^\infty\left(\Pi^3\right) $ be a sequence of periodic functions with $ \Arrowvert \vphi_j\Arrowvert_\infty\leq M $. Let $\eps>0$, then there exists $h_0>0$ such that 
    \begin{multline}\label{equiintegrability general}
        \int_{\Pi^3}dx\left|\fint_\stw dn \left(L_{n,t-s}[\vphi_j](x)-L_{n,t-s}[\vphi_j](x+h)\right)\right|^2\\\leq \int_{\Pi^3}dx\fint_\stw dn \left|L_{n,t-s}[\vphi_j](x)-L_{n,t-s}[\vphi_j](x+h)\right|^2<\eps
    \end{multline}
    for all $|h|<h_0$ and $j\in\mathbb{N}$ and uniformly in $t,s$. This equi-integrability condition implies as in Proposition \ref{compact} the compactness of any sequence $\fint_\stw dn \left(L_{n,t-s}[\vphi_j](x)\right)^m$ in $L^2(\Pi^3)$ for any fixed $m\in\mathbb{N}$.
    \begin{proof}
        We expand the functions $\vphi_j$ in their respectively Fourier series as $\vphi_j(x)=\sum_{k\in\frac{\pi}{L}\mathbb{Z}^3} a_k^j e^{ik\cdot x}$, hence their associated auxiliary measures $\mu_j^R$ are given for $R\in \frac{\pi}{L}\mathbb{Z}$ by 
        \begin{equation}\label{measure}
\mu_j^R=\sum_{\substack{k\in\frac{\pi}{L}\mathbb{Z}^3\\|k|>R}} |a_k^j|^2\delta_{\frac{k}{|k|}}\in\mathcal{M}_+\left(\stw\right) . 
	\end{equation}
	For any $ s,t\in\left[0,\frac{L}{2}\right] $ with $ s<t $ we compute for $ k\in\frac{\pi}{L}\mathbb{Z}^3 $ and $ n\in\stw $
\begin{equation*}\int_s^t d\lambda e^{-ik\cdot\lambda n}=\frac{e^{-it(k\cdot n)}-e^{-is(k\cdot n)}}{-ik\cdot n}= 2e^{-i(t+s)\frac{(k\cdot n)}{2}}\frac{\sin\left((t-s)\frac{(k\cdot n)}{2}\right)}{k\cdot n}.
	\end{equation*}
	Therefore, the auxiliary measures associated to the operators $ L_{n,t-s} $ acting respectively on $\vphi_j $ are given by
	\begin{equation}\label{nut}
	\nu_{n,j}^{t-s,R}(\omega)= \sum_{\substack{k\in\frac{\pi}{2D+2}\mathbb{Z}^3\\|k|>R}} |a_k^j|^24\left|\frac{\sin\left((t-s)\frac{(k\cdot n)}{2}\right)}{k\cdot n}\right|^2\delta_{\frac{k}{|k|}}\in\mathcal{M}_+\left(\stw\right).
	\end{equation}
	
 Since $ \left|\frac{\sin(ax)}{x}\right|\leq a $ we notice that the auxiliary measures are uniformly bounded and satisfy $$ \nu_{n,j}^{t-s,R}\leq (t-s)^2\mu_{j}^R\leq \frac{L^2}{4}\mu_j^R\leq \frac{L^2}{4}\mu_j .$$
	Hence, exactly as we have argued in Proposition \ref{compact} we see also in this case that for any $ 0<\kappa<1 $ 
	\begin{equation*}
	\int_\stw dn\int_\stw d\nu_{n,j}^{t-s,R}(\omega)\chi_{\{0\leq |\omega\cdot n|<\kappa\}}\leq\frac{L^2}{4} C(L,M)\kappa\to 0,
	\end{equation*}
	as $ \kappa\to 0 $ uniformly in $ j\in\mathbb{N} $ and $ t,s\in\left[0,\frac{L}{2}\right] $. Moreover, estimating the sine function by $ 1 $ we also have for fixed $ \kappa>0 $
	\begin{equation*}
	\int_\stw dn\int_\stw d\nu_{n,j}^{t-s,R}(\omega)\chi_{\{|\omega\cdot n|\geq\kappa\}}\leq\frac{C(L,M)}{R^2 \kappa^2}\to 0,
	\end{equation*}
	s $ R\to\infty $ uniformly in $ j\in\mathbb{N} $ and $ t,s\in\left[0,\frac{L}{2}\right] $.
	Thus, we conclude once again that for any $ \eps>0 $ there exists some $ R_0(\eps)>0 $ (independent of $ j\in\mathbb{N} $ and $ t,s\in\left[0,\frac{L}{2}\right] $) such that
	\begin{equation*}
	\int_\stw dn\int_\stw d\nu_{n,j}^{t-s,R}(\omega)<\eps,
	\end{equation*}
	for all $ R\geq R_0 $, for all $ j\in\mathbb{N} $ and for all $ t,s\in\left[0,\frac{L}{2}\right] $.

 Let us define $ C_L=\left|\Pi^3\right|=(2L)^3 $. We can write for any $ t,s\in\left[0,\frac{L}{2}\right] $ with $ s<t $ using first Jensen's inequality and secondly that $ \left\{\frac{1}{C_L}e^{ik\cdot x}\right\}_{k\in \frac{\pi}{L}\mathbb{Z}^3} $ form an orthonormal basis of $ \Pi^3 $
	\begin{multline*}
	\int_{\Pi_D^3} \left|\fint_\stw dn\;L_{n,t-s}[\vphi_j](x)-L_{n,t-s}[\vphi_j](x+h)\right|^2\\
 \leq\int_{\Pi_D^3} \fint_\stw dn\left|L_{n,t-s}[\vphi_j](x)-L_{n,t-s}[\vphi_j](x+h)\right|^2\\\leq C_L^2\sum_{\substack{k\in\frac{\pi}{L}\mathbb{Z}^3\\|k|\leq R}} \fint_\stw dn\left|a_k^j\right|^2 4 \left|\frac{\sin\left((t-s)\frac{(k\cdot n)}{2}\right)}{k\cdot n}\right|^2 \left|1-e^{ik\cdot h}\right|^2+4C_L^2 \fint_\stw dn\int_\stw d\nu_{n,j}^{(t-s),R}(\omega)\\\leq C_L^3M^2 (t-s)^2R^2 |h|^2 + 4C_L^2 \fint_\stw dn\int_\stw d\nu_{n,j}^{(t-s),R}(\omega)\\\leq C_L^3M^2 \frac{L^2}{4}R^2 |h|^2 + 4C_L^2\fint_\stw dn\int_\stw d\nu_{n,j}^{(t-s),R}(\omega).
	\end{multline*}
	Hence, taking $R=R_0\left(\frac{\eps}{8C_L^2}\right)$ and $h_0=\frac{2\sqrt{\eps}}{MC_LR_0L\sqrt{C_L}}$ we conclude the desired equi-integrability result
 \begin{multline*}
	\int_{\Pi_D^3} \left|\fint_\stw dn\;L_{n,t-s}[\vphi_j](x)-L_{n,t-s}[\vphi_j](x+h)\right|^2\\
 \leq\int_{\Pi_D^3} \fint_\stw dn\left|L_{n,t-s}[\vphi_j](x)-L_{n,t-s}[\vphi_j](x+h)\right|^2<\eps
	\end{multline*}
 for all $|h|<h_0$ uniformly in$j\in\mathbb{N}$ and $ t,s\in\left[0,\frac{L}{2}\right] $.
    \end{proof}
\end{lemma}

\subsection{Proof of Theorems \ref{main theorem pseudo grey} and \ref{Grey thm}}\label{sec.grey proof}
We can now prove Theorems \ref{Grey thm} and \ref{main theorem pseudo grey}. A crucial step will be to adapt Proposition \ref{compact} in order to show the compactness of the operators $\mathcal{B}^\eps$ instead of the operator defined only by one line integral.
\begin{proof}[Proof of Theorem \ref{Grey thm}]
We first extend by $ 0 $ the function $ u_\eps $ and $ \gamma(u_\eps) $. Assuming without loss of generality that $ 0\in\Omega $, since $ \phi_\eps $ has compact support in $ B_\eps(0)\subset B_1(0) $ for all $ \eps<1 $ we see that $ \gamma(u_\eps)*\phi_\eps $ and $ u_\eps $ have both support contained in $ [-D-1,D+1]^3 $, where we denote by $ D=\text{diam}(\Omega) $. Let us extend periodically in $ \rth $ both functions $ u_\eps $ and $ \gamma(u_\eps)*\phi_\eps $ in $ \Pi_D^3\eqdef [-2D-2,2D+2]^3 $. Then we see that $$ \int_\Omega dx \int_\stw dn\int_0^{s(x,n)} u_\eps(x-rn) dr= \int_\Omega dx \int_\stw dn\int_0^D u_\eps(x-rn) dr. $$ With the same notation of Lemma \ref{general compact} we consider the operators $L_{n,r-s}[\vphi](x)$ acting on $ \vphi\in L^\infty (\Pi_D^3) $, $ n\in\stw $ and $ x\in \Pi_D^3 $ given by \eqref{line integral}. In the case $s=0$ we simplify the notation by $L_{n,r-0}[\vphi](x)=L_{n,r}[\vphi](x)$. Using \eqref{convolution} and the radial symmetry of $ \phi_\eps $ we see
	\begin{equation*}
	\int_\rth d\xi \gamma(u_\eps)(\xi)\delta_{[x,x-rn]}*\phi_\eps (\xi)=\int_0^r d\lambda \left(\gamma(u_\eps)*\phi_\eps\right)(x-\lambda n)=L_{n,r}\left[\gamma(u_\eps)*\phi_\eps\right](x).
	\end{equation*}
	Thus, we can write the operator $ \mathcal{B}^\eps(u_\eps) $ in the following way changing the variables according to $\eta=x-rn$
	\begin{multline}\label{polarB}
	\mathcal{B}^\eps(u_\eps)(x)=\int_0^Ddr \fint_\stw dn \left[\left(\gamma(u_\eps)*\phi_\eps\right)u_\eps\right](x-rn)\exp\left(-L_{n,r}\left(\gamma(u_\eps)*\phi_\eps\right)(x)\right)\\
	+\int_\stw dn\ G(n) \exp\left(-L_{n,s(x,n)}\left(\gamma(u_\eps)*\phi_\eps\right)(x)\right)\\
	=\int_0^Ddr \fint_\stw dn \left[\left(\gamma(u_\eps)*\phi_\eps\right)u_\eps\right](x-rn)\sum_{m=0}^{\infty} \frac{(-1)^m}{m!}\left(L_{n,r}\left(\gamma(u_\eps)*\phi_\eps\right)(x)\right)^m\\
	+ \int_\stw dn\ G(n) \sum_{m=0}^{\infty} \frac{(-1)^m}{m!}\left(L_{n,s(x,n)}\left(\gamma(u_\eps)*\phi_\eps\right)(x)\right)^m.
	\end{multline}
	Since the sequence $ \mathcal{B}^\eps(u_\eps) $ is uniformly bounded by $ K $ in $ L^\infty(\Omega) $, and thus by $ |\Omega|^{\frac{1}{2}} K $ in $ L^2(\Omega) $, for the compactness we need again to show only the equi-integrability. Let now $ h\in\rth $ with $ |h|<\frac{1}{2} $ and $ \eps<\frac{1}{2} $. Since we extended by $0$ the function $u_\eps=\mathcal{B}^\eps(u_\eps)$ outside $\Omega$ it is true that $\mathcal{B}^\eps(u_\eps)(x+h)=\mathcal{B}^\eps(u_\eps)(x+h)\chi_{\{{x+h\in\Omega}\}}(x)$. Hence, we multiply by this characteristic function also the integral definition of the operator as in \eqref{polarB}, this guarantees the well-definiteness of the function at $x+h\not\in\Omega$. We thus compute using Jensen's inequality
	\begin{multline*}
	\int_\Omega dx \left|\mathcal{B}^\eps(u_\eps)(x)-\mathcal{B}^\eps(u_\eps)(x+h)\right|^2\\
	\leq C \int_\Omega dx \fint_\stw dn\left|\int_0^Ddr\left[\left(\gamma(u_\eps)*\phi_\eps\right)u_\eps\right](x-rn)\sum_{m=0}^{\infty} \frac{(-1)^m}{m!}\left(L_{n,r}\left(\gamma(u_\eps)*\phi_\eps\right)(x)\right)^m\right.\\\left.-\int_0^Ddr\chi_{\{{x+h\in\Omega}\}}\left[\left(\gamma(u_\eps)*\phi_\eps\right)u_\eps\right](x+h-rn)\sum_{m=0}^{\infty} \frac{(-1)^m}{m!}\left(L_{n,r}\left(\gamma(u_\eps)*\phi_\eps\right)(x+h)\right)^m\right|^2\\
	+C(G) \int_\Omega dx \int_\stw dn\ G(n)\left|\sum_{m=0}^{\infty} \frac{(-1)^m}{m!}\left(L_{n,s(x,n)}\left(\gamma(u_\eps)*\phi_\eps\right)(x)\right)^m\right.\\\left.-\sum_{m=0}^{\infty} \frac{(-1)^m}{m!}\chi_{\{{x+h\in\Omega}\}}\left(L_{n,s(x+h,n)}\left(\gamma(u_\eps)*\phi_\eps\right)(x+h)\right)^m\right|^2.
	\end{multline*}
	Applying now the triangle inequality we can further estimate
	\begin{multline}\label{estimate0}
	\int_\Omega dx \left|\mathcal{B}^\eps(u_\eps)(x)-\mathcal{B}^\eps(u_\eps)(x+h)\right|^2\\
	\leq C \Arrowvert\gamma\Arrowvert_\infty^2 K^2 \int_\Omega dx\chi_{\{{x+h\in\Omega}\}}\int_0^Ddr \fint_\stw dn\left|\sum_{m=0}^{\infty} \frac{(-1)^m}{m!}\left(L_{n,r}\left(\gamma(u_\eps)*\phi_\eps\right)(x)\right)^m\right.\\\left.-\sum_{m=0}^{\infty} \frac{(-1)^m}{m!}\left(L_{n,r}\left(\gamma(u_\eps)*\phi_\eps\right)(x+h)\right)^m\right|^2\\
	+C\int_\Omega dx\chi_{\{{x+h\in\Omega}\}} \fint_\stw dn\left|\int_0^D dr\; \sum_{m=0}^\infty \frac{(-1)^m}{m!}\left(L_n,r\left(\gamma(u_\eps)*\phi_\eps\right)(x)\right)^m\right.\\\times \left[\left(\left(\gamma(u_\eps)*\phi_\eps\right)u_\eps\right)(x-rn)-\left(\left(\gamma(u_\eps)*\phi_\eps\right)u_\eps\right)(x+h-rn)\right]\bigg|^2\\
	+C(G) \int_\Omega dx \chi_{\{{x+h\in\Omega}\}}\fint_\stw dn \left|\sum_{m=0}^{\infty} \frac{(-1)^m}{m!}\left(L_{n,D}\left(\gamma(u_\eps)*\phi_\eps\right)(x)\right)^m\right.\\\left.-\sum_{m=0}^{\infty} \frac{(-1)^m}{m!}\left(L_{n,D}\left(\gamma(u_\eps)*\phi_\eps\right)(x+h)\right)^m\right|^2\\
	+ C(G) \int_\Omega dx \chi_{\{{x+h\in\Omega}\}}\fint_\stw dn\left\{\left[e^{\left(-L_{n,s(x,n)}\left(\gamma(u_\eps)*\phi_\eps\right)(x)\right)}-e^{\left(-L_{n,D}\left(\gamma(u_\eps)*\phi_\eps\right)(x)\right)}\right]^2\right.\\\left.+\left[e^{\left(-L_{n,s(x+h,n)}\left(\gamma(u_\eps)*\phi_\eps\right)(x+h)\right)}-e^{\left(-L_{n,D}\left(\gamma(u_\eps)*\phi_\eps\right)(x+h)\right)}\right]^2\right\}\\
	+\left(C(G)+K^2\right)\left|\{x\in\Omega:x+h\not\in\Omega\}\right|,
	\end{multline}
	where the term in the second line is obtained applying Jensen's inequality again and in the last term we estimate the exponential by $ 1 $.
	We notice that $$ \Arrowvert L_{n,r}[\gamma(u_\eps)*\phi_\eps]\Arrowvert_\infty\leq D \Arrowvert\gamma\Arrowvert_\infty, $$ for any $ 0\leq r\leq D $ and any $n\in\stw$. Hence, for any $ \delta>0 $ there exists some $ M>0 $ such that $$ 	\left\Arrowvert \sum_{m>M}^{\infty} \frac{(-1)^m}{m!}\left|L_{n,r}\left(\gamma(u_\eps)*\phi_\eps\right)(\cdot)\right|^{m}\right\Arrowvert^2_\infty<\frac{\delta}{2},  $$ for all $ r\in[0,D] $ and $n\in\stw$. Moreover, the smoothness of $\bnd$ implies that there exists a constant $ C(\Omega)>0 $ such that $$ \left|\{x\in\Omega:x+h\not\in\Omega\}\right|\leq C(\Omega)|h| .$$ In addition to that the convexity of the domain $\Omega$ and a geometric argument implies that, since $ \gamma(u_\eps)*\phi_\eps $ is supported in $ \overline{\Omega+B_\eps(0)} $, there exists a constant $ C(\Omega)>0 $ which depends on the curvature of $ \Omega $ such that $$ \sup\limits_{\substack{x\in\Omega\\n\in\stw}}\int_{s(x,n)}^D \gamma(u_\eps)*\phi_\eps(x-\lambda n)d\lambda \leq C(\Omega)\Arrowvert\gamma\Arrowvert_\infty \sqrt{\eps}, $$ where $\sqrt{\eps}$ is due to the set of directions $n\in\stw$ that are tangent to the boundary $\bnd$. Thus, using also the well-known inequality $ |e^{-b}-e^{-a}|\leq |a-b| $ for $a,b\geq 0$ we compute
	\begin{multline}\label{estimate1}
	\int_\Omega dx \left|\mathcal{B}^\eps(u_\eps)(x)-\mathcal{B}^\eps(u_\eps)(x+h)\right|^2\\
	\leq C(D,G)\Arrowvert\gamma\Arrowvert_\infty K^2 \sup_{n\in\stw}\sup_{0\leq r\leq D}\left\Arrowvert \sum_{m>M}^{\infty} \frac{(-1)^m}{m!}\left|L_{n,r}\left(\gamma(u_\eps)*\phi_\eps\right)(\cdot)\right|^{m}\right\Arrowvert^2_\infty\\
	+ C \Arrowvert\gamma\Arrowvert_\infty^2 K^2 \sum_{m=0}^M (M+1)\left(\frac{m}{m!}\right)^2\left(\Arrowvert \gamma\Arrowvert_\infty D\right)^{2(m-1)} \int_{\Pi_D^3} dx \int_0^Ddr \\\times \fint_\stw dn \left|L_{n,r}\left(\gamma(u_\eps)*\phi_\eps\right)(x)-L_{n,r}\left(\gamma(u_\eps)*\phi_\eps\right)(x+h)\right|^2\\
	+C  \int_{\Pi_D^3} dx \fint_\stw dn\sum_{m=0}^{M}\frac{M+1}{(m!)^2}\Arrowvert\gamma\Arrowvert_\infty^{2m}D^m\int_0^D d\lambda_1...\int_0^D d\lambda_m\left|\int_{\max_{0\leq i\leq M}(\lambda_i)}^Ddr\right.\\\times \big[ \left(\left(\gamma(u_\eps)*\phi_\eps\right)u_\eps\right)(x-rn)-\left(\left(\gamma(u_\eps)*\phi_\eps\right)u_\eps\right)(x+h-rn)\big]\bigg|^2\\
	+C(G) \sum_{m=0}^M (M+1) \left(\frac{m}{m!}\right)^2\left(\Arrowvert \gamma\Arrowvert_\infty D\right)^{2(m-1)} \int_{\Pi_D^3} \fint_\stw dn \left|L_{n,D}\left(\gamma(u_\eps)*\phi_\eps\right)(x)\right.\\\left.-L_{n,D}\left(\gamma(u_\eps)*\phi_\eps\right)(x+h)\right|^2+C(G,\Omega) \left(\Arrowvert\gamma\Arrowvert_\infty^2\eps+K^2|h|^2\right).
	\end{multline}
	In order to obtain these last estimates we used also that $ \Omega \subset \Pi_D^3 $ since we are considering non-negative integrands. Moreover, in order to obtain the term containing\begin{multline*}
\int_0^D d\lambda_1...\int_0^D d\lambda_m\left|L_{n,D-\max(\lambda)}[\left(\gamma(u_\eps)*\phi_\eps\right)u_\eps](x)\right.\\\left.- L_{n,D-\max(\lambda)}[\left(\gamma(u_\eps)*\phi_\eps\right)u_\eps](x+h)\right|^2
	\end{multline*} for $ \max(\lambda)=\max_{0\leq i\leq M}(\lambda_i) $ we changed the order of integration applying Fubini's Theorem and we saw recursively that
	\begin{multline}\label{fubini}
		\int_0^D dr\left(\int_0^r d\lambda\right)^m=\int_0^D dr\int_0^r d\lambda_1...\int_0^r d\lambda_m\\=\int_0^D d\lambda_1...\int_0^D d\lambda_m\int_{\max_{0\leq i\leq M}(\lambda_i)}^Ddr.
	\end{multline}
	Hence, applying Fubini's Theorem and afterwards Jensen's inequality we conclude
	\begin{multline*}
		\left|\int_0^D dr\left(\int_0^r d\lambda \left(\gamma(u_\eps)*\phi_\eps\right)(x-\lambda n) \right)^m \right.\\\times \left[\left(\left(\gamma(u_\eps)*\phi_\eps\right)u_\eps\right)(x-rn)-\left(\left(\gamma(u_\eps)*\phi_\eps\right)u_\eps\right)(x+h-rn)\right]\bigg|^2\\
		=	\left|\int_0^D d\lambda_1 \left(\gamma(u_\eps)*\phi_\eps\right)(x-\lambda_1 n)...\int_0^Dd\lambda_m\left(\gamma(u_\eps)*\phi_\eps\right)(x-\lambda_m n) \right.\\\times \int_{\max(\lambda)}^Ddr\left[\left(\left(\gamma(u_\eps)*\phi_\eps\right)u_\eps\right)(x-rn)-\left(\left(\gamma(u_\eps)*\phi_\eps\right)u_\eps\right)(x+h-rn)\right]\bigg|^2\\
		\leq D^m \Arrowvert\gamma\Arrowvert^{2m}\int_0^D d\lambda_1...\int_0^D d\lambda_m\left|\int_{\max(\lambda)}^Ddr\left[\left(\left(\gamma(u_\eps)*\phi_\eps\right)u_\eps\right)(x-rn)\right.\right.\\\left.-\left(\left(\gamma(u_\eps)*\phi_\eps\right)u_\eps\right)(x+h-rn)\right]\bigg|^2.
		\end{multline*}
  
	We apply now the modification of Proposition \ref{compact} as in Lemma \ref{general compact}. Let us take the sequence $ \eps_j\eqdef \frac{1}{j} $ for $ j\in\mathbb{N} $. As we can notice in equation \eqref{estimate1} we have to consider the operators $ L_{n,t} $ and $ L_{n,D-t} $ for some $ 0\leq t\leq D $ acting respectively on two different sequences of functions, i.e. $ \left(\gamma(u_{\eps_j})*\phi_{\eps_j}\right)_j $ respectively $ \left(\left[\gamma(u_{\eps_j})*\phi_{\eps_j}\right]u_{\eps_j}\right)_j $. In order to simplify the notation we write $ f_j $ instead of $ f_{\eps_j} $. We recall that these sequences are uniformly bounded. Young's convolution inequality implies indeed $$ \sup\limits_{j\geq 0}\Arrowvert\gamma(u_j)*\phi_j\Arrowvert_\infty\leq \Arrowvert \gamma\Arrowvert_\infty\  \text{  and  } \ \sup\limits_{j\geq 0}\Arrowvert\left(\gamma(u_j)*\phi_j\right)u_j\Arrowvert_\infty\leq K \Arrowvert \gamma\Arrowvert_\infty .$$ Hence, we can apply Lemma \ref{general compact} to these sequences.
	Let now $ \beta>0 $ be arbitrarily small. We consider the terms appearing in \eqref{estimate1}. The convergence of the exponential implies the existence of $ M_0(\beta)>0 $ such that $$  C(D,G)\Arrowvert\gamma\Arrowvert_\infty K^2 \sup_{n\in\stw}\sup_{0\leq r\leq D}\left\Arrowvert \sum_{m>M}^{\infty} \frac{(-1)^m}{m!}\left|L_{n,r}\left(\gamma(u_j)*\phi_j\right)(\cdot)\right|^{m}\right\Arrowvert^2_\infty<\frac{\beta}{4}, $$ for any $ M\geq M_0 $.
 Moreover, Lemma \ref{general compact} applied to $L_{n,r}\left[\gamma(u_j)*\phi_j\right](x)$, \\ $L_{n,D}\left[\gamma(u_j)*\phi_j\right](x)$, and $L_{n,D-\max(\lambda)}\left[\left(\gamma(u_j)*\phi_j\right)u_j\right](x)$ implies the existence of some $h_0(M_0,\beta)>0$ such that 
 \begin{multline}\label{grey.estimate}
    C \Arrowvert\gamma\Arrowvert_\infty^2 K^2 \sum_{m=0}^{M_0} (M_0+1)\left(\frac{m}{m!}\right)^2\left(\Arrowvert \gamma\Arrowvert_\infty D\right)^{2(m-1)} \int_{\Pi_D^3} dx \int_0^Ddr \\\times \fint_\stw dn \left|L_{n,r}\left(\gamma(u_j)*\phi_j\right)(x)-L_{n,r}\left(\gamma(u_j)*\phi_j\right)(x+h)\right|^2\\
	+C  \int_{\Pi_D^3} dx \fint_\stw dn\sum_{m=0}^{M_0} \frac{M_0+1}{(m!)^2}\Arrowvert\gamma\Arrowvert_\infty^{2m}D^m\int_0^D d\lambda_1...\int_0^D d\lambda_m\left|\int_{\max_{0\leq i\leq M_0}(\lambda_i)}^Ddr\right.\\\times \big[ \left(\left(\gamma(u_j)*\phi_j\right)u_j\right)(x-rn)-\left(\left(\gamma(u_j)*\phi_j\right)u_j\right)(x+h-rn)\big]\bigg|^2\\
	+C(G) \sum_{m=0}^{M_0} (M_0+1) \left(\frac{m}{m!}\right)^2\left(\Arrowvert \gamma\Arrowvert_\infty D\right)^{2(m-1)} \int_{\Pi_D^3} \fint_\stw dn \left|L_{n,D}\left(\gamma(u_j)*\phi_j\right)(x)\right.\\\left.-L_{n,D}\left(\gamma(u_j)*\phi_j\right)(x+h)\right|^2<\frac{\beta}{4}. 
 \end{multline}
 This is true because we are applying Lemma \ref{general compact} finitely many times, since the sum is finite. Moreover, the equi-integrability result of Lemma \ref{general compact} is uniform with respect to the length of the line along which we are integrating. It applies hence to all terms appearing in \eqref{grey.estimate}.
 Taking now in \eqref{estimate1} $$ J_0=\frac{4C(D,G)\Arrowvert\gamma\Arrowvert_\infty^2}{\beta}\;\; \text{ and }\;\; h_1<\frac{\sqrt{\beta}}{2K\sqrt{C(G,\Omega)}}, $$ we obtain 
	\begin{equation*}
	\int_\Omega dx \left|\mathcal{B}^{j}(u_{j})(x)-\mathcal{B}^{j}(u_{j})(x+h)\right|^2<\beta
	\end{equation*}
	for all $ j\geq J_0  $ and for all $ |h|<\min\left(h_0,h_1\right) $. The continuity of the functions $ u_{j} $ and the fact that for $ j<J_0 $ we have only finitely many elements of the sequence imply the existence of some $ 0<H_0\leq \min\left(h_0,h_1\right)$ such that
	\begin{equation*}
	\int_\Omega dx \left|\mathcal{B}^{j}(u_{j})(x)-\mathcal{B}^{j}(u_{j})(x+h)\right|^2<\beta
	\end{equation*}
	for all $ j\geq 0 $ and all $ |h|<H_0 $. Hence, the sequence  $ \left(\mathcal{B}^{j}(u_{j})\right)_{j\in\mathbb{N}} $ is compact in $ L^2 $ and there exists a subsequence $ \left(\mathcal{B}^{j_l}(u_{j_l})\right)_{l\in\mathbb{N}} $ and a function $ u\in L^2(\Omega)\cap L^\infty(\Omega) $ such that $ u_{j_l}=\mathcal{B}^{j_l}(u_{j_l})\to u $ both in $ L^2 $ and pointwise almost everywhere as $ l\to \infty $.   
	
	The uniformly boundedness of $u_{j_l}$ and also of $\gamma\left(u_{j_l}\right)$ implies the convergence in $L^p$ for $p<\infty$ of $\gamma\left(u_{j_l}\right)*\phi_{j_l}\to \gamma(u)$ as $l\to \infty$ and hence for a subsequence (say still $u_{j_l}$) the convergence holds also pointwise almost everywhere. Indeed, 
 \begin{multline*}
     \Arrowvert \gamma\left(u_{j_l}\right)*\phi_{j_l}-\gamma(u)\Arrowvert_p\\\leq\Arrowvert \gamma\left(u_{j_l}\right)*\phi_{j_l}-\gamma(u)*\phi_{j_l}\Arrowvert_p+\Arrowvert \gamma\left(u\right)*\phi_{j_l}-\gamma(u)\Arrowvert_p\\\leq  \Arrowvert \gamma\left(u_{j_l}\right)-\gamma(u)\Arrowvert_p \Arrowvert \phi_{j_l}\Arrowvert_1+ \Arrowvert \gamma\left(u\right)*\phi_{j_l}-\gamma(u)\Arrowvert_p\to 0 \;\;\text{ uniformly in }l,
 \end{multline*}
 where we used the Young's convolution inequality combined with the fact that $\phi_{j_l}$ are positive and with the dominated convergence. Finally another application of the dominated convergence theorem implies
	\begin{equation}
	u_{j_l}=\mathcal{B}^{j_l}(u_{j_l})\to u=\mathcal{B}(u) 
	\end{equation}
	pointwise almost everywhere as $l\to\infty$ and  $ u=\mathcal{B}(u) $ pointwise a.e. Hence, $u$ is the desired solution to \eqref{nablaFreducedu}.
\end{proof}
A direct corollary of the proof of Theorem \ref{thm full equation grey} is the following.
\begin{corollary}\label{corollary.cpt}
Let $ \{\varphi_j\}_{j\in\mathbb{N}} $ and $ \{\psi_j\}_{j\in\mathbb{N}} $ be two bounded sequences in $ L^\infty(\Omega) $ for $ \Omega\subset \rth $  bounded with $ C^1 $-boundary and strictly positive curvature. Let also $ f\in L^\infty(\stw) $ be non-negative. Then the sequences $$ \int_\stw dn \int_0^D dr \varphi_j(x-rn)\exp\left(-\int_0^r \psi_j(x-\lambda n)d\lambda\right)$$ and $$\int_\stw dn\;f(n)\exp\left(-\int_0^D \psi_j(x-\lambda n)d\lambda\right)$$ are compact in $ L^2(\Omega) $. In particular they are $L^2$-equiintegrable in the following way: For any $\eps>0$ there exists some $h_0>0$ such that for all $j\in\mathbb{N}$ and all $|h|<h_0$ both estimates holds
\begin{multline*}
    \int_\Omega dx \left|\int_\stw dn \int_0^D dr \left[\varphi_j(x-rn)\exp\left(-\int_0^r \psi_j(x-\lambda n)d\lambda\right)\right.\right.\\\left.\left.-\varphi_j(x+h-rn)\exp\left(-\int_0^r \psi_j(x+h-\lambda n)d\lambda\right)\right]\right|^2\\
    \leq 4\pi  \int_\Omega dx \int_\stw dn \left|\int_0^D dr \left[\varphi_j(x-rn)\exp\left(-\int_0^r \psi_j(x-\lambda n)d\lambda\right)\right.\right.\\\left.\left.-\varphi_j(x+h-rn)\exp\left(-\int_0^r \psi_j(x+h-\lambda n)d\lambda\right)\right]\right|^2    < \eps
\end{multline*}
and 
\begin{multline*}
    \int_\Omega dx \left|\int_\stw dn f(n) \left[\exp\left(-\int_0^D \psi_j(x-\lambda n)d\lambda\right)\right.\right.\\\left.\left.-\exp\left(-\int_0^D \psi_j(x+h-\lambda n)d\lambda\right)\right]\right|^2\\
    \leq 4\pi \Arrowvert f\Arrowvert_{L^{\infty}} \int_\Omega dx \int_\stw dn f(n)\left|\exp\left(-\int_0^D \psi_j(x-\lambda n)d\lambda\right)\right.\\\left.-\exp\left(-\int_0^
D \psi_j(x+h-\lambda n)d\lambda\right)\right|^2    < \eps.
\end{multline*}
\end{corollary}
Combining this result with Corollary \ref{corollary.cpt1} we see that the following proposition holds.
\begin{proposition}\label{prop.cpt}
	Let $ \Omega\subset \rth $  bounded with $ C^1 $-boundary and strictly positive curvature. Let also $ \{\varphi_j(x,\omega)\}_{j\in\mathbb{N}}\subset C\left(\stw,L^\infty\left(\Omega\right)\right)  $ be uniformly bounded and satisfying the assumption of Corollary \ref{corollary.cpt1}; i.e., $$ \Arrowvert \vphi_j(\cdot,n_1)-\vphi_j(\cdot,n_2)\Arrowvert_{L^\infty(\Omega)}\leq \sigma(d(n_1,n_2))\to 0,$$ uniformly in $ j\in\mathbb{N} $ if $d(n_1,n_2)\to 0 $, where $ d $ is the metric on the sphere and $\sigma\in C\left(\mathbb{R}_+,\mathbb{R}_+\right)$ with $\sigma(0)=0$ is a uniform modulus of continuity.. Let $ \{\psi_j(x)\}_{j\in\mathbb{N}}\subset L^\infty\left(\Omega\right) $ be a uniform bounded sequence. Then the new sequence
	$$ \int_\stw dn \int_0^D dr \varphi_j(x-rn,n)\exp\left(-\int_0^r \psi_j(x-\lambda n)d\lambda\right)$$ is compact in $ L^2(\Omega) $.
	\begin{proof}
	We apply Federer-Besicovitch covering Lemma for the sphere $ \stw $ as we did in Corollary \ref{corollary.cpt1} to the result of Corollary \ref{corollary.cpt}.
	\end{proof}
\end{proposition}
This completes the proof of the existence of solutions to the case of Grey approximation. Then we can further use this result to prove our main theorem (Theorem \ref{main theorem pseudo grey}) for the pseudo Grey case as follows.

   \begin{proof}[Proof of Theorem \ref{main theorem pseudo grey}]
       We expect for the pseudo Grey case the same result to hold. Let $$\alpha_\nu(T(x))=Q(\nu) h(T(x)),$$ for some non-negative bounded function $Q(\nu)$. Then define $$u(x)=4\pi\int_0^\infty d\nu\ Q(\nu) B_\nu(T(x))=F(T(x)).$$ 
We notice that by the monotonicity in of $B_\nu(T)$ in $T$ also $F$ is monotone with respect to $T$ and hence $T(x)=F^{-1}(u)(x)$. Denoting by $\gamma=h(F^{-1})$ and by $f_\nu=B_\nu(F^{-1})$ we see that $u$ solves
        \begin{multline}\label{nablaFreducedu1}
  u(x) =\int_0^\infty d\nu\int_\Omega d\eta \ Q(\nu)^2\gamma(u(\eta))f_\nu(u(\eta) )\exp\left(-\int_{[x,\eta]}Q(\nu)\gamma(u\left(\zeta\right))d\zeta\right)\\+\int_0^\infty d\nu\int_{\mathbb{S}^2}dn \ \exp\left(-\int_{[x,y(x,n)]}Q(\nu)\gamma(u(\zeta))d\zeta\right)Q(\nu) g_\nu(n).
  \end{multline}
  We regularize this equation in the same way as in Section \ref{sec.regular} and \eqref{regular} and obtain for $\phi_\eps$ a standard positive and symmetric mollifier the following fixed-point equation.
  \begin{multline}\label{pseudo1}
      u(x)=\mathcal{B}_\eps(u)(x)=\int_0^\infty d\nu\int_\Omega d\eta \ Q(\nu)^2\left(\gamma(u)*\phi_\eps\right)(\eta)f_\nu(u(\eta) ) \\\times\exp\left(-\int_{[x,\eta]}Q(\nu)\left(\gamma(u)*\phi_\eps\right)(\zeta)d\zeta\right)\\+\int_0^\infty d\nu\int_{\mathbb{S}^2}dn \ \exp\left(-\int_{[x,y(x,n)]}Q(\nu)\left(\gamma(u)*\phi_\eps\right)(\zeta)d\zeta\right)Q(\nu) g_\nu(n).
  \end{multline}Then the $L^\infty$-estimate and the equicontinuity (or more precisely uniform H\"older continuity) of the right-hand side of \eqref{pseudo1} hold once more in the same way as in Subsection \ref{sec.grey proof}. Also, $\mathcal{B}_\eps$ is a continuous operator. 
  We hence have solutions $u_\eps$. We need to show the compactness of the sequence of regularized solutions $u_j$, where $\eps=\frac{1}{j}$ . We consider similarly as in the proof of Theorem \ref{Grey thm} the line operators acting on some suitable sequence given for $r\in[0,D]$ by
  \begin{equation*}
      Q(\nu)L_{n,r}\left(\gamma(u_j)*\phi_j\right)(x) \;\;\text{ and }\;\;L_{n,D-r}\left(\gamma(u_j)*\phi_j\int_0^\infty Q(\nu)B_\nu\left(F^{-1}(u)\right)d\nu\right)(x).
  \end{equation*}
  By the definition of $u$ obtain the following uniform estimate
  $$\sup_{x\in\Omega}\left|\gamma(u_j)*\phi_j(x)\int_0^\infty Q(\nu)B_\nu\left(F^{-1}(u_j)\right)(x)d\nu\right|\leq \Arrowvert\gamma\Arrowvert_\infty \Arrowvert u\Arrowvert_\infty,$$
  where $\Arrowvert u\Arrowvert_\infty$ is the uniform upper bound of $u_j$.
  Hence, we can write as before the operator $\mathcal{B}_j$ in polar coordinates according to
  \begin{multline*}
  \mathcal{B}_j(u_j)(x)=\int_0^\infty d\nu Q(\nu)\int_\stw dn\int_0^{s(x,n)} dr  \left[\left(\gamma(u_j)*\phi_j\right)Q(\nu)f_\nu(u_j)\right](x-rn) \\\times\exp\left(-Q(\nu)L_{n,r}\left(\gamma(u_j)*\phi_j\right)(x)\right)\\+\int_0^\infty d\nu\int_{\mathbb{S}^2}dn \ \exp\left(-Q(\nu)L_{n,s(x,n)}\left(\gamma(u_j)*\phi_j\right)(x)\right)Q(\nu) g_\nu(n).    
  \end{multline*}
  Thus, using the boundedness of $Q$ and the estimate $\int_0^\infty Q^{m+1}(\nu)f_\nu(x)\leq \Arrowvert Q\Arrowvert_\infty^m \Arrowvert u\Arrowvert_\infty$ similarly as we did for equation \eqref{estimate1} we can obtain
\begin{multline}\label{estimate2}
	\int_\Omega dx \left|\mathcal{B}_j(u_j)(x)-\mathcal{B}_j(u_j)(x+h)\right|^2\\
	\leq C(D,\Omega,g)\Arrowvert\gamma\Arrowvert^2_\infty \Arrowvert Q\Arrowvert^2_\infty\Arrowvert u_j\Arrowvert^2_\infty \sup_{\substack{\nu\geq0\\n\in\stw\\0\leq r\leq D}}\left\Arrowvert \sum_{m>M}^{\infty} \frac{(-1)^m}{m!}Q^m(\nu)\left|L_{n,r}\left(\gamma(u_j)*\phi_j\right)(\cdot)\right|^{m}\right\Arrowvert^2_\infty\\
	+ C \Arrowvert\gamma\Arrowvert_\infty^2\Arrowvert u_j\Arrowvert_\infty^2 \sum_{m=0}^M (M+1)\left(\frac{m}{m!}\right)^2\left(\Arrowvert \gamma\Arrowvert_\infty D\right)^{2(m-1)}\Arrowvert Q\Arrowvert_\infty^{2m+2} \int_{\Pi_D^3} dx \int_0^Ddr \\\times \int_\stw dn \left|L_{n,r}\left(\gamma(u_j)*\phi_j\right)(x)-L_{n,r}\left(\gamma(u_j)*\phi_j\right)(x+h)\right|^2\\
	+C  \int_{\Pi_D^3} dx \int_\stw dn\sum_{m=0}^{M} \frac{M+1}{(m!)^2}\Arrowvert\gamma\Arrowvert_\infty^{2m}D^m\int_0^D d\lambda_1...\int_0^D d\lambda_m\left|\int_{\max_{0\leq i\leq M}(\lambda_i)}^Ddr\int_0^\infty d\nu\right.\\\times \big[ \left(\left(\gamma(u_j)*\phi_j\right)Q^{m+2}(\nu)f_\nu(u_j)\right)(x-rn)-\left(\left(\gamma(u_j)*\phi_j\right)Q^{m+2}(\nu)f_\nu(u_j)\right)(x+h-rn)\big]\bigg|^2\\
	+C(g_\nu) \sum_{m=0}^M (M+1) \left(\frac{m}{m!}\right)^2\left(\Arrowvert \gamma\Arrowvert_\infty D\right)^{2(m-1)} \Arrowvert Q\Arrowvert_\infty^{2m+2}\int_{\Pi_D^3} \int_\stw dn\bigg|L_{n,D}\left(\gamma(u_j)*\phi_j\right)(x)\\-L_{n,D}\left(\gamma(u_j)*\phi_j\right)(x+h)\bigg|^2+C(g_\nu,Q,\Omega) \left(\Arrowvert\gamma\Arrowvert_\infty^2\frac{1}{j}+\Arrowvert u\Arrowvert_\infty^2|h|^2\right),
	\end{multline}
where we used the triangle inequality as we did in \eqref{estimate0}. In addition, for the tails of the exponential terms in the estimate, we use the supremum norm and use that $\int_0^\infty Q(\nu)f_\nu(u_j)=u_j$. For the terms involving the finite difference of powers of line integrals we argue as we did in \eqref{estimate1} taking the absolute value inside the integrals, estimating each term using the boundedness of $Q(\nu)$ and the integrability of $f_\nu$ and applying Jensen's inequality in the end. The term in the fifth line of \eqref{estimate2} is obtained using the identity \eqref{fubini} given by Fubini's theorem and changing the order of integration so that the integral with respect to $\nu$ is the most interior one. Hence, we conclude with Jensen's inequality. The last term in \eqref{estimate2} is obtained exactly as the last term in \eqref{estimate1}.
 We conclude the compactness of the sequence $u_j=\mathcal{B}_j(u_j)$ in $L^2$ as we did in the proof of Theorem \ref{Grey thm}. 
 We hence fix first of all the $M_0>0$ such that the first term in the right hand side of \eqref{estimate2} is smaller than $\frac{\beta}{5}$ for an arbitrarily small $\beta>0$. This is possible because $$\sup_{\substack{\nu\geq0,\ n\in\stw\\0\leq r\leq D,\ x\in\Omega}}\left|Q(\nu)L_{n,r}\left(\gamma(u_j)*\phi_j\right)(x)\right|\leq D\Arrowvert Q\Arrowvert_\infty\Arrowvert\gamma\Arrowvert_\infty.$$ After that, since the sequence $\gamma(u_j)*\phi_j(x)\int_0^\infty Q^{m+2}(\nu)B_\nu\left(F^{-1}(u_j)\right)(x)d\nu$ is uniformly bounded for all $m\leq M_0+1$, all arguments in the proof of Theorem \ref{Grey thm} still apply and hence the line itegral of this sequence is also equi-integrable. Hence, arguing in the same way as in the proof of Theorem \ref{Grey thm} we see that a subsequence $u_{j_l}$ converges pointwise almost everywhere to the desired solution $u=\mathcal{B}(u)$.
   \end{proof}

\section{Full equation with both scattering and emission-absorption}\label{sec.full}

In this section we consider the full equation with both scattering and emission-absorption terms. We study the case when the scattering coefficient and the absorption coefficient depend on the local temperature $T(x)$. The radiative transfer equation can be written as
\begin{multline}\label{full rte}
    n\cdot \nabla_x I_\nu(x,n)=\alpha_\nu^a\left(T(x)\right)\left(B_\nu\left(T(x)\right)-I_\nu(x,n)\right)\\+\alpha_\nu^s\left(T(x)\right)\left[\left(\int_{\stw}dn'\; K(n,n')I_\nu(x,n')\right)-I_\nu(x,n)\right].
\end{multline}
We consider as in the previous sections equation \eqref{full rte} coupled with the condition of divergence-free total flux in equation \eqref{heat equation} and the incoming boundary condition \eqref{incoming boundary}. We will consider in this paper only the case of isotropic scattering, i.e. the case where the scattering kernel is invariant under rotation.

We notice first of all that the isotropic property of the scattering kernel implies its symmetry.
\begin{lemma}\label{symmetry K}
Let $K(n,n')$ be rotation invariant, i.e. $K(n,n')=K(Rn,Rn')$ for all $n,n'\in\stw$ and $R\in SO(3)$. Then $K(n,n')=K(n',n)$.
\begin{proof}
    Let $n,n'\in\stw$. We denote by $\theta\in[0,\pi]$ the angle formed by $n,n'$ on the plane spanned by these unit vectors. We denote moreover by $R_\theta\in SO(3)$ the rotation matrix defined by a rotation of $\pi$ around the bisectrix of $\theta$ lying in the plane spanned by $n$ and $n'$. Then we see that $R_\theta n=n'$ and $R_\theta n'=n$. Hence by assumption, $K(n,n')=K\left(R_\theta n',R_\theta n\right)=K(n',n)$. 
\end{proof}
\end{lemma}
\subsection{Main result in the case of the Grey approximation}\label{Sec.full.grey}
We consider first the case of Grey approximation, i.e. we assume the coefficients and the scattering kernel to be independent of the frequency and we denote $\alpha_\nu^a=\alpha^a$ and $\alpha_\nu^s=\alpha^s$.
We will now prove a theorem about the existence of solutions to \eqref{full rte} similar to Theorem \ref{Grey thm}. The main difference with the setting of Theorem \ref{Grey thm} is the presence of the scattering operator. In the pure emission-absorption case the motion of the photons between one emission and the next absorption is rectilinear. On the contrary, in the presence of scattering, the photons move along a polygonal path between emission and absorption events. In order to take it into account we will define suitable Green functions that incorporate the polygonal motion due to the scattering. Using these Green functions it will be possible to find a fixed-point equation for the temperature analogous to \eqref{nablaFreduced} that includes the non-rectilinear motion between emission and absorption events (cf. equation \eqref{fixedpointfull}).
We show the following theorem.
\begin{theorem}\label{thm full equation grey} Let $\Omega\subset \rth$ be bounded, convex and open with $C^1$-boundary and strictly positive curvature. 
Let $ \alpha^a $ and $ \alpha^s $ be positive and bounded $C^1$-functions of the temperature, independent of the frequency. Assume $ K\in C^1\left(\stw\times\stw\right) $ be non-negative, rotationally symmetric and independent of the frequency with the property \eqref{scattering assume}. 
Then there exists a solution $ (T,I_\nu) \in L^\infty(\Omega)\times L^\infty\left(\Omega, L^\infty\left(\stw,L^1(\mathbb{R}_+)\right)\right)$ to the equation \eqref{full rte} coupled with \eqref{heat equation} satisfying the boundary condition \eqref{incoming boundary}, where the $I_\nu$ is a solution to \eqref{full rte} in the sense of distribution. 
\end{theorem}

For the proof we proceed in the following way. As indicated above we begin constructing a fixed-point equation for the temperature which contains information about the scattering processes. We will hence regularize the problem, similarly as we did in Section \ref{sec.regular} and will prove the existence of regularized solution using the Schauder fixed-point theorem. At the end we will use the compactness theory developed in Subsection \ref{sec.compactness.defect} in order to show the convergence of a subsequence of the regularized solutions to the desired solution. 

We define for $ x,x_0\in\Omega $ and $ n \in\stw$ the fundamental solution $ \tI(x,n;x_0) $ solving the following equation in distributional sense
\begin{multline}\label{eq.green}
n\cdot \nabla_x\tI(x,n;x_0)=\alpha^s(T(x))\int_\stw K(n,n')\tI(x,n';x_0)\;dn'\\-\left(\alpha^a\left(T(x)\right)+\alpha^s\left(T(x)\right)\right)\tI(x,n;x_0)+\delta(x-x_0)
\end{multline}
and the boundary condition for $ x\in\bnd $
\begin{equation*}
\tI(x,n;x_0)\chi_{\{n\cdot n_x<0\}}=0,
\end{equation*}
where $ n_x $ denotes the normal outer vector to $ \bnd $ at $ x $. Similar to the Poisson kernel for the Laplace equation, for $ x\in\Omega $, $ x_0\in\bnd $ and $ n,n_0 \in\stw$ we define the function $ \psi(x,n;x_0,n_0) $ by the equation
\begin{equation}\label{eq.poisson}
\begin{split}
n\cdot\nabla_x \psi(x,n;x_0,n_0)&= \alpha^s(T(x))\int_\stw K(n,n')\psi(x,n';x_0,n_0)\;dn'\\
&\ \ \ \ -\left(\alpha^a\left(T(x)\right)+\alpha^s\left(T(x)\right)\right)\psi(x,n;x_0,n_0),\  x\in\Omega,\\
\psi(x,n;x_0,n_0)\chi_{\{n\cdot n_x<0\}}&=\delta_{\bnd}(x-x_0)\frac{\delta^{(2)}(n,n_0)}{4\pi},\  x\in\Omega, \ n_0\cdot N_{x_0}<0,
\end{split}
\end{equation}
where we denoted by $ \delta^{(2)} $ the two dimensional delta distribution on the sphere and by $ \delta_{\bnd} $ the two dimensional delta distribution on $ \bnd $. This allows to include the effect of the boundary. Before moving to the computations of such functions we see that the intensity of radiation can be expressed by these two functions as follows.
\begin{multline}\label{I.green.poisson}
	I_\nu(x,n)=\int_\Omega dx_0\; \alpha^a(T(x_0))B_\nu(T(x_0))\tI(x,n;x_0)\\ +\int_\stw dn_0\int_{\bnd}dx_0\; g_\nu(n_0)\psi(x,n;x_0,n_0).
\end{multline}
Thus, plugging \eqref{I.green.poisson} into \eqref{heat equation} and using \eqref{integration of blackbody} we obtain
\begin{multline*}
0=\nabla_x\cdot \mathcal{F}(x)= \Div\left(\int_0^\infty d\nu\int_\stw dn\int_\Omega dx_0\; \alpha^a(T(x_0))B_\nu(T(x_0))n\tI(x,n;x_0)\right.\\
+\left.\int_0^\infty d\nu\int_\stw dn\int_\stw dn_0\int_{\bnd}dx_0\; g_\nu(n_0)n\psi(x,n;x_0,n_0)\right)\\
=\sigma \int_\stw dn\int_\Omega dx_0 \; \alpha^a(T(x_0))T^4(x_0)\left[\delta(x-x_0)-\alpha^a(T(x))\tI(x,n;x_0)\right]\\
+\sigma \int_\stw dn\int_\Omega dx_0 \; \alpha^a(T(x_0))T^4(x_0)\alpha^s(T(x))\left[\int_\stw dn' K(n,n')\tI(x,n';x_0)-\tI(x,n;x_0)\right]\\
+\int_\stw dn\int_\stw dn_0\int_{\bnd}dx_0 \ G(n_0)\left[\alpha^s(T(x))\int_\stw K(n,n')\psi(x,n';x_0,n_0)	\;dn'\right.\\-(\alpha^s(T(x))+\alpha^a(T(x)))\psi(x,n;x_0,n_0) \bigg]\\
=4\pi\sigma \alpha^a(T(x)) T^4(x)-\alpha^a(T(x))\sigma\int_\stw dn\int_\Omega dx_0\; \alpha^a(T(x_0))T^4(x_0)\tI(x,n;x_0)\\
-\alpha^a(T(x))\int_\stw dn\int_\stw dn_0\int_{\bnd}dx_0\;G(n_0)\psi(x,n;x_0,n_0),
\end{multline*}
where we defined $G(n)\eqdef \int_0^\infty d\nu \ g_\nu(n),$ and the last equality holds by the property \eqref{scattering assume} of the kernel $ K $ integrating first with respect to $ n $. Hence, defining $ u(x)=4\pi\sigma T^4(x) $ and dividing by $ \alpha^a(T(x)) $ we get the following non-linear fixed-point equation
\begin{multline}\label{fixedpointfull}
u(x)=\int_\Omega dx_0\int_\stw dn\;\frac{\alpha^a(u(x_0))u(x_0)}{4\pi}\tI(x,n;x_0)\\
+\int_\stw dn_0\int_\stw dn\int_{\bnd}dx_0\;G(n_0)\psi(x,n;x_0,n_0),
\end{multline}
where by an abuse of notation we define $ \alpha^a(\cdot)=\alpha^a\left(\sqrt[4]{\frac{\cdot}{4\pi\sigma}}\right)$. 
\subsection{Construction of the Green functions in the Grey case}\label{Sec.full.green}
Let us now construct the Green functions $ \tI $ and $ \psi $. We start with the first function. Denoting by $ H(\cdot) $ the Heaviside function and by $ P_n^\perp $ the projection $ P_n^\perp=I-n\otimes n $, we see using the Fourier transform that the distribution $ f_0(x,n;x_0)=H(n\cdot(x-x_0))\delta^{(2)}\left(P_n^\perp(x-x_0)\right) $ solves in distributional sense the equation
\begin{equation*}
n\cdot \nabla_xf_0(x,n;x_0)=\delta(x-x_0)
\end{equation*}
with zero boundary condition. Hence, the function $ f_1(x,n;x_0)=f_0(x,n;x_0)+\int_\rth dy\ F(y)f_0(x,n;y) $ solves in distributional sense the equation
\begin{equation*}
n\cdot\nabla_x f_1(x,n;x_0)=F(x)+\delta(x-x_0).
\end{equation*}
Notice that by definition of $f_0$ we have $\int_\rth dy\ F(y)f_0(x,n;y)=\int_0^{s(x,n)} dt F(x-tn)$.
Moreover, the function $f_2(x,n;x_0)=f_0(x,n;x_0)\exp\left(-\int_{[x_0,x]} \alpha(\xi)ds(\xi)\right)$ solves in distributional sense the equation \begin{equation*}
n\cdot\nabla_x f_2(x,n;x_0)=\delta(x-x_0)-\alpha(x)f_2(x,n;x_0),
\end{equation*}
since $\frac{x-x_0}{|x-x_0|}\cdot \nabla_x\int_{[x_0,x]} \alpha(\xi)ds(\xi)=\alpha(x)$. Hence, we conclude that \begin{multline*}
    f(x,n;x_0)=f_0(x,n;x_0)\exp\left(-\int_{[x_0,x]} \alpha(\xi)ds(\xi)\right)\\+\int_\rth dy\ F(y)f_0(x,n;y)\exp\left(-\int_{[y,x]} \alpha(\xi)ds(\xi)\right),
\end{multline*} solves in distributional sense the equation
\begin{equation*}
n\cdot\nabla_x f(x,n;x_0)=\delta(x-x_0)-\alpha(x)f(x,n;x_0)+ F(x).
\end{equation*}
Thus, with these considerations we write the Green function $ \tI $ as
\begin{multline}\label{green exact}
\tI(x,n;x_0)\\=\chi_\Omega(x_0)\exp\left(-\int_{[x_0,x]}\left[\alpha^a(u(\xi))+\alpha^s(u(\xi))\right]d\xi\right)H(n\cdot(x-x_0))\delta^{(2)}\left(P_n^\perp(x-x_0)\right)\\
+\int_0^{s(x,n)}dt\; \alpha^s(u(x-tn))\exp\left(-\int_{[x-tn,x]}\left[\alpha^a(u(\xi))+\alpha^s(u(\xi))\right]d\xi\right)\\\times\int_\stw dn'K(n,n')\tI(x-tn,n';x_0).
\end{multline}
This is a recursive formula. After having regularized it we will write down the Duhamel series for this Green function.

Similarly, we can construct the function $\psi$. We notice first of all that for $ x_0\in\bnd $ the distribution $ \psi $ solving the equation \eqref{eq.poisson} is a solution to the equation
\begin{multline*}
n\cdot \nabla_x W(x,n;x_0,n_0)=\alpha^s(T(x))\int_\stw K(n,n')W(x,n';x_0,n_0)\;dn'\\-\left(\alpha^a\left(T(x)\right)+\alpha^s\left(T(x)\right)\right)W(x,n;x_0,n_0)+\delta(x-x_0)\frac{\delta^{(2)}(n,n_0)}{4\pi}.
\end{multline*}
As we have computed above for $ f_0 $, as $ x $ approaches to $ x_0 $ the leading term of the distribution $ W(x,n;x_0,n_0) $ is given by
\begin{equation*}
	W(x,n;x_0,n_0)\simeq H(n_0\cdot(x-x_0))\delta^{(2)}\left(P_{n_0}^\perp (x-x_0)\right)\frac{\delta^{(2)}(n,n_0)}{4\pi}.
\end{equation*}
As for the Poisson Kernel in the case of the Poisson equation, we expect $ W $ to differ from $ \psi $ only for a Jacobian as $ x\to\overline{x}\in\bnd $ with $ n\cdot N_{\overline{x}}<0 $ and $ n_0\cdot N_{x_0}<0 $. We compute now the Jacobian. Hence, we consider $ \varphi\in C_c^{\infty}(\bnd) $ with supp$ (\varphi)\subset B^{\bnd}_\eps(\overline{x}) $. We assume without loss of generality $ \overline{x}=0 $, $ N_{\overline{x}}=-e_1 $ and $ n\cdot e_3=0 $. We compute 
\begin{multline}\label{poisson near bnd 1}
\int_\stw dn_0 \int_{\bnd} dx_0 \;\varphi(x_0)H(n_0\cdot(x-x_0))\delta^{(2)}\left(P_{n_0}^\perp (x-x_0)\right)\frac{\delta^{(2)}(n,n_0)}{4\pi}\\
=\int_{\bnd} dx_0 \; \varphi(x_0)H(n\cdot(x-x_0))\delta^{(2)}\left(P_{n}^\perp (x-x_0)\right).
\end{multline}
For $ \eps>0 $ small enough we can approximate $ \bnd\cap B^{\bnd}_\eps(0)$ by $ \rth_+\cap B_\eps(0) $ and we define for $ x_0^1>0 $ the constant extension $ \overline{\varphi}(x_0^1,x_0^2,x_0^3)=\varphi(x_0^2,x_0^3) $. Moreover we see that in the rotated coordinate system given by $ y_1\parallel n $ and $ y_3\parallel e_3 $ we have 
\begin{multline}\label{delta identity}
\delta^{(2)}(P_n^\perp\left(\overline{x}-x_0)\right)=\delta(y_2)\delta(y_3)\\=\delta\left(-\sqrt{1-\left|n\cdot N_{\overline{x}}\right|^2}\left(\overline{x}^1-x_0^1\right)+\left|n\cdot N_{\overline{x}}\right|\left(\overline{x}^2-x_0^2\right)\right)\delta\left(x^3-x_0^3\right)\\
=\frac{1}{\left|n\cdot N_{\overline{x}}\right|}\delta\left(-\tan(\theta)\left(\overline{x}^1-x_0^1\right)+\left(\overline{x}^2-x_0^2\right)\right)\delta\left(x^3-x_0^3\right),
\end{multline}
where $ \theta $ is the angle between $ n $ and $ e_1 $, hence $ \left|n\cdot N_{\overline{x}}\right|=\cos(\theta) $. Since $ x\to\overline{x} $, we conclude our computation putting \eqref{delta identity} into \eqref{poisson near bnd 1} and thus
\begin{multline}\label{poisson near bnd 2}
\int_\stw dn_0 \int_{\bnd} dx_0 \;\varphi(x_0)H(n_0\cdot(x-x_0))\delta^{(2)}\left(P_{n_0}^\perp (x-x_0)\right)\frac{\delta^{(2)}(n,n_0)}{4\pi}\\
=\int_{\rth_+} dx_0 \;\delta(\overline{x}^1-x_0^1) \frac{1}{\left|n\cdot N_{\overline{x}}\right|} \overline{\varphi}(x_0)\delta\left(-\tan(\theta)\left(\overline{x}^1-x_0^1\right)+\left(\overline{x}^2-x_0^2\right)\right)\delta\left(x^3-x_0^3\right)\\
=\frac{\varphi(\overline{x})}{\left|n\cdot N_{\overline{x}}\right|}.
\end{multline}
We have just proved that in distributional sense we have 
\begin{multline*}
 H(n_0\cdot(x-x_0))\delta^{(2)}\left(P_{n_0}^\perp (x-x_0)\right)\frac{\delta^{(2)}(n,n_0)}{4\pi}\underset{x\to \overline{x}}{\longrightarrow}\frac{\delta_{\bnd}(\overline{x}-x_0)}{|n\cdot N_{\overline{x}}|}\frac{\delta^{(2)}(n,n_0)}{4\pi}\\
 \overset{\mathcal{D}'}{=}\frac{\delta_{\bnd}(\overline{x}-x_0)}{|n_0\cdot N_{x_0}|}\frac{\delta^{(2)}(n,n_0)}{4\pi}. 
\end{multline*}
Hence, as $ x\to\overline{x} $ the distribution $ \psi $ is given at the leading order by
\begin{equation}\label{bndpoisson}
|n_0\cdot N_{x_0}|H(n_0\cdot(x-x_0))\delta^{(2)}\left(P_{n_0}^\perp (x-x_0)\right)\frac{\delta^{(2)}(n,n_0)}{4\pi} .
\end{equation}
We note also, that $ P_n^\perp(x-x_0) $ is not non-trivial only in a neighborhood of $ y(x,n)\in\bnd $ and $ y(x,-n)\in\bnd $. Moreover, $ H(n\cdot(x-y(x,-n)))=0 $ while $ H(n\cdot(x-y(x,n)))=1 $. Hence, with the same reasoning as in equations \eqref{poisson near bnd 1} and \eqref{poisson near bnd 2} we see that for any $ x\in\Omega $
\begin{multline}\label{poisson leading order}
\int_\stw dn_0 \int_{\bnd} dx_0 \;|n_0\cdot N_{x_0}|\varphi(x_0)H(n_0\cdot(x-x_0))\delta^{(2)}\left(P_{n_0}^\perp (x-x_0)\right)\frac{\delta^{(2)}(n,n_0)}{4\pi}\\
=\varphi(y(x,n)).
\end{multline}
We conclude the derivation of the Green function $\psi$ integrating by characteristics the equation \eqref{eq.poisson} with the boundary value given by \eqref{bndpoisson} and we obtain
\begin{multline}\label{poisson formula}
\psi(x,n;x_0,n_0)=|n_0\cdot N_{x_0}|H(n_0\cdot(x-x_0))\\\times\delta^{(2)}\left(P_{n_0}^\perp (x-x_0)\right)\frac{\delta^{(2)}(n,n_0)}{4\pi}\exp\left(-\int_{[x_0,x]}\left[\alpha^a(u(\xi))+\alpha^s(u(\xi))\right]d\xi\right)\\
+\int_0^{s(x,n)}dt\; \alpha^s(u(x-tn))\exp\left(-\int_{[x-tn,x]}\left[\alpha^a(u(\xi))+\alpha^s(u(\xi))\right]d\xi\right)\\\times\int_\stw dn'K(n,n')\psi(x-tn,n';x_0,n_0).
\end{multline}
\subsection{Regularized fixed-point equation in the Grey case}\label{Sec.full.reg}
We proceed now with the regularization of the fixed-point problem stated in \eqref{fixedpointfull}. Similarly as we did in Section \ref{sec.deriv.reg} we regularized the fixed-point equation mollifying with a standard positive and rotationally symmetric mollifier the absorption and scattering coefficients. For $ l\in\{a,s\} $ denote in order to simplify the notation $ \alpha^l(u)*\phi_\eps(x)=\alpha_\eps^l(x) $. Notice that $ \alpha_\eps^l(x) $ still depends on the temperature. We recall also that $$ \int_\rth f(\xi)\delta_{[x,y]}*\phi_\eps(\xi)\;d\xi =\int_{[x,y]} f*\phi_\eps (\xi)\;ds(\xi).$$ Hence, we define $ I_\eps(x,n;x_0) $ and $ \psi_\eps(x,n;x_0,n_0) $ solving the regularized equations
\begin{multline}\label{eq.green.reg}
n\cdot \nabla_x I_\eps(x,n;x_0)=\alpha_\eps^s(x)\int_\stw K(n,n')I_\eps(x,n';x_0)\;dn'\\-\left(\alpha_\eps^a\left(x\right)+\alpha_\eps^s\left(x\right)\right)I_\eps(x,n;x_0)+\delta(x-x_0),
\end{multline}
with zero incoming boundary conditions and 
\begin{equation}\label{eq.poisson.reg}
\begin{split}
n\cdot\nabla_x \psi_\eps(x,n;x_0,n_0)&= \alpha_\eps^s(x)\int_\stw K(n,n')\psi_\eps(x,n';x_0,n_0)\;dn'\\
&\ \ \ -\left(\alpha_\eps^a\left(x\right)+\alpha_\eps^a\left(x\right)\right)\psi_\eps(x,n;x_0,n_0),\  x\in\Omega\\
\psi_\eps(x,n;x_0,n_0)\chi_{\{n\cdot n_x<0\}}&=\delta_{\bnd}(x-x_0)\frac{\delta^{(2)}(n,n_0)}{4\pi},\  x\in\Omega,\  n_0\cdot N_{x_0}<0.
\end{split}
\end{equation}
Hence, the exact recursive formulas defining the regularized distributions are given by
\begin{multline}\label{green exact reg}
I_\eps(x,n;x_0)\\=\chi_\Omega(x_0)\exp\left(-\int_{[x_0,x]}\left[\alpha_\eps^a(\xi)+\alpha_\eps^s(\xi)\right]d\xi\right)H(n\cdot(x-x_0))\delta^{(2)}\left(P_n^\perp(x-x_0)\right)\\
+\int_0^{s(x,n)}dt\; \alpha_\eps^s(x-tn)\exp\left(-\int_{[x-tn,x]}\left[\alpha_\eps^a(\xi)+\alpha_\eps^s(\xi)\right]d\xi\right)\\\times\int_\stw dn'K(n,n')I_\eps(x-tn,n';x_0),
\end{multline}
and
\begin{multline}\label{poisson formula reg}
\psi_\eps(x,n;x_0,n_0)=|n_0\cdot N_{x_0}|H(n_0\cdot(x-x_0))\\\times\delta^{(2)}\left(P_{n_0}^\perp (x-x_0)\right)\frac{\delta^{(2)}(n,n_0)}{4\pi}\exp\left(-\int_{[x-tn,x]}\left[\alpha_\eps^a(\xi)+\alpha_\eps^s(\xi)\right]d\xi\right)\\
+\int_0^{s(x,n)}dt\; \alpha_\eps^s(x-tn)\exp\left(-\int_{[x-tn,x]}\left[\alpha_\eps^a(\xi)+\alpha_\eps^s(\xi)\right]d\xi\right)\\\times\int_\stw dn'K(n,n')\psi_\eps(x-tn,n';x_0,n_0).
\end{multline}
Next we show the existence of regularized solutions $ u_\eps$ to the equation
\begin{multline}\label{reg.fix.point}
u_\eps(x)=\int_\Omega dx_0\int_\stw dn\;\frac{\alpha^a_\eps(u_\eps(x_0))u_\eps(x_0)}{4\pi}I_\eps(x,n;x_0)\\
+\int_\stw dn_0\int_\stw dn\int_{\bnd}dx_0\;G(n_0)\psi_\eps(x,n;x_0,n_0)\\
\eqdef \mathcal{B}_\eps(u_\eps)(x)+\mathcal{C}_\eps(u_\eps)(x).
\end{multline}
We aim to use Schauder fixed-point theorem. We start showing that the operator mapping $ u\in\{C(\Omega): 0\leq u\leq M\} $ to the right-hand side of \eqref{reg.fix.point} is a self map taking $ M $ large enough. Similarly as in \eqref{u upper bnd} we will first show that $\mathcal{B}^\eps$ is a contractive operator. To this end we consider the function \begin{equation}\label{Heps}
    H_\eps(x,n)=\int_\Omega dx_0 \alpha^a_\eps(u(x_0))I_\eps(x,n;x_0)
\end{equation} and we will show by means of the weak maximum principle formulated in the next Subsection that $0\leq H_\eps(x,n)\leq\theta<1$, which will imply the contractivity as
\begin{equation}\label{schauder.estm.1}
\left| \mathcal{B}_\eps(u)(x)\right|\leq \Arrowvert u\Arrowvert_{\sup}\int_\Omega dx_0\int_\stw dn\;\frac{\alpha^a_\eps(u(x_0))}{4\pi}I_\eps(x,n;x_0)\leq \theta \Arrowvert u\Arrowvert_{\sup}.
\end{equation}
\subsection{Weak maximum principle}\label{Sec.full.weak.max}
In order to show that $\mathcal{B}^\eps$ is a contractive operator we consider first the function $H_\eps$ defined in \eqref{Heps}. Integrating with respect to $ x_0 $ the differential equation \eqref{eq.green.reg} satisfied by $ I_\eps $ we obtain the differential equation satisfied by the function $ H_\eps $ in  the sense of distribution:
\begin{multline}\label{eq.H.eps}
0=L_\eps(H_\eps)(x,n)-\alpha_\eps^a(x)\\=n\cdot \nabla_x H_\eps(x,n)-\alpha_\eps^a(x)\left(1-H_\eps(x,n)\right)\\-\alpha_\eps^s(x)\int_\stw dn'\;K(n,n')\left(H_\eps(x,n)-H_\eps(x,n')\right).
\end{multline}
With the following weak maximum principle we will show that $0\leq H_\eps(x,n)\leq 1$. To this end we consider the adjoint operator defined by \begin{multline}\label{adjoint.max.principle}
L^*_\eps(\varphi)(x,n)=-n\cdot \nabla_x \varphi(x,n)+\left(\alpha_\eps^a(x)+\alpha_\eps^s(x)\right)\varphi(x,n)\\-\alpha_\eps^s(x)\int_\stw dn'\;K(n,n')\varphi(x,n').
\end{multline}

\begin{lemma}[Weak maximum principle]\label{weakmax} If a continuous bounded function $ F(x,n) $ satisfies the boundary condition $F(x,n)\geq 0$ for $x\in\bnd$ and $n\cdot n_x<0$ and the inequality $ \int_\stw dn\;\int_\Omega dx\;L^*_\eps(\varphi)(x,n)F(x,n)\geq 0 $ for all non-negative $ \varphi\in C^1\left(\bar{\Omega}\times\stw\right) $ with $ \varphi(x,n)=0 $ for $ x\in\bnd $ and $ n\cdot n_x\geq0 $, then $ F(x,n)\geq 0 $ for all $ x,n\in\Omega\times\stw $.
\begin{remark}
    Before proving Lemma \ref{weakmax} we notice that by definition $ H_\eps $ is a continuous function which also satisfies $ H_\eps(x,n)=0 $ for $ x\in\bnd $ and $ n\cdot n_x<0 $.
\end{remark}
\begin{proof}
 Assume that the claim of Lemma \ref{weakmax} is not true. Then there exists an open set $U\subset \bar{\Omega}\times\stw$ such that $ F(x,n)<0 $ for every $(x,n)\in U$. Let $\xi\in C_c^1\left(U\right)$ with $\xi\geq 0$ and $\xi\not= 0$. We then consider the continuously differentiable function $\varphi$ defined by \begin{equation}\label{max.princ}
L_\eps^*(\varphi)(x,n)=\xi(x,n).
\end{equation}
Let us assume first that such $ \varphi $ exists. Then we can compute
\begin{multline*}
0\leq \int_\stw dn\;\int_\Omega dx\;L^*_\eps(\varphi)(x,n)F(x,n)=\int_\stw dn\;\int_\Omega dx\;\xi(x,n)F(x,n)\\=\int_U dn\;dx\;\xi(x,n)F(x,n)<0.
\end{multline*}
This contradiction implies the claim $ F(x,n)\geq 0 $ for all $ x,n\in\Omega\times\stw $.\\
We show now that such $ \varphi $ exists. Solving by characteristics the equation $$L_\eps^*(\varphi)(x,n)=\xi(x,n),$$ with boundary condition $ \varphi(x,n)=0 $ for $ x\in\bnd $ and $ n\cdot n_x>0 $ we obtain the following recursive formula
\begin{multline*}
\varphi(x,n)=\int_0^{s(x,-n)}\xi(x+tn,n)\exp\left(-\int_0^t\left[\alpha_\eps^a(x+\tau n)+\alpha_\eps^s(x+\tau n)\right]d\tau\right)dt\\+\int_0^{s(x,-n)}dt\;\alpha_\eps^s(x+tn)\exp\left(-\int_0^t\left[\alpha_\eps^a(x+\tau n)+\alpha_\eps^s(x+\tau n)\right]d\tau\right)\\\times\int_\stw dn' K(n,n')\varphi(x+tn,n'),
\end{multline*}
where $ s(x,-n) $ is the length of the line connecting $ x\in\Omega $ with the boundary $ \bnd $ in direction $ n\in\stw $. We still have to prove that $ \varphi $ is continuously differentiable and that it is non-negative. Since all functions $ \alpha_\eps^l $, $ K $ and the exponential functions are non-negative and continuously differentiable we consider the Duhamel expansion of $ \varphi $ as
\begin{multline}\label{max.pric.2}
\varphi(x,n)=\int_0^{s(x,-n)}\xi(x+tn,n)\exp\left(-\int_0^t\left[\alpha_\eps^a(x+\tau n)+\alpha_\eps^s(x+\tau n)\right]d\tau\right)dt\\
+\int_0^{s(x,-n)}dt\;\alpha_\eps^s(x+tn)\exp\left(-\int_0^t\left[\alpha_\eps^a(x+\tau n)+\alpha_\eps^s(x+\tau n)\right]d\tau\right)\int_\stw dn' K(n,n')\\\times\int_0^{s(x+tn,-n')}\xi(x+tn+t_1n',n')\exp\left(-\int_0^t(\alpha_\eps^a+\alpha_\eps^s)(x+tn+\tau n')d\tau\right)dt_1\\+\cdots=\sum_{i=1}^\infty T_i(x,n).
\end{multline}
Recursively, using 
\begin{equation*}
\int_0^D dr \;-\frac{d}{dr}\exp\left(-\int_{0}^r \left[\alpha_\eps^a(z+rn)+\alpha_\eps^s(z+rn)\right]\right) \leq 1-e^{\Arrowvert \alpha\Arrowvert_\infty D}\theta<1
\end{equation*}
for $ D=\text{diam}(\Omega) $ and $ \Arrowvert\alpha\Arrowvert_\infty=\Arrowvert \alpha^s+\alpha^a\Arrowvert_\infty $, the symmetry of $ K $ so that $$ \int_\stw dn' \ K(n,n')=1 ,$$ we can estimate each term of the Duhamel expansion of $ \varphi $ by $ \left| T_i(x,n)\right|\leq \Arrowvert\xi\Arrowvert_\infty D \theta^{i-1} $ and hence we obtain the absolute convergence of the Duhamel series since
\begin{equation*}
	\Arrowvert\varphi\Arrowvert_\infty \leq \Arrowvert\xi\Arrowvert_\infty D\sum_{i=0}^\infty \theta^i<\infty.
\end{equation*}
This implies the non-negativity and the continuity of $ \varphi $. To prove that $ \varphi $ is differentiable one proceeds in the same way. We write the recursive formula for the derivative of $ \varphi $ and we estimate the Duhamel expansion similarly as we did for the boundedness of $ \varphi $ using this time also the uniformly boundedness of $ \varphi $. We omit this computation since it is very similar to the one in \eqref{max.pric.2}.
\end{proof}
\end{lemma}
With this weak maximum principle we can carry on the proof of the contractivity of the operator $\mathcal{B}^\eps$ in \eqref{reg.fix.point}.
\subsection{Existence of solution to the regularized problem}\label{Sec.full.reg.existence}
We can apply the weak maximum principle to $ 1-H_\eps(x,n) $. Indeed, $ L_\eps(1-H_\varepsilon)=0 $ in distributional sense. With an approximation argument we see that $$\int_\stw  dn\int_\Omega dx\ L^*_\eps(\vphi)(x,n)(1-H_\eps(x,n))\geq \int_\stw dn\int_{\bnd} dx\ \vphi(x,n)(1-H_\eps(x,n)) n \cdot n_x\geq 0$$ for any non-negative $\vphi\in C^1\left(\bar{\Omega}\times\stw\right)$ with $\vphi(x,n)=0$ if $x\in\bnd$ and $n\cdot n_x\geq 0$. For similar arguments see also \cite{zbMATH07686011}. Therefore the weak maximum principle implies $ H_\eps(x,n)\leq 1 $ for all $ x,n\in\Omega\times \stw $.

Hence, estimating then $ H_\eps(x,n) $ by $ 1 $ in the following equation obtained by integrating the equation \eqref{green exact reg} for $ I_\eps $ we get
\begin{multline}\notag
H_\eps(x,n)= \int_\Omega dx_0 \;\alpha^a_\eps(u(x_0))\exp\left(-\int_{[x_0,x]}\left[\alpha_\eps^a(\xi)+\alpha_\eps^s(\xi)\right]d\xi\right)\\\times H(n\cdot(x-x_0))\delta^{(2)}\left(P_n^\perp(x-x_0)\right)\\
+\int_0^{s(x,n)}dt\; \alpha_\eps^s(x-tn)\exp\left(-\int_{[x-tn,x]}\left[\alpha_\eps^a(\xi)+\alpha_\eps^s(\xi)\right]d\xi\right)\\\times\int_\stw dn'K(n,n')H_\eps(x-tn,n')
\\ =\int_0^{s(x,n)}dt\;\alpha_\eps^a(x-tn)\exp\left(-\int_{[x-tn,x]}\left[\alpha_\eps^a(\xi)+\alpha_\eps^s(\xi)\right]d\xi\right)\\
+\int_0^{s(x,n)}dt\; \alpha_\eps^s(x-tn)\exp\left(-\int_{[x-tn,x]}\left[\alpha_\eps^a(\xi)+\alpha_\eps^s(\xi)\right]d\xi\right)\\\times\int_\stw dn'K(n,n')H_\eps(x-tn,n')\\
\leq \int_0^{s(x,n)}dt\;\left(\alpha_\eps^a(x-tn)+\alpha_\eps^s(x-tn)\right)\exp\left(-\int_{[x-tn,x]}\left[\alpha_\eps^a(\xi)+\alpha_\eps^s(\xi)\right]d\xi\right)\\
\leq \left(1-e^{-\Arrowvert\alpha\Arrowvert_\infty D}\right)=\theta<1,
\end{multline}  
where $ D $ is the diameter of $ \Omega $, $ \Arrowvert\alpha\Arrowvert_\infty=\Arrowvert\alpha^a+\alpha^s\Arrowvert_\infty $. The second equality is given solving the delta distribution together with the Heaviside function, while the first inequality is obtained by the isotropy of $ K $ so that $ \int_\stw K(n,n')\;dn'=1 $. Thus, equation \eqref{schauder.estm.1} implies that $\mathcal{B}^\eps$ is contractive with $\left| \mathcal{B}_\eps(u)(x)\right|\leq \theta \Arrowvert u\Arrowvert_{\sup}.$

We move now to the estimate for the boundary term given by $\mathcal{C}_\eps(u) $. It is enough to show that this term is uniformly bounded (say by a constant $ C>0 $), then for $ M\geq  \frac{C}{1-\theta} $ we have $ \left|\mathcal{B}_\eps(u)+\mathcal{C}_\eps(u)\right|\leq M $ for all $ x\in\Omega $ and $ 0\leq u\leq M $. In order to prove the boundedness we expand this boundary term in its Duhamel series taking as starting point the equation satisfied by $\mathcal{C}_\eps(u)$. We simplify the notation denoting by $ A_\eps^l(y,z-y) $ the function 
\begin{equation}\label{Adef}
		A_\eps^l(y,z-y)=\frac{\alpha_\eps^l(u(y))\exp\left(-\int_{[y,z]}\left[\alpha_\eps^a(u)+\alpha_\eps^s(u)\right]d\xi\right)}{|z-y|^2},
		\end{equation}
  for $ l\in\{a,s\} $ and by $E_\eps(z,\omega)$ the function of $\omega\in\stw$ and $z\in\Omega$ given by
  \begin{equation}\label{Edef}
      E_\eps(z,\omega)=\exp\left(-\int_{[y(z,\omega),z]}\left[\alpha_\eps^a(u)+\alpha_\eps^s(u)\right]d\xi\right).
  \end{equation}
  
  We put \eqref{poisson formula reg} into the definition of $ \mathcal{C}_\eps $ in \eqref{reg.fix.point}, we use \eqref{poisson leading order} and the notation above and we compute
\begin{multline}\label{duhamel.boundary}
		\mathcal{C}_\eps(u)(x)=\int_\stw dn\;G(n)E_\eps(x,n)\\+\int_\Omega d\eta A_\eps^s(\eta,x-\eta)\int_\stw dn_0 G(n_0) E_\eps(\eta,n_0)K\left(\frac{x-\eta}{|x-\eta|},n_0\right)\\
		+\int_\Omega d\eta A_\eps^s(\eta,x-\eta)\int_\Omega  d\eta_1\;A_\eps^s(\eta_1,\eta-\eta_1)K\left(\frac{x-\eta}{|x-\eta|},\frac{\eta-\eta_1}{|\eta-\eta_1|}\right)\\\times\int_\stw dn_0 G(n_0) E_\eps(\eta_1,n_0)K\left(\frac{\eta-\eta_1}{|\eta-\eta_1|},n_0\right)+\cdots=\sum_{i=1}^{\infty}\mathcal{U}^\eps_i(u)(x).
		\end{multline}
  
		We now estimate every term $ \mathcal{U}^\eps_i $. The first term is estimated by $$ \left|\mathcal{U}^\eps_1(u)\right|\leq 4\pi \Arrowvert G\Arrowvert_\infty ,$$ since the exponential term is bounded by $ 1 $. Estimating again the exponential term by $ 1 $ and $ g(n) $ by $ \Arrowvert G\Arrowvert_\infty $ and using the isotropy of the scattering kernel we compute
		\begin{multline}\label{est.duhamel.bnd2}
		\left|\mathcal{U}^\eps_2(u)\right|\leq  \Arrowvert G\Arrowvert_\infty \int_\Omega d\eta A_\eps^s(\eta,x-\eta)\\
		\leq \Arrowvert G\Arrowvert_\infty \int_\stw dn\int_0^D dr\left(-\frac{d}{dr}\exp\left(-\int_0^r dt\;\left[\alpha_\eps^a(x-tn)+\alpha_\eps^s(x-tn)\right]\right)\right)\\
		\leq 4\pi \theta \Arrowvert G\Arrowvert_\infty,
		\end{multline}
		where $ \theta=1-e^{-\Arrowvert\alpha\Arrowvert_\infty D} $. For the next terms we proceed similarly.
		\begin{multline}\label{est.duhamel.bnd3}
		\left|\mathcal{U}^\eps_3(u)\right|\leq  \Arrowvert G\Arrowvert_\infty \int_\Omega d\eta A_\eps^s(\eta,x-\eta)\int_\Omega  d\eta_1\;A_\eps^s(\eta_1,\eta-\eta_1)K\left(\frac{x-\eta}{|x-\eta|},\frac{\eta-\eta_1}{|\eta-\eta_1|}\right)\\
		\leq \theta\Arrowvert G\Arrowvert_\infty \int_\Omega d\eta\  A_\eps^s(\eta,x-\eta)
		\leq 4\pi \theta^2 \Arrowvert G\Arrowvert_\infty,
		\end{multline}
		where we estimated $$ \int_\stw dn\int_0^D \left(-\frac{d}{dr}\exp\left(-\int_0^r dt\;\left[\alpha_\eps^a(\eta-tn)+\alpha_\eps^s(\eta-tn)\right]\right)\right)K\left(\frac{x-\eta}{|x-\eta|},n\right)\leq \theta .$$
		Recursively, we conclude that the boundary term is uniformly bounded as
		\begin{equation}\label{est.duhamel.bnd}
		|\mathcal{C}_\eps(u)(x)|\le \sum_{i=1}^{\infty}|\mathcal{U}^\eps_i(u)(x)|\leq 4\pi \Arrowvert G\Arrowvert_\infty \sum_{i=0}^\infty \theta^i=4\pi \Arrowvert G\Arrowvert_\infty \frac{1}{1-\theta}<\infty.
		\end{equation}
        In a similar way, combining the fact that each term $\mathcal{U}^\eps_i$ maps continuously bounded (continuous) maps to bounded (continuous) maps and the uniform absolute convergence of the Duhamel series we can conclude that $\mathcal{C}_\eps$ is a continuous operator. 
		Next we prove that $ \mathcal{B}_\eps(u)(x)+\mathcal{C}_\eps(u)(x) $ is H\"older continuous. This will imply on the one hand that the operator is a self map and on the other hand that it is compact. Hence, Schauder fixed-point theorem concludes the existence of the regularize solutions satisfying \eqref{reg.fix.point}. Before starting this proof we recall that we have shown in Section \ref{sec.reg.exist} the H\"older continuity of all kind of operators given by
		\begin{equation*}
		\int_\Omega d\eta\;\frac{\alpha_\eps^l(u(\eta))\exp\left(-\int_{[\eta,x]}\left[\alpha_\eps^a(u)+\alpha_\eps^s(u)\right]d\xi\right)}{|x-\eta|^2},
		\end{equation*}
		for $ l\in\{a,s\} $ and 
		\begin{equation*}
		\int_\stw dn\;\exp\left(-\int_{[y(x,n),x]}\left[\alpha_\eps^a(u)+\alpha_\eps^s(u)\right]d\xi\right).
		\end{equation*}
		Moreover, in order to see the H\"older continuity of the interior term we have to expand the recursive formula of $ \mathcal{B}_\eps $ in its Duhamel series. We hence put \eqref{green exact reg} into the definition of $ \mathcal{B}_\eps $ in \eqref{reg.fix.point} and we compute
		\begin{multline}\label{duhamel.interior}
		\mathcal{B}_\eps(u)(x)=\int_\Omega dx_0\; \frac{A_\eps^a(x_0,x-x_0)u(x_0)}{4\pi}\\+\int_\Omega d\eta\int_\Omega d\eta_1\;A_\eps^s(\eta,x-\eta)\frac{A_\eps^a(\eta_1,\eta-\eta_1)u(\eta_1)}{4\pi}K\left(\frac{x-\eta}{|x-\eta|},\frac{\eta-\eta_1}{|\eta-\eta_1|}\right)\\+\int_\Omega d\eta\int_\Omega d\eta_1\int_\Omega d\eta_2\;A_\eps^s(\eta,x-\eta)A_\eps^s(\eta_1,\eta-\eta_1)\frac{A_\eps^a(\eta_2,\eta_1-\eta_2)u(\eta_2)}{4\pi}\\\times K\left(\frac{x-\eta}{|x-\eta|},\frac{\eta-\eta_1}{|\eta-\eta_1|}\right)K\left(\frac{\eta-\eta_1}{|\eta-\eta_1|},\frac{\eta_1-\eta_2}{|\eta_1-\eta_2|}\right)
		+\cdots
		=\sum_{i=1}^\infty \mathcal{V}^\eps_i(u)(x),
		\end{multline}
		where we used that $$ \int_\stw dn \;H(n\cdot (x-x_0))\delta^{(2)}(P_n^\perp(x-x_0))=\frac{1}{|x-x_0|^2},$$ in distributional sense.
  \begin{remark}
      Notice that \eqref{duhamel.boundary} and \eqref{duhamel.interior} encode the fact that due to the scattering the photons move along a polygonal line.
  \end{remark}
  Notice in addition that also $\mathcal{B}_\eps$ maps continuously bounded (continuous) functions to bounded (continuous) functions. This is due to the uniform absolute convergence of the Duhamel series (similar calculation as for \eqref{est.duhamel.bnd} and \eqref{holder.interior}) and the continuity of each term $\mathcal{V}_i^\eps$ in \eqref{duhamel.interior}.

  We aim to show the H\"older continuity of the operators $\mathcal{B}^\eps$ and $\mathcal{C}^\eps$. We consider hence $u\in C(\Omega)$ with $0\leq u\leq M$. As we did in Subsection \ref{sec.reg.exist} we extend $u$ continuously on the boundary $\bnd$ and then $u$, $\alpha^a(u)$ and $\alpha^s(u)$ by zero outside $\bar{\Omega}$.
  We proceed now estimating term by term the following difference for $ h\in\rth $ and $x,\ x+h\in\rth$
		\begin{equation*}
		\left|\mathcal{B}_\eps(u)(x)-\mathcal{B}_\eps(u)(x+h)\right|\leq \sum_{i=1}^\infty \left|\mathcal{V}^\eps_i(x)-\mathcal{V}^\eps_i(x+h)\right|.
		\end{equation*}
		For the first term we use the result in \eqref{I} and \eqref{II} and conclude
		\begin{equation*}
		\left|\mathcal{V}^\eps_1(x)-\mathcal{V}^\eps_1(x+h)\right|\leq C(\Omega, \Arrowvert\alpha\Arrowvert_\infty,\phi_\eps)\Arrowvert u\Arrowvert_\infty\left|h\right|^{\frac{1}{2}}.
		\end{equation*}
		For the next order terms we need also to estimate expressions of the form $$\left|K\left(\frac{x-\eta}{|x-\eta|},\frac{\eta-\xi}{|\eta-\xi|}\right)-K\left(\frac{x+h-\eta}{|x+h-\eta|},\frac{\eta-\xi}{|\eta-\xi|}\right)\right| .$$ Using the property of $ K $ being continuously differentiable in both variables and making use of the triangle inequality we know that there exists a constant $ C_K>0 $ depending exclusively on $ K $ such that 
		\begin{equation}\label{estimate.K}
		\left|K\left(\frac{x-\eta}{|x-\eta|},\frac{\eta-\xi}{|\eta-\xi|}\right)-K\left(\frac{x+h-\eta}{|x+h-\eta|},\frac{\eta-\xi}{|\eta-\xi|}\right)\right| \leq \frac{2C_K|h|^{\frac{1}{2}}}{|x+h-\eta|^{\frac{1}{2}}}.
		\end{equation}
		We can hence proceed with the second term of the Duhamel series. We apply the triangle inequality first and then we combine the results for \eqref{I} and \eqref{II} with the estimate \eqref{estimate.K} and with the fact that $|x|^{-\frac{5}{2}}\in L^1(B_1(0))$. Then we have
		\begin{multline}\label{estimate.interior.holder2}
		\left|\mathcal{V}^\eps_2(x)-\mathcal{V}^\eps_2(x+h)\right|\leq \Arrowvert u\Arrowvert \int_\Omega d\eta\int_\Omega d\eta_1\;\frac{A_\eps^a(\eta_1,\eta-\eta_1)}{4\pi}K\left(\frac{x-\eta}{|x-\eta|},\frac{\eta-\eta_1}{|\eta-\eta_1|}\right)\\\times\left|A_\eps^s(\eta, x-\eta)-A_\eps^s(\eta,x+h-\eta)\right|\\+\Arrowvert u\Arrowvert \int_\Omega d\eta\int_\Omega d\eta_1\;\frac{A_\eps^a(\eta_1,\eta-\eta_1)}{4\pi}A_\eps^s(\eta,x+h-\eta)\\\times\left|K\left(\frac{x-\eta}{|x-\eta|},\frac{\eta-\eta_1}{|\eta-\eta_1|}\right)-K\left(\frac{x+h-\eta}{|x+h-\eta|},\frac{\eta-\eta_1}{|\eta-\eta_1|}\right)\right|\\
		\leq\Arrowvert u\Arrowvert_\infty \theta\int_\Omega d\eta\;\frac{1}{4\pi}\left|A_\eps^s(\eta, x-\eta)-A_\eps^s(\eta,x+h-\eta)\right|\\+\Arrowvert u\Arrowvert_\infty C_K|h|^{\frac{1}{2}}\theta \int_\Omega d\eta\; \frac{A_\eps^s(\eta,x+h-\eta)}{|x+h-\eta|^{\frac{1}{2}}}\\
		\leq \Arrowvert u\Arrowvert_\infty\theta \left(C(\Omega,\phi_\eps,\Arrowvert \alpha\Arrowvert_\infty)+4\pi C_K \Arrowvert\alpha\Arrowvert_\infty D^{\frac{1}{2}}\right)|h|^{\frac{1}{2}},
		\end{multline}
		where we estimated as we did above\begin{multline*}
		\int_\Omega d\eta_1\;A_\eps^a(\eta_1,\eta-\eta_1)K\left(\frac{x-\eta}{|x-\eta|},\frac{\eta-\eta_1}{|\eta-\eta_1|}\right)\\\leq\int_\stw dn\int_0^D dr \;-\frac{d}{dr}\exp\left(-\int_{0}^r...\right)K\left(\frac{x-\eta}{|x-\eta|},n\right) \leq \theta,
		\end{multline*}
		with $ \theta=1-e^{-\Arrowvert\alpha\Arrowvert_\infty D} $ and similarly also $$ \int_\Omega d\eta_1\;A_\eps^a(\eta_1,\eta-\eta_1)\frac{1}{4\pi}\leq \theta .$$ We can iterate this procedure for all terms in the Duhamel series and we obtain the following estimate
		\begin{multline}\label{holder.interior}
		\left|\mathcal{B}_\eps(u)(x)-\mathcal{B}_\eps(u)(x+h)\right|\\\leq \Arrowvert u\Arrowvert_\infty\left(C(\Omega,\phi_\eps,\Arrowvert \alpha\Arrowvert_\infty)+4\pi C_K \Arrowvert\alpha\Arrowvert_\infty D^{\frac{1}{2}}\right)|h|^{\frac{1}{2}}\sum_{i=0}^\infty \theta^i\\= \Arrowvert u\Arrowvert_\infty\left(C(\Omega,\phi_\eps,\Arrowvert \alpha\Arrowvert_\infty)+4\pi C_K \Arrowvert\alpha\Arrowvert_\infty D^{\frac{1}{2}}\right)|h|^{\frac{1}{2}}\frac{1}{1-\theta}.
		\end{multline}
		Using the result \eqref{III} combined with \eqref{estimate.K} we see in the same way that also the boundary term operator is H\"older continuous with
		\begin{multline}\label{holder.boundary}
		\left|\mathcal{C}_\eps(u)(x)-\mathcal{C}_\eps(u)(x+h)\right|\\\leq 4\pi \Arrowvert G\Arrowvert_\infty\left(C(\Omega,\phi_\eps,\Arrowvert\alpha\Arrowvert_\infty)+C_K D^{\frac{1}{2}}\Arrowvert\alpha\Arrowvert_\infty\right)\left(1+\frac{1}{1-\theta}\right)|h|^{\frac{1}{2}}.
		\end{multline}
		Indeed, we compute using \eqref{III} for the first term in the Duhamel series of the boundary term
		\begin{multline*}
		\left|\mathcal{U}^\eps_1(x)-\mathcal{U}^\eps_1(x+h)\right|\leq \Arrowvert G\Arrowvert_\infty \int_\stw dn\; \left|E_\eps(x,n)-E_\eps(x+h,n)\right|\leq \Arrowvert G\Arrowvert_\infty C_\eps C(\Omega)|h|^{\frac{1}{2}}.
		\end{multline*}
		Moreover, the estimates \eqref{I}, \eqref{II} together with \eqref{estimate.K} gives for the second term
		\begin{multline*}
		\left|\mathcal{U}^\eps_2(x)-\mathcal{U}^\eps_2(x+h)\right|\\\leq \Arrowvert G\Arrowvert_\infty \int_\stw dn_0\int_\Omega d\eta\; K\left(\frac{x-\eta}{|x-\eta|},n_0\right)  \left|A^s_\eps(\eta,x-\eta)-A^s_\eps(\eta,x+h-\eta)\right|\\
		+\Arrowvert G\Arrowvert_\infty \int_\stw dn_0\int_\Omega d\eta\;A^s_\eps(\eta,x+h-\eta) \left| K\left(\frac{x-\eta}{|x-\eta|},n_0\right) -K\left(\frac{x+h-\eta}{|x+h-\eta|},n_0\right)\right|\\ \leq \Arrowvert G\Arrowvert_\infty |h|^{\frac{1}{2}}\left(C_\eps C(\Omega)+C_K 4\pi D^{\frac{1}{2}}\right).
		\end{multline*}
		Similarly integrating first with respect to $ n_0 $, then with respect to $ \eta_1 $ and finally with respect to $ \eta $ we obtain 
		\begin{multline*}
		\left|\mathcal{U}^\eps_3(x)-\mathcal{U}^\eps_3(x+h)\right|\\\leq \Arrowvert G\Arrowvert_\infty \int_\Omega d\eta\int_\Omega d\eta_1\int_\stw dn_0\;A^s_\eps(\eta_1,\eta-\eta_1) K\left(\frac{x-\eta}{|x-\eta|},\frac{\eta-\eta_1}{|\eta-\eta_1|}\right)\\\times K\left(\frac{\eta-\eta_1}{|\eta-\eta_1|},n_0\right)  \left|A^s_\eps(\eta,x-\eta)-A^s_\eps(\eta,x+h-\eta)\right|\\
		+\Arrowvert G\Arrowvert_\infty \int_\Omega d\eta\int_\Omega d\eta_1\int_\stw dn_0\;A^s_\eps(\eta_1,\eta-\eta_1)A^s_\eps(\eta,x+h-\eta)K\left(\frac{\eta-\eta_1}{|\eta-\eta_1|},n_0\right) \\\times \left| K\left(\frac{x-\eta}{|x-\eta|},\frac{\eta-\eta_1}{|\eta-\eta_1|}\right) -K\left(\frac{x+h-\eta}{|x+h-\eta|},\frac{\eta-\eta_1}{|\eta-\eta_1|}\right)\right|\\ \leq \Arrowvert G\Arrowvert_\infty \theta \int_\Omega d\eta  \left|A^s_\eps(\eta,x-\eta)-A^s_\eps(\eta,x+h-\eta)\right|\\+\Arrowvert G\Arrowvert_\infty \theta C_K |h|^{\frac{1}{2}} \int_\Omega d\eta \frac{A_\eps^s(\eta,x+h-\eta)}{|x+h-\eta|^{\frac{1}{2}}}
		\\\leq\Arrowvert G\Arrowvert_\infty \theta|h|^{\frac{1}{2}}\left(C_\eps C(\Omega)+C_K 4\pi D^{\frac{1}{2}}\right).
		\end{multline*}
		Iterating this procedure we obtain the estimate \eqref{holder.boundary}.
		
		Hence, \eqref{holder.interior} and \eqref{holder.boundary} imply that the operator $ \mathcal{B}_\eps+\mathcal{C}_\eps $ is a compact selfmap, mapping continuously uniformly bounded continuous functions to H\"older continuous functions.
        Thus, by the Schauder fixed-point theorem we obtain for every $ \eps>0 $ a solution $ u_\eps $ to \eqref{reg.fix.point}.
		\subsection{Compactness of the sequence of regularized solution and proof of Theorem \ref{thm full equation grey}}\label{Sec.full.cmpt}
  \begin{proof}[Proof of Theorem \ref{thm full equation grey}]
		In order to end the proof of Theorem \ref{thm full equation grey}, we will show that the sequence of regularized solution $ u_\eps $ to the equation \eqref{reg.fix.point} is compact in $ L^2 $. We already know that this is true in the case of pure absorption and emission, as we have seen in Subsection \ref{sec.grey proof}. We will use the compactness result of Subsection \ref{sec.grey proof} in order to show that the same result holds also in the case of scattering. A crucial role is played in this proof by the result of Proposition \ref{prop.cpt}. Let us consider a sequence $ \eps_j=\frac{1}{j} $. In order to simplify the notation we define the sequence of regularized solutions $ u_{\eps_j}=u_j $, the coefficients $ \alpha_{\eps_j}^l(u_{\eps_j})=\alpha_j^l(u_j) $ as well as all kind of operators $ \mathcal{B}_{\eps_j}=\mathcal{B}_j $, $ \mathcal{C}_{\eps_j}=\mathcal{C}_j $, $ A_{\eps_j}^l(y,z+y)=A_j^l(y,z+y) $ and $ E_{\eps_j}(z,\omega)=E_j(z,\omega) $. 
		
	
	By the uniformly boundedness of the sequence and the boundedness of $ \Omega $ we have only to show the equicontinuity, i.e. we want to prove that for any $ \beta>0 $ there exists a $ H_1(\beta)>0 $ such that 
	\begin{multline}\label{cpt.scattering1}
\int_\Omega dx\;\left|\mathcal{B}_{j}(u_{j})(x)+\mathcal{C}_{j}(u_{j})(x)-\mathcal{B}_{j}(u_{j})(x+h)-\mathcal{C}_{j}(u_{j})(x+h)\right|^2\\<C(\alpha^s,\alpha^a,\Omega, M,G)\beta
	\end{multline}
	for all $ |h|<H_1 $ and all $ j\in\mathbb{N} $. Notice that the constant $ C(\alpha^s,\alpha^a,\Omega, M,G) $ is independent of $ j\in\mathbb{N} $, of $\beta>0$ and of $h\in\rth$. This would imply the $ L^2- $compactness of the sequence $ \mathcal{B}_{j}(u_{j})+\mathcal{C}_{j}(u_{j}) $.

	In order to prove this statement we start recalling the boundedness of the interior term $ \mathcal{B}_j(u_{j}) $ and of the boundary term $  \mathcal{C}_{j}(u_{j}) $ as
	\begin{equation}
		\sup_{x\in\Omega}\left| \mathcal{B}_{j}(u_{j})(x)\right|\leq \sup_{x\in\Omega}\sum_{i=1}^\infty \left|\mathcal{V}^{j}_i(x)\right|\leq M\theta \sum_{i=0}^\infty \left(\theta\right)^i=\frac{M\theta}{1-\theta},
	\end{equation}
	where $ \theta=1-e^{\Arrowvert\alpha\Arrowvert_\infty D}<1 $. The computation is similar to the one we did in \eqref{est.duhamel.bnd} for the boundary term and to the H\"older estimate in \eqref{estimate.interior.holder2}. Moreover, \eqref{est.duhamel.bnd} implies $$\sup_{x\in\Omega}\left| \mathcal{C}_{j}(u_{j})(x)\right|\leq \sup_{x\in\Omega}\sum_{i=1}^\infty \left|\mathcal{U}^{j}_i(x)\right|\leq  \frac{4\pi}{1-\theta}\Arrowvert G\Arrowvert_\infty  .$$ Hence, let $ \beta>0 $. There exists an $ N_0(\beta)>0 $ such that
	\begin{equation}\label{cpt.scattering.tails}
	\sup_{x\in\Omega}\left| \sum_{i=N_0}^\infty \left|\mathcal{V}^{j}_i(x)\right|+\left|\mathcal{U}^{j}_i(x)\right|\right|^2<\beta.
	\end{equation}
	Thus, using the triangle inequality we obtain
	\begin{multline}\label{cpt.scattering.2}
\int_\Omega dx\;\left|\mathcal{B}_{j}(u_{j})(x)+\mathcal{C}_{j}(u_{j})(x)-\mathcal{B}_{j}(u_{j})(x+h)-\mathcal{C}_{j}(u_{j})(x+h)\right|^2\\\leq 2\beta|\Omega|+2\int_\Omega dx\;\left|\sum_{i=1}^{N_0-1} \mathcal{V}^{j}_i(x)- \mathcal{V}^{j}_i(x+h)\right|^2\\+2\int_\Omega dx\;\left|\sum_{i=1}^{N_0-1} \mathcal{U}^{j}_i(x)- \mathcal{U}^{j}_i(x+h)\right|^2\\
\leq 2\beta |\Omega| +2N_0 \sum_{i=1}^{N_0-1} \int_\Omega dx\;\left|\mathcal{V}^{j}_i(x)- \mathcal{V}^{j}_i(x+h)\right|^2\\+2N_0\sum_{i=1}^{N_0-1}\int_\Omega dx\;\left| \mathcal{U}^{j}_i(x)- \mathcal{U}^{j}_i(x+h)\right|^2.
	\end{multline}
 We aim to use for each term $$\int_\Omega dx\;\left|\mathcal{V}^{j}_i(x)- \mathcal{V}^{j}_i(x+h)\right|^2 \text{ and }\int_\Omega dx\;\left| \mathcal{U}^{j}_i(x)- \mathcal{U}^{j}_i(x+h)\right|^2,$$ for $0\leq 1\leq N_0-1$ the results in Corollary \ref{corollary.cpt} and Proposition \ref{prop.cpt} in order to show that they are equi-integrable.

	Let us start considering the interior terms and we write each of them in spherical coordinates. We extend by $ 0 $ all the functions $ \alpha^l $ and $ u_j $ which are defined only on the domain $ \Omega $. We denote by $ D=\text{diam}(\Omega)$ as usually. For $ i=1 $ we have 
	\begin{multline*}
	\mathcal{V}_i^j(x)=\int_\stw dn\int_0^D dr u_j(x-rn)\alpha^a_j\left(u_j(x-rn)\right)\\\times\exp\left(-\int_0^r \left[\alpha_j^a\left(u_j(x-\lambda n)\right)+\alpha^s_j\left(u_j(x-\lambda n)\right)\right]d\lambda\right).
	\end{multline*}
	We notice that taking $ \vphi_j(x)= u_j(x)\alpha^a_j\left(u_j(x)\right) $ and $ \psi_j(x)=\alpha_j^a\left(u_j(x)\right)+\alpha^s_j\left(u_j(x)\right) $ Corollary \ref{corollary.cpt} implies the compactness of the first interior term.
	Let us consider the second term. There we define 
	\begin{multline*}
	F_j^{(2)}(x,\omega)=\int_0^D d\lambda \int_\stw dn\; K(\omega,n)u_j(x-rn)\alpha^a_j\left(u_j(x-\lambda n)\right)\\\times\exp\left(-\int_0^\lambda \left[\alpha_j^a\left(u_j(x-r n)\right)+\alpha^s_j\left(u_j(x-r n)\right)\right]dr\right)
	\end{multline*}
	We notice that $ F_j^{(2)} $ is uniformly bounded in both variables and that it is uniformly continuous with respect to the second variable. Indeed, we can estimate on the one hand 
	\begin{equation}
\left|F_j^{(2)}(x,\omega)\right|\leq M\theta\int_\stw dn\; K(\omega,n)=M\theta
	\end{equation}
	and on the other hand also
		\begin{equation}\label{equicont.F.2}
	\left|F_j^{(2)}(x,\omega_1)-F_j^{(2)}(x,\omega_2)\right|\leq M\theta\int_\stw dn\;\left| K(\omega_1,n)-K(\omega_2,n)\right|\leq 4\pi M\theta C_K d(\omega_1,\omega_2).
	\end{equation}
	Hence, defining also  the error term
	\begin{multline}
\mathcal{R}_j^{(2)}(x)=\fint_\stw dn\int_{(s(x,n))}^D d\lambda F_j^{(2)}(x-\lambda n,n)\alpha_j^s\left(u_j(x-\lambda n)\right)\\\times\exp\left(-\int_0^\lambda \left[\alpha_j^a\left(u_j(x-r n)\right)+\alpha^s_j\left(u_j(x-r n)\right)\right]dr\right)
	\end{multline}
	we can write the second term of the operator $ \mathcal{B}_j $ as
	\begin{multline}
	\mathcal{V}_j^2(x)=-\mathcal{R}_j^{(2)}(x)+\fint_\stw dn\int_0^D d\lambda F_j^{(2)}(x-\lambda n,n)\alpha_j^s\left(u_j(x-\lambda n)\right)\\\times\exp\left(-\int_0^\lambda \left[\alpha_j^a\left(u_j(x-r n)\right)+\alpha^s_j\left(u_j(x-r n)\right)\right]dr\right).
	\end{multline}
	We notice that since $ \alpha_j(u_j) $ is supported in $ \Omega +\frac{1}{j} $ we can estimate the error term by $$ \left|\mathcal{R}_j^{(2)}(x)\right|\leq M\Arrowvert\alpha\Arrowvert_\infty \theta C(\Omega)\left(\frac{1}{j}\right)^{\frac{1}{2}} .$$ Hence, taking  $ \vphi_j(x,\omega)= F_j^{(2)}(x,\omega)\alpha^s_j\left(u_j(x)\right) $ and $ \psi_j(x)=\alpha_j^a\left(u_j(x)\right)+\alpha^s_j\left(u_j(x)\right) $, Proposition \ref{prop.cpt} implies the compactness in $ L^2(\Omega) $ of $ \mathcal{V}_j^2(x)+\mathcal{R}_j^{(2)}(x) $.
	
	We are ready now for the generalization of this result. We define for $ i\geq 2 $ the functions $ F_j^{(i)}(x,\omega) $ and $ \mathcal{R}_j^{(i)}(x) $ by
	\begin{multline}\label{F.grey}
	F_j^{(i)}(x,\omega)=\int_0^D d\lambda_2...\int_0^D d\lambda_i \int_\stw dn_2...\int_\stw dn_i\; K(\omega,n_2)...K(n_{i-1},n_i)\\\times\alpha^s_j\left(u_j(x-\lambda_2 n_2)\right)\exp\left(-\int_0^{\lambda_2} \left[\alpha_j^a\left(u_j(x-r n_2)\right)+\alpha^s_j\left(u_j(x-r n_2)\right)\right]dr\right)\\\times ...\times u_j(x-\lambda_2n_2+\cdots-\lambda_in_i)\alpha^a_j\left(u_j(x-\lambda_2n_2+\cdots-\lambda_in_i)\right)\\\times\exp\left(-\int_0^{\lambda_i} \left[\alpha_j^a\left(u_j(x-\lambda_2n_2+\cdots-rn_i)\right)+\alpha^s_j\left(u_j(x-\lambda_2n_2+\cdots-rn_i)\right)\right]dr\right)
	\end{multline}
	and 
	\begin{multline}
	\mathcal{V}_i^j(x)+\mathcal{R}_j^{(i)}(x)=\fint_\stw dn\int_0^D d\lambda F_j^{(i)}(x-\lambda n,n)\alpha_j^s\left(u_j(x-\lambda n)\right)\\\times\exp\left(-\int_0^\lambda \left[\alpha_j^a\left(u_j(x-r n)\right)+\alpha^s_j\left(u_j(x-r n)\right)\right]dr\right).
	\end{multline}
	Thus, we estimate $$ \left|F_j^{(i)}(x,\omega)\right|\leq M\theta^{i-1} , \  \left|	\mathcal{R}_j^{(i)}(x)\right|\leq (i-1)M\Arrowvert \alpha\Arrowvert_\infty \theta^{i-1} C(\Omega)\left(\frac{1}{j}\right)^{\frac{1}{2}},$$ as well as $$ \left|F_j^{(i)}(x,\omega_1)-F_j^{(i)}(x,\omega_2)\right|\leq 4\pi M\theta^{i-1} C_K d(\omega_1,\omega_2). $$ Again, Proposition \ref{prop.cpt} implies the $ L^2 $-compactness of $ \mathcal{V}_i^j(x)+\mathcal{R}_j^{(i)}(x) $. The compactness of $ \mathcal{V}_i^j(x)+\mathcal{R}_j^{(i)}(x) $ for $ 1\leq i<N_0(\beta) $ implies the existence of an $ h_0>0 $ such that
\begin{equation}
	\int_\Omega dx \left|\mathcal{V}_i^j(x)-\mathcal{V}_i^j(x+h)\right|^2\leq \frac{\beta}{2N_0(\beta)}+|\Omega|N_0M\Arrowvert \alpha\Arrowvert_\infty \theta^{i-1} C(\Omega)\left(\frac{1}{j}\right)
\end{equation}
for all $ |h|<h_0 $, for all $ j\geq 0 $ and for all $ 1\leq i<N_0 $. Hence,
\begin{multline}\label{cpt.scattering.interior}
\int_\Omega dx\;\left|\mathcal{B}_{j}(u_{j})(x)-\mathcal{B}_{j}(u_{j})(x+h)\right|^2\\
\leq 2\beta |\Omega| +2N_0 \sum_{i=1}^{N_0-1} \int_\Omega dx\;\left|\mathcal{V}^{j}_i(x)- \mathcal{V}^{j}_i(x+h)\right|^2\\
\leq \beta \left(2|\Omega|+1\right)+C(\Omega,\Arrowvert\alpha\Arrowvert_\infty,M)N_0^2\frac{1}{j},
\end{multline}
for all $ |h|<h_0 $ and for all $ j\geq 0 $.

We examine now to the operator $ \mathcal{C}_j(x) $ associated to the boundary term. We proceed similarly as for the interior term rewriting each expression $ \mathcal{U}_i^j $ in spherical coordinates. We start as usual with $ i=1 $, where we have
\begin{multline*}
\mathcal{U}_1^j(x)=\int_\stw dn\:G(n)\exp\left(-\int_0^{s(x,n)}\left[\alpha_j^a\left(u_j(x-r n)\right)+\alpha^s_j\left(u_j(x-r n)\right)\right]dr \right)\\
=\int_\stw dn\:G(n)\exp\left(-\int_0^{D}\left[\alpha_j^a\left(u_j(x-r n)\right)+\alpha^s_j\left(u_j(x-r n)\right)\right]dr \right)+\mathcal{R}_j^{(1)}(x),
\end{multline*}
where
\begin{multline*}
\left|\mathcal{R}_j^{(1)}(x)\right|=\left|\int_\stw dn\:G(n)\left[\exp\left(-\int_0^{s(x,n)}\left[\alpha_j^a\left(u_j(x-r n)\right)+\alpha^s_j\left(u_j(x-r n)\right)\right]dr\right)\right.\right.\\\left.\left.-\exp\left(-\int_0^{D}\left[\alpha_j^a\left(u_j(x-r n)\right)+\alpha^s_j\left(u_j(x-r n)\right)\right]dr \right)\right]\right|\\\leq\int_\stw dn\:G(n)\int_{s(x,n)}^D \left|\alpha_j^a\left(u_j(x-r n)\right)+\alpha^s_j\left(u_j(x-r n)\right)\right| dr\\\leq \Arrowvert G\Arrowvert_{L^1}\Arrowvert\alpha\Arrowvert_\infty C(\Omega)\left(\frac{1}{j}\right)^{\frac{1}{2}}.
\end{multline*}
Moreover, taking $ \psi_j=\alpha_j^a(u_j)+\alpha^s_j(u_j) $ and $ f=G $, Corollary \ref{corollary.cpt} implies the compactness of $ \mathcal{U}_1^j-\mathcal{R}_j^{(1)} $ in $ L^2(\Omega) $.

We proceed with $ i=2 $. Here with the change of variables $ \eta=x-\lambda_1n_1 $ we obtain
\begin{multline*}
\mathcal{U}_2^j(x)=\int_\Omega d\eta\ A_j^s(\eta,x-\eta)\int_\stw dn_0 G(n_0)E_j(\eta,n_0)K\left(\frac{x-\eta}{|x-\eta|},n_0\right)\\
=\int_\stw dn_1\int_0^{s(x,n_1)}d\lambda_1\int_\stw dn_0 K(n_1,n_0)G(n_0)  \alpha_j^s(u_j(x-\lambda_1n_1))\\\times \exp\left(-\int_0^{\lambda_1}\left[\alpha_j^a\left(u_j\right)+\alpha^s_j\left(u_j\right)\right](x-r n_1)dr \right)\\\times \exp\left(-\int_0^{s(x-\lambda_1n_1,n_0)}\left[\alpha_j^a\left(u_j\right)+\alpha^s_j\left(u_j\right)\right](x-\lambda_1n_1-rn_0)dr \right)\\
=\mathcal{R}_j^{(2)}(x)+\int_\stw dn_1\int_0^{D}d\lambda_1 Q_j^{(2)}(x-\lambda_1 n_1,n_1)\alpha_j^s(u_j(x-\lambda_1n_1))\\\times \exp\left(-\int_0^{\lambda_1}\left[\alpha_j^a\left(u_j\right)+\alpha^s_j\left(u_j\right)\right](x-r n_1)dr \right),
\end{multline*}
where
\begin{multline*}
 Q_j^{(2)}(x,\omega)=\int_\stw dn\ G(n) K(\omega,n) \exp\left(-\int_0^D\left[\alpha_j^a\left(u_j\right)+\alpha^s_j\left(u_j\right)\right](x-rn)dr \right)
\end{multline*}
and $$\left|\mathcal{R}_j^{(2)}(x)\right|\leq C(G,\Omega)\Arrowvert \alpha\Arrowvert_\infty \left(\theta +1\right)\left(\frac{1}{j}\right)^{\frac{1}{2}} .$$
Moreover, we see $ \left|Q_j^{(2)}(x,\omega)\right|\leq \Arrowvert G\Arrowvert_\infty $ and a similar computation to the one in \eqref{equicont.F.2} shows $$ \left|Q_j^{(2)}(x,\omega_1)-Q_j^{(2)}(x,\omega_2)\right|\leq \Arrowvert G\Arrowvert_{L^1}C_K d(\omega_1,\omega_2). $$ Hence, Proposition \eqref{prop.cpt} implies for $\vphi_j(x,\omega)=\alpha_j^s(u_j(x))Q_j^{(2)}(x,\omega)$ and $\psi_j=\alpha_j^a(u_j)+\alpha^s_j(u_j)$  the $ L^2 $-compactness of $ \mathcal{U}_2^j-\mathcal{R}_j^{(2)} $.

We proceed iteratively. We define for $ i\geq 3 $ the functions
\begin{multline}\label{Q.grey}
Q_j^{(i)}(x,\omega)=\int_\stw dn_2\int_0^D d\lambda_2 K(\omega,n_2) \alpha_j^s(u_j(x-\lambda_2n_2))\\\times \exp\left(-\int_0^{\lambda_2}\left[\alpha_j^a\left(u_j\right)+\alpha^s_j\left(u_j\right)\right](x-rn_2)dr \right)\\\times ...\times \int_\stw dn_i G(n_i)K(n_{i-1},n_i)\\\times \exp\left(-\int_0^{D}\left[\alpha_j^a\left(u_j\right)+\alpha^s_j\left(u_j\right)\right](x-\lambda_2n_2-...-rn_i)dr \right)
\end{multline}
and $ \mathcal{R}_j^{(i)}=  \mathcal{U}_i^j-Q_j^{(i)}$. Similarly as for the case $ i=2 $ we can estimate $$\left|Q_j^{(i)}(x,\omega)\right|\leq \theta^{(i-2)}\Arrowvert G\Arrowvert_\infty$$ and $$ \left|Q_j^{(i)}(x,\omega_1)-Q_j^{(i)}(x,\omega_2)\right|\leq \theta^{(i-2)}\Arrowvert G\Arrowvert_{L^1}C_K d(\omega_1,\omega_2). $$ Moreover, the remainder statisfies $$ \left|\mathcal{R}_j^{(i)}(x)\right|\leq C(G,\Omega)\Arrowvert \alpha\Arrowvert_\infty \theta^{(i-2)}\left(\frac{1}{j}\right)^{\frac{1}{2}} .$$

Again, Proposition \ref{prop.cpt} implies the $ L^2 $-compactness of $ \mathcal{U}_i^j(x)-\mathcal{R}_j^{(i)}(x) $. The compactness of $ \mathcal{U}_i^j(x)-\mathcal{R}_j^{(i)}(x) $ for $ 1\leq i<N_0(\beta) $ implies the existence of an $ h_1>0 $ such that
\begin{equation}
\int_\Omega dx \left|\mathcal{U}_i^j(x)-\mathcal{U}_i^j(x+h)\right|^2\leq \frac{\beta}{2N_0(\beta)}+|\Omega|N_0C(G,\Omega)\Arrowvert \alpha\Arrowvert_\infty \theta^{(i-2)}\left(\frac{1}{j}\right)
\end{equation}
for all $ |h|<h_1 $, for all $ j\geq 0 $ and for all $ 1\leq i<N_0 $. Hence,
\begin{multline}\label{cpt.scattering.bnd}
\int_\Omega dx\;\left|\mathcal{C}_{j}(u_{j})(x)-\mathcal{C}_{j}(u_{j})(x+h)\right|^2\\
\leq 2\beta |\Omega| +2N_0 \sum_{i=1}^{N_0-1} \int_\Omega dx\;\left|\mathcal{U}^{j}_i(x)- \mathcal{U}^{j}_i(x+h)\right|^2\\
\leq \beta \left(2|\Omega|+1\right)+C(\Omega,\Arrowvert\alpha\Arrowvert_\infty,G)N_0^2\frac{1}{j},
\end{multline}
for all $ |h|<h_1 $ and for all $ j\geq 0 $.
Putting equations \eqref{cpt.scattering.interior} and \eqref{cpt.scattering.bnd} in \eqref{cpt.scattering.2} we obtain for $ \beta>0 $, which was chosen arbitrary, the following estimate
	\begin{multline}\label{cpt.scattering.3}
\int_\Omega dx\;\left|\mathcal{B}_{j}(u_{j})(x)+\mathcal{C}_{j}(u_{j})(x)-\mathcal{B}_{j}(u_{j})(x+h)-\mathcal{C}_{j}(u_{j})(x+h)\right|^2\\\leq \beta \left(4|\Omega|+2\right)+C(\Omega,\Arrowvert\alpha\Arrowvert_\infty,G,M)N_0^2\frac{1}{j},
\end{multline}
for all $ |h|<\min\left(h_0,h_1\right) $ and for all $ j\geq 0 $. Taking now $ J_0= \frac{2N_0(\beta)^2}{\beta} $ we obtain
\begin{multline}\label{cpt.scattering.4}
\int_\Omega dx\;\left|\mathcal{B}_{j}(u_{j})(x)+\mathcal{C}_{j}(u_{j})(x)-\mathcal{B}_{j}(u_{j})(x+h)-\mathcal{C}_{j}(u_{j})(x+h)\right|^2\\\leq C(\Omega,\Arrowvert\alpha\Arrowvert_\infty,G,M)\beta,
\end{multline}
for all $ |h|<\min\left(h_0,h_1\right) $ and for all $ j\geq J_0 $. Since $ J_0\in\mathbb{N} $ is finite and $ \mathcal{B}_j(u_{j}) $ and $ \mathcal{C}_{j}(u_{j})$ are continuous there exists an $ H_0\leq \min\left(h_0,h_1\right) $ such that 
\begin{equation}\label{cpt.scattering.5}
\int_\Omega dx\;\left|\mathcal{B}_{j}(u_{j})(x)+\mathcal{C}_{j}(u_{j})(x)-\mathcal{B}_{j}(u_{j})(x+h)-\mathcal{C}_{j}(u_{j})(x+h)\right|^2\leq\beta,
\end{equation}
for all $ |h|<H_0 $ and $ 1\leq j<J_0 $. Hence, we have just proved that the uniformly bounded and tight sequence $ \mathcal{B}_{j}(u_{j})+\mathcal{C}_{j}(u_{j})$ is equicontinuous in $ L^2 $ and thus compact. There exists hence a subsequence $ u_{j_l} = \mathcal{B}_{j_l}(u_{j_l})+\mathcal{C}_{j_l}(u_{j_l})$ and a function $ u\in L^2(\Omega)\cap L^\infty(\Omega) $ such that $ u_{j_l} = \mathcal{B}_{j_l}(u_{j_l})+\mathcal{C}_{j_l}(u_{j_l})\to u $ in $ L^2(\Omega) $ and pointwise almost everywhere as $ l\to\infty $.

The uniformly boundedness of $u_{j_l}$ and also of $\alpha^a\left(u_{j_l}\right)$ and $ \alpha^s\left(u_{j_l}\right) $ implies the convergence in $L^p$ of $\alpha^a\left(u_{j_l}\right)*\phi_{j_l}\to \alpha^a(u)$ and $\alpha^s\left(u_{j_l}\right)*\phi_{j_l}\to \alpha^s(u)$ as $l\to \infty$ for $p<\infty$. Therefore for a subsequence (say still $u_{j_l}$) the convergence holds also pointwise almost everywhere. Finally a combination of the dominated convergence theorem for finitely many terms in terms in the Duhamel series and the convergence of such Duhamel series implies
\begin{equation}
u_{j_l}=\mathcal{B}_{j_l}(u_{j_l})+\mathcal{C}_{j_l}(u_{j_l})\to u=\mathcal{B}(u) +\mathcal{C}(u)
\end{equation}
pointwise almost everywhere as $l\to\infty$ and  $ u=\mathcal{B}(u) +\mathcal{C}(u) $ pointwise almost everywhere. Hence, $u$ is the desired solution to \eqref{fixedpointfull}.
\end{proof}
\subsection{Existence of solution for the pseudo Grey case}\label{Sec.full.pseudo}
We want to show the existence of solutions also in the pseudo Grey case, i.e. when the absorption and scattering coefficient depends also on the frequency via the relation $ \alpha_\nu^a(T(x))=Q_a(\nu)\alpha^a(T(x)) $ and $ \alpha_\nu^s(T(x))=Q_s(\nu)\alpha^s(T(x)) $. We assume that $ Q_i\in C^1\left(\mathbb{R}_+\right) $ and $ \alpha^i\in C^1\left(\mathbb{R}_+\right)$ for $ i=a,s $. It is not difficult to see that similarly as Theorem \ref{Grey thm} implies Proposition \ref{prop.cpt} and the Corollary \ref{corollary.cpt}, also Theorem \ref{main theorem pseudo grey} and the Federer-Besicovitch covering's lemma implies the following Proposition.

\begin{proposition}\label{prop.pseudo.grey}
Let $ \{\varphi_j\}_{j\in\mathbb{N}}\subset L^\infty\left(\Omega, L^1(\mathbb{R}_+)\right) $ and $ \{\psi_j\}_{j\in\mathbb{N}}\subset L^\infty\left(\Omega, L^1(\mathbb{R}_+)\right) $ be two non-negative bounded sequences with $ \Omega\subset \rth $  bounded, convex with $ C^1 $-boundary and strictly positive curvature. Let also $ f\in L^\infty\left(\stw, L^1(\mathbb{R}_+)\right) $ be non-negative. Then the sequences $$ \int_\stw dn \int_0^D dr\int_0^\infty d\nu \varphi_j(x-rn,\nu)\exp\left(-\int_0^r \psi_j(x-\lambda n,\nu)d\lambda\right)$$ and $$\int_\stw dn\int_0^\infty d\nu\;f(n,\nu)\exp\left(-\int_0^D \psi_j(x-\lambda n,\nu)d\lambda\right)$$ are compact in $ L^2(\Omega) $. 
If moreover $ \{\varphi_j\}_{j\in\mathbb{N}}\subset C\left(\stw,L^\infty\left(\Omega, L^1(\mathbb{R}_+)\right)\right) $ with $$\Arrowvert\vphi_j(\cdot,\cdot,\omega_1)-\vphi_j(\cdot,\cdot,\omega_2)\Arrowvert_{L^\infty\left(\Omega, L^1(\mathbb{R}_+)\right)}\leq \sigma(d(\omega_1,\omega_2))\to 0$$
as $ d(\omega_1,\omega_2)\to 0 $, where $ d $ is the metric on the sphere and $\sigma\in C\left(\mathbb{R}_+,\mathbb{R}_+\right)$ with $\sigma(0)=0$ is a uniform modulus of continuity, then the sequence
$$\int_\stw dn \int_0^D dr\int_0^\infty d\nu \varphi_j(x-rn,\nu,n)\exp\left(-\int_0^r \psi_j(x-\lambda n,\nu)d\lambda\right)$$ is also compact in $ L^2(\Omega) $.
\end{proposition}
Now we are ready to prove the existence of solution in the pseudo Grey case.
\begin{proof}[Proof of Theorem \ref{thm.full.pseudo.grey}]
	We proceed similarly as we did in the proof of Theorem \ref{thm full equation grey} indicating where differences arise. With the same notation as in Theorem \ref{thm full equation grey} we define therefore the Green functions $ \tI_\nu(x,n\;x_0) $ $ \psi_\nu(x,n;x_0,n_0) $ by
	\begin{multline}\label{eq.green.nu}
	n\cdot \nabla_x\tI_\nu(x,n;x_0)=Q_s(\nu)\alpha^s(T(x))\int_\stw K(n,n')\tI_\nu(x,n';x_0)\;dn'\\-\left(Q_a(\nu)\alpha^a\left(T(x)\right)+Q_s(\nu)\alpha^s\left(T(x)\right)\right)\tI_\nu(x,n;x_0)+\delta(x-x_0)
	\end{multline}
	with boundary condition for $ x\in\bnd $
	\begin{equation*}
	\tI_\nu(x,n;x_0)\chi_{\{n\cdot n_x<0\}}=0
	\end{equation*}
	and 
	\begin{equation}\label{eq.poisson.nu}
	\begin{split}
	n\cdot\nabla_x &\psi_\nu(x,n;x_0,n_0)=Q_s(\nu)\alpha^s(T(x))\int_\stw K(n,n')\psi_\nu(x,n';x_0,n_0)\;dn'\\
	&\ \ \ \ -\left(Q_a(\nu)\alpha^a\left(T(x)\right)+Q_s(\nu)\alpha^s\left(T(x)\right)\right)\psi_\nu(x,n;x_0,n_0),\  x\in\Omega,\\
	\psi_\nu(x,n;x_0,n_0)&\chi_{\{n\cdot n_x<0\}}=\delta_{\bnd}(x-x_0)\frac{\delta^{(2)}(n,n_0)}{4\pi},\  x\in\Omega, \ n_0\cdot N_{x_0}<0.
	\end{split}
	\end{equation}
	Then the intensity of radiation can be expressed in terms of these two functions as follows.
	\begin{multline}\label{I.green.poisson.nu}
	I_\nu(x,n)=\int_\Omega dx_0\; Q_a(\nu)\alpha^a(T(x_0))B_\nu(T(x_0))\tI_\nu(x,n;x_0)\\ +\int_\stw dn_0\int_{\bnd}dx_0\; g_\nu(n_0)\psi_\nu
	(x,n;x_0,n_0).
	\end{multline}
	Once again plugging in the definition of $ I_\nu(x,n) $ into equation \eqref{heat equation} we obtain the following fixed-point equation
	\begin{multline}\label{fixed.point.full.pseudo}
u(x)= \int_\stw dn \int_0^\infty d\nu \int_\Omega dx_0\; Q_a(\nu)^2 \alpha^a(u(x_0))B_\nu\left(F^{-1}(u(x_0))\right)\tI_\nu(x,n\;x_0)\\
+\int_0^\infty d\nu\int_\stw dn\int_\stw dn_0\int_\bnd dx_0\; Q_a(\nu)g_\nu(n_0)\psi_\nu(x,n;x_0,n_0),
	\end{multline}
	where $ u(x)=4\pi\int_0^\infty Q_a(\nu)B_\nu(T(x))= F(T(x)) $. Since $ B_\nu $ is a monotone function of the temperature, $ F $ is invertible.
	
Once more, we regularize the equation through a a sequence $ \phi_\eps $ of standard positive radial symmetric mollifiers. We hence define
$ I^\eps_\nu(x,n\;x_0) $ and $ \psi^\eps_\nu(x,n;x_0,n_0) $ by
\begin{multline}\label{eq.green.nu.eps}
n\cdot \nabla_xI^\eps_\nu(x,n;x_0)=Q_s(\nu)\alpha^s(T(\cdot))*\phi_\eps(x)\int_\stw K(n,n')I^\eps_\nu(x,n';x_0)\;dn'\\-\left(Q_a(\nu)\alpha^a\left(T(\cdot)\right)*\phi_\eps(x)+Q_s(\nu)\alpha^s\left(T(\cdot)\right)*\phi_\eps(x)\right)I^\eps_\nu(x,n;x_0)+\delta(x-x_0)
\end{multline}
with boundary condition for $ x\in\bnd $
\begin{equation*}
I_\nu^\eps(x,n;x_0)\chi_{\{n\cdot n_x<0\}}=0
\end{equation*}
and 
\begin{equation}\label{eq.poisson.nu.eps}
\begin{split}
&n\cdot\nabla_x \psi^\eps_\nu(x,n;x_0,n_0)=Q_s(\nu)\alpha^s(T(\cdot))*\phi_\eps(x)\int_\stw K(n,n')\psi_\nu^\eps(x,n';x_0,n_0)\;dn'\\
&\ \ \ \ -\left(Q_a(\nu)\alpha^a\left(T(\cdot)\right)*\phi_\eps(x)+Q_s(\nu)\alpha^s\left(T(\cdot)\right)*\phi_\eps(x)\right)\psi^\eps_\nu(x,n;x_0,n_0),\  x\in\Omega,\\
&\psi^\eps_\nu(x,n;x_0,n_0)\chi_{\{n\cdot n_x<0\}}=\delta_{\bnd}(x-x_0)\frac{\delta^{(2)}(n,n_0)}{4\pi},\  x\in\Omega, \ n_0\cdot N_{x_0}<0.
\end{split}
\end{equation}
The associated regularized fixed-point equation for $$ u(x)=\int_0^\infty d\nu\  Q_a(\nu)B_\nu(T(x))=F(T(x)) ,$$ is defined for any $ \eps>0 $ by
\begin{multline}\label{reg.fix.point.pseudo}
u_\eps(x)=\int_0^\infty d\nu Q_a(\nu)^2\int_\Omega dx_0\int_\stw dn\;\alpha^a_\eps(x_0)B_\nu\left(F^{-1}(u_\eps(x_0))\right)I_\nu^\eps(x,n;x_0)\\
+\int_0^\infty d\nu Q_a(\nu)\int_\stw dn_0\int_\stw dn\int_{\bnd}dx_0\;G(n_0)\psi_\nu^\eps(x,n;x_0,n_0)\\
\eqdef \mathcal{B}_\eps(u_\eps)(x)+\mathcal{C}_\eps(u_\eps)(x),
\end{multline}
where $ u_\eps $ is the solution of the fixed-point equation for $ \eps>0 $ and we used the notation $ \alpha_\eps^i(x)=\alpha^i(u_\eps(\cdot))*\phi_\eps(x) $ for $ i=a,s $. The same reasoning and computations we did in Subsection \eqref{Sec.full.green} hold also in this case, so that we can write the explicit recursive formula for both $ I_\nu^\eps $ and $ \psi_\nu^\eps $ as
\begin{multline}\label{green exact reg pseudo}
I_\nu^\eps(x,n;x_0)=\chi_\Omega(x_0)\exp\left(-\int_{[x_0,x]}\left[Q_a(\nu)\alpha_\eps^a(\xi)+Q_s(\nu)\alpha_\eps^s(\xi)\right]d\xi\right)\\\times H(n\cdot(x-x_0))\delta^{(2)}\left(P_n^\perp(x-x_0)\right)\\
+\int_0^{s(x,n)}dt\; Q_s(\nu)\alpha_\eps^s(x-tn)\exp\left(-\int_{[x-tn,x]}\left[Q_a(\nu)\alpha_\eps^a(\xi)+Q_s(\nu)\alpha_\eps^s(\xi)\right]d\xi\right)\\\times\int_\stw dn'K(n,n')I_\eps(x-tn,n';x_0),
\end{multline}
and
\begin{multline}\label{poisson formula reg pseudo}
\psi_\nu^\eps(x,n;x_0,n_0)=|n_0\cdot N_{x_0}|H(n_0\cdot(x-x_0))\\\times\delta^{(2)}\left(P_{n_0}^\perp (x-x_0)\right)\frac{\delta^{(2)}(n,n_0)}{4\pi}\exp\left(-\int_{[x-tn,x]}\left[Q_a(\nu)\alpha_\eps^a(\xi)+Q_s(\nu)\alpha_\eps^s(\xi)\right]d\xi\right)\\
+\int_0^{s(x,n)}dt\; Q_s(\nu)\alpha_\eps^s(x-tn)\exp\left(-\int_{[x-tn,x]}\left[Q_a(\nu)\alpha_\eps^a(\xi)+Q_s(\nu)\alpha_\eps^s(\xi)\right]d\xi\right)\\\times\int_\stw dn'K(n,n')\psi_\eps(x-tn,n';x_0,n_0).
\end{multline}
With these expressions we recover also the Duhamel representation of the bulk and boundary operators by
\begin{multline}\label{duhamel.interior.pseudo}
\mathcal{B}_\eps(u)(x)=\int_0^\infty d\nu Q_a(\nu)\int_\Omega dx_0\; \frac{Q_a(\nu)\alpha_\eps^a(x_0)B_\nu\left(u(x_0)\right)}{|x-x_0|^2}\\\times\exp\left(-\int_{[x_0,x]}\left[Q_a(\nu)\alpha_\eps^a(\xi)+Q_s(\nu)\alpha_\eps^s(\xi)\right]d\xi\right)\\+\int_0^\infty d\nu Q_a(\nu)\int_\Omega d\eta\int_\Omega d\eta_1\;K\left(\frac{x-\eta}{|x-\eta|},\frac{\eta-\eta_1}{|\eta-\eta_1|}\right)\frac{Q_s(\nu)\alpha_\eps^s(\eta)}{|x-\eta|^2}\\\times\exp\left(-\int_{[\eta,x]}\left[Q_a(\nu)\alpha_\eps^a(\xi)+Q_s(\nu)\alpha_\eps^s(\xi)\right]d\xi\right)\frac{Q_a(\nu)\alpha_\eps^a(\eta_1)B_\nu\left(u(\eta_1)\right)}{|\eta-\eta_1|^2}\\\times\exp\left(-\int_{[\eta_1,\eta]}\left[Q_a(\nu)\alpha_\eps^a(\xi)+Q_s(\nu)\alpha_\eps^s(\xi)\right]d\xi\right) \\+\int_0^\infty d\nu Q_a(\nu)\int_\Omega d\eta\int_\Omega d\eta_1\int_\Omega d\eta_2\;A_\eps^s(\eta,x-\eta,\nu)A_\eps^s(\eta_1,\eta-\eta_1,\nu)\\\times A_\eps^a(\eta_2,\eta_1-\eta_2,\nu)B_\nu(u(\eta_2)) K\left(\frac{x-\eta}{|x-\eta|},\frac{\eta-\eta_1}{|\eta-\eta_1|}\right) K\left(\frac{\eta-\eta_1}{|\eta-\eta_1|},\frac{\eta_1-\eta_2}{|\eta_1-\eta_2|}\right)\\
+\cdots
=\sum_{i=1}^\infty \mathcal{V}^\eps_i(u)(x),
\end{multline}
where we used the definition $$ A_\eps^i(z,y-z,\nu)=\frac{Q_i(\nu)\alpha_\eps^i(z)}{|y-z|^2}\exp\left(-\int_{[z,y]}\left[Q_a(\nu)\alpha_\eps^a(\xi)+Q_s(\nu)\alpha_\eps^s(\xi)\right]d\xi\right). $$
For the boundary operator we obtain similarly
\begin{multline}\label{duhamel.boundary.pseudo}
\mathcal{C}_\eps(u)(x)\\=\int_0^\infty d\nu\  Q_a(\nu)\int_\stw dn\;g_\nu(n)\exp\left(-\int_{[y(x,n),x]}\left[Q_a(\nu)\alpha_\eps^a(u)+Q_s(\nu)\alpha_\eps^s(u)\right]d\xi\right)\\+\int_0^\infty d\nu\  Q_a(\nu)\int_\Omega d\eta A_\eps^s(\eta,x-\eta,\nu)\int_\stw dn_0 g_\nu(n_0)\\\times\exp\left(-\int_{[y(\eta,n_0),\eta]}\left[Q_a(\nu)\alpha_\eps^a(u)+Q_s(\nu)\alpha_\eps^s(u)\right]d\xi\right)K\left(\frac{x-\eta}{|x-\eta|},n_0\right)\\
+\int_0^\infty d\nu\  Q_a(\nu)\int_\Omega d\eta A_\eps^s(\eta,x-\eta,\nu)\int_\Omega  d\eta_1\;A_\eps^s(\eta_1,\eta-\eta_1,\nu)\\\times K\left(\frac{x-\eta}{|x-\eta|},\frac{\eta-\eta_1}{|\eta-\eta_1|}\right)\int_\stw dn_0 g_\nu(n_0)\\\times\exp\left(-\int_{[y(\eta_1,n_0),\eta]}\left[Q_a(\nu)\alpha_\eps^a(u)+Q_s(\nu)\alpha_\eps^s(u)\right]d\xi\right)K\left(\frac{\eta-\eta_1}{|\eta-\eta_1|},n_0\right)\\+\cdots=\sum_{i=1}^{\infty}\mathcal{U}^\eps_i(u)(x).
\end{multline}
It can be shown, as we did in the pure Grey case, that the operator $ \mathcal{B}_\eps $ is a contraction, while the operator $ \mathcal{C}_\eps $ is bounded. For the first claim, we need to use a new version of the weak maximum-principle. We see that defining the function $ H_\eps(x,n,\nu) $ by 
$$ H_\eps(x,n,\nu)=\int_\Omega dx_0\;Q_a(\nu)\alpha_\eps^a(x_0)I^\eps_\nu\nu(u(x_0))$$
it satisfies the equation 
\begin{multline}\label{eq.H.eps.pseudo}
0=L_\eps(H_\eps)(x,n,\nu)-Q_a(\nu)\alpha_\eps^a(x)\\=n\cdot \nabla_x H_\eps(x,n,\nu)-Q_a(\nu)\alpha_\eps^a(x)\left(1-H_\eps(x,n,\nu)\right)\\-Q_s(\nu)\alpha_\eps^s(x)\int_\stw dn'\;K(n,n')\left(H_\eps(x,n,\nu)-H_\eps(x,n',\nu)\right).
\end{multline}
Notice that also in this case by definition $ H_\eps $ is non-negative, continuous and bounded, where the last assertions are due to its Duhamel expansion. Defining the adjoint operator by \begin{multline}\label{adjoint.max.principle.pseudo}
L^*_\eps(\varphi)(x,n,\nu)=-n\cdot \nabla_x \varphi(x,n,\nu)+\left(Q_a(\nu)\alpha_\eps^a(x)+Q_s(\nu)\alpha_\eps^s(x)\right)\varphi(x,n,\nu)\\-Q_s(\nu)\alpha_\eps^s(x)\int_\stw dn'\;K(n,n')\varphi(x,n',\nu)
\end{multline} 
we use the following weak-maximum principle
\begin{lemma}\label{weak.max.pseudo}
    If a continuous bounded $ F(x,n,\nu) $ satisfies the boundary condition $ F(x,n,\nu)\geq 0 $ for $ x\in\bnd $ and $ n\cdot n_x<0 $ and the inequality $$ \int_0^\infty d\nu\int_\stw dn\;\int_\Omega dx\;L^*_\eps(\varphi)(x,n,\nu)F(x,n,\nu)\geq 0 $$ for all non-negative $ \varphi\in C^1\left(\bar{\Omega}\times\stw\times\mathbb{R}_+\right) $ with $ \varphi(x,n,\nu)=0 $ for $ x\in\bnd $ and $ n\cdot n_x\geq0 $, then $ F(x,n,\nu)\geq 0 $ for all $ x,n,\nu\in\Omega\times\stw\times \mathbb{R}_+ $.
    \begin{proof}
    We assume that Lemma \ref{weak.max.pseudo} is not true. Hence, there exists an open set $ U\subset\bar{\Omega}\times\stw\times \mathbb{R}_+  $ such that $ F(x,n,\nu)<0 $ there. Taking then a function $ \xi\in C^1_c(U) $ with $ \xi\geq 0 $ and $ \xi\neq 0 $ we define the non-negative continuously differentiable function $ \vphi(x,n,\nu) $ by $$ L^*_\eps(\varphi)(x,n,\nu)=\xi(x,n,\nu) $$ and with boundary condition $ \varphi(x,n,\nu)=0 $ for $ x\in\bnd $ and $ n\cdot n_x>0 $. As we did in the proof of Lemma \ref{weakmax} one can show that $ \vphi\geq 0 $ and that it is continuously differentiable in all variables. Finally, one uses the constructed function in order to obtain a contradiction since 
\begin{multline*}
    0\leq \int_0^\infty d\nu\int_\stw dn\int_\Omega dxL^*_\eps(\varphi)(x,n,\nu)F(x,n,\nu)\\=\int_U d\nu dn dx \xi(x,n,\nu)F(x,n,\nu)<0.
\end{multline*}
\end{proof}
\end{lemma}
Using the fact that $ L_\eps(1-H_\eps)(x,n,\nu)=0 $ and the weak maximum principle in Lemma \ref{weak.max.pseudo} we conclude that $ 0\leq H_\eps\leq 1 $, where we used that by definition $ H_\eps(x,n,\nu)=0 $ for $ x\in\bnd $ and $ n\cdot n_x<0 $. Once again, using the recursive formula for $ H_\eps(x,n,\nu) $ and the estimate $$\int_0^D dr |f(x-rn)|\exp\left(-\int_0^r |f(x-tn)|dt\right)\leq 1-e^{-\Arrowvert f\Arrowvert_\infty D }<1,$$ we obtain, by defining $ \theta=1-e^{-\Arrowvert \alpha_\nu\Arrowvert D}<1 $ for $ \Arrowvert\alpha_\nu\Arrowvert=\Arrowvert Q_a\alpha^a+Q_s\alpha^s\Arrowvert_\infty $, the following estimate
\begin{multline*}
0\leq H_\eps(x,n,\nu)=\int_0^{s(x,n)} dt\;\alpha_\eps^a(x-tn)Q_a(\nu)\\\times\exp\left(-\int_0^t \left[Q_a(\nu)\alpha_\eps^a(x-rn)+Q_s(\nu)\alpha_\eps^s(x-rn)\right]dr\right)\\
+\int_0^{s(x,n)} dt\; \alpha_\eps^s(x-tn)Q_s(\nu)\exp\left(-\int_0^t \left[Q_a(\nu)\alpha_\eps^a(x-rn)+Q_s(\nu)\alpha_\eps^s(x-rn)\right]dr\right)\\\times\int_\stw dn'K(n,n')H_\eps(x-tn,n',\nu)\\
\leq \int_0^{s(x,n)}dt \left(\alpha_\eps^a(x-tn)Q_a(\nu)+\alpha_\eps^s(x-tn)Q_s(\nu)\right)\\\times\exp\left(-\int_0^t \left[Q_a(\nu)\alpha_\eps^a(x-rn)+Q_s(\nu)\alpha_\eps^s(x-rn)\right]dr\right)\leq \theta<1. 
\end{multline*}
Hence, we conclude the contractivity of the bulk operator via
\begin{multline*}
0\leq \mathcal{B}_\eps(u)(x)=\int_0^\infty d\nu Q_a(\nu)^2\int_\Omega dx_0\int_\stw dn\;\alpha^a_\eps(x_0)u_\eps(x_0)I_\nu^\eps(x,n;x_0)\\
\leq \int_0^\infty d\nu Q_a(\nu) B_\nu\left(F^{-1}(\Arrowvert u\Arrowvert_\infty)\right)H_\eps(x,n,\nu)\\\leq \theta F\left(F^{-1}(\Arrowvert u\Arrowvert_\infty)\right)=\theta \Arrowvert u\Arrowvert_\infty.
\end{multline*}
On the other hand also the boundary term is bounded, indeed in the same way as we had in the pure grey case using the fact that $$ \int_0^\infty d\nu Q_a(\nu) \int_\stw dn\ g_\nu(n)\leq \Arrowvert Q\Arrowvert_\infty \Arrowvert g\Arrowvert,$$ we obtain $$ \left|\mathcal{U}^\eps_i(x)\right|\leq  \Arrowvert Q\Arrowvert_\infty \Arrowvert g\Arrowvert \theta^{i-1} ,$$ for $ \theta=1-e^{-\Arrowvert \alpha_\nu\Arrowvert D}<1  $. 
Hence, the Duhamel series is absolutely convergent and $$ \left|\mathcal{C}_\eps(u)(x)\right|\leq C(Q,\alpha,D,g)<\infty .$$ Moreover, the continuity of the operator $ \mathcal{B}_\eps+\mathcal{C}_\eps $ can be shown using the convergence of the Duhamel expansions as we argued in Subsection \ref{Sec.full.reg.existence}. 

Thus, the operator $ \mathcal{B}_\eps+\mathcal{C}_\eps $ is a continuous self-map on the set $ \{u\in L^\infty(\Omega):0\leq u\leq M\} $ for some $ M>0 $ large enough. Moreover, in the same way as we have shown the H\"older continuity in the pure Grey case in Subsection \ref{Sec.full.reg.existence}, using that $$ \int_0^\infty d\nu\ Q_a(\nu) B_\nu\left(F^{-1}(\Arrowvert u\Arrowvert_\infty)\right)\leq  \Arrowvert u\Arrowvert_\infty,$$ we can show that $ \mathcal{B}_\eps+\mathcal{C}_\eps $ acting on $ \{u\in C(\Omega):0\leq u\leq M\} $ is a continuous self-map mapping continuous functions to H\"older continuous function, hence it is a compact continuous self-map. Schauder's fixed-point theorem implies the existence of regularized solutions $ u_\eps $ to the equation \eqref{reg.fix.point.pseudo}.\\

We are ready for the last step of the proof. We want to show the compactness of the sequence of regularized solutions $ u_{\frac{1}{j}}\eqdef u_j $. To this end we will use Proposition \ref{prop.pseudo.grey}. We proceed in the same way as in the proof of the pure Grey case of Theorem \ref{thm full equation grey}. We  write the terms in the Duhamel expansion of the bulk and boundary operators in spherical coordinates. For the interior terms we replace the definition of $ F_j^{i} $ in \eqref{F.grey} by
\begin{multline*}
F_j^{(i)}(x,\omega,\nu)=\int_0^D d\lambda_2...\int_0^D d\lambda_i \int_\stw dn_2...\int_\stw dn_i\; K(\omega,n_2)...K(n_{i-1},n_i)\\\times Q_s(\nu)\alpha^s_j(x-\lambda_2 n_2)\exp\left(-\int_0^{\lambda_2} \left[Q_a(\nu)\alpha_j^a(x-r n_2)+Q_s(\nu)\alpha^s_j(x-r n_2)\right]dr\right)\\\times ...\times B_\nu\left(F^{-1}(u_j(x-\lambda_2n_2+..-\lambda_in_i))\right)Q_a(\nu)\alpha^a_j(x-\lambda_2n_2+..-\lambda_in_i)\\\times\exp\left(-\int_0^{\lambda_i} \left[Q_a(\nu)\alpha_j^a(x-\lambda_2n_2+..-rn_i)+Q_s(\nu)\alpha^s_j(x-\lambda_2n_2+..-rn_i)\right]dr\right),
\end{multline*}
for $ i\geq 2 $ and $ F_j^{(1)}(x,\omega,\nu)=Q_a(\nu)B_\nu\left(F^{-1}(u)(x)\right) $. We notice that $$ 0\leq\int_0^\infty d\nu F_j^{(i)}(x,\omega,\nu)\leq \theta^{i-1}M. $$ Moreover, $ F_j^{(i)} $ is also uniformly continuous with respect to the variable $ \omega $. Then, Proposition \ref{prop.pseudo.grey} implies that all terms of the form
\begin{multline*}
\tilde{\mathcal{V}}_j^i=\int_0^\infty d\nu \int_0^Ddt\int_\stw dn F_j^{(i)}(x-tn,n,\nu)Q_s(\nu)\alpha_j^s(x-tn)\\\times \exp\left(-\int_0^{t} \left[Q_a(\nu)\alpha_j^a(x-r n_2)+Q_s(\nu)\alpha^s_j(x-r n_2)\right]dr\right),
\end{multline*}
for $ i\geq 2 $ and 
\begin{multline*}
\tilde{\mathcal{V}}_j^1=\int_0^\infty d\nu \int_0^Ddt\int_\stw dn F_j^{(1)}(x-tn,n,\nu)Q_a(\nu)^2\alpha_j^a(x-tn)\\\times \exp\left(-\int_0^{t} \left[Q_a(\nu)\alpha_j^a(x-r n_2)+Q_s(\nu)\alpha^s_j(x-r n_2)\right]dr\right),
\end{multline*}
are compact in $ L^2(\Omega) $. Since the error terms can be still estimated by $$ \left|\mathcal{V}_j^i(x)-\tilde{\mathcal{V}}_j^i(x)\right|\leq C(M,\Arrowvert\alpha\Arrowvert,\Arrowvert Q\Arrowvert, \Omega)\theta^{i-1}\frac{1}{j} $$ and the Duhamel series is convergent, for any $ \beta $ there exists some $ N_0>0 $ and an $ h_0>0 $ such that 
\begin{multline}\label{cpt.scattering.interior.pseudo}
\int_\Omega dx\;\left|\mathcal{B}_{j}(u_{j})(x)-\mathcal{B}_{j}(u_{j})(x+h)\right|^2\\
\leq \beta \left(2|\Omega|+1\right)+C(\Omega,\Arrowvert\alpha\Arrowvert,\Arrowvert Q\Arrowvert,M)N_0^2\frac{1}{j},
\end{multline}
for all $ |h|<h_0 $ and for all $ j\geq 0 $.\\

In a very similar way we consider the terms in the Duhamel expansion of the boundary operator written in spherical coordinates. Here we replace the definition of $ Q_j^{(i)} $ in \eqref{Q.grey} by
\begin{multline}\label{Q.grey.pseudo}
Q_j^{(i)}(x,\omega,\nu)=\int_\stw dn_2\int_0^D d\lambda_2 K(\omega,n_2) Q_s(\nu)\alpha_j^s(x-\lambda_2n_2)\\\times \exp\left(-\int_0^{\lambda_2}\left[Q_a(\nu)\alpha_j^a\left(u_j\right)+Q_s(\nu)\alpha^s_j\left(u_j\right)\right](x-rn_2)dr \right)\\\times ...\times \int_\stw dn_i Q_a(\nu)g_\nu(n_i)K(n_{i-1},n_i)\\\times \exp\left(-\int_0^{D}\left[Q_a(\nu)\alpha_j^a\left(u_j\right)+Q_s(\nu)\alpha^s_j\left(u_j\right)\right](x-\lambda_2n_2-...-rn_i)dr \right),
\end{multline}
for $ i\geq 2 $. Again, $ Q_j^{(i)} $ satisfies the assumption of Proposition \ref{prop.pseudo.grey} with $$ 0\leq \int_0^\infty d\nu\  Q_j^{(i)}(x,\omega,\nu)\leq \Arrowvert Q\Arrowvert_\infty\Arrowvert g\Arrowvert_{L^\infty\left(\stw, L^1(\mathbb{R}_+)\right)}\theta^{i-2} .$$ Hence, all terms of the form
\begin{multline*}
\tilde{\mathcal{U}}_j^i=\int_0^\infty d\nu \int_0^Ddt\int_\stw dn\  Q_j^{(i)}(x-tn,n,\nu)Q_s(\nu)\alpha_j^s(x-tn)\\\times \exp\left(-\int_0^{t} \left[Q_a(\nu)\alpha_j^a(x-r n_2)+Q_s(\nu)\alpha^s_j(x-r n_2)\right]dr\right),
\end{multline*}
for $ i\geq 2 $ and 
\begin{multline*}
\tilde{\mathcal{U}}_j^1=\int_0^\infty d\nu \int_\stw dn \  Q_a(\nu)g_\nu(n)\\\times \exp\left(-\int_0^{D} \left[Q_a(\nu)\alpha_j^a(x-r n_2)+Q_s(\nu)\alpha^s_j(x-r n_2)\right]dr\right),
\end{multline*}
are compact in $ L^2(\Omega) $. Once more, the error terms can be still estimated by $$ \left|\mathcal{U}_j^i(x)-\tilde{\mathcal{U}}_j^i(x)\right|\leq C(\Arrowvert g\Arrowvert,\Arrowvert\alpha\Arrowvert,\Arrowvert Q\Arrowvert, \Omega)\theta^{i-2}\frac{1}{j} $$ and  $$ \left|\mathcal{U}_j^1(x)-\tilde{\mathcal{U}}_j^1(x)\right|\leq C(\Arrowvert g\Arrowvert,\Arrowvert\alpha\Arrowvert,\Arrowvert Q\Arrowvert, \Omega)\frac{1}{j} .$$ This together with the absolute convergence of the Duhamel series implies for any $ \beta $ the existence of some $ N_0>0 $ and an $ h_1>0 $ such that 
\begin{multline}\label{cpt.scattering.bnd.pseudo}
\int_\Omega dx\;\left|\mathcal{C}_{j}(u_{j})(x)-\mathcal{C}_{j}(u_{j})(x+h)\right|^2\\
\leq \beta \left(2|\Omega|+1\right)+C(\Omega,\Arrowvert\alpha\Arrowvert,\Arrowvert Q\Arrowvert,\Arrowvert g\Arrowvert)N_0^2\frac{1}{j},
\end{multline}
for all $ |h|<h_1 $ and for all $ j\geq 0 $.
Now we can conclude exactly as in the proof of Theorem \ref{thm full equation grey} that the sequence $ u_j= \mathcal{B}_{j}(u_{j})+\mathcal{C}_{j}(u_{j})$ is compact in $ L^2 $. Extracting a subsequence converging pointwise almost every where to some $ u $ and arguing with the dominated convergence theorem and the absolute convergence of the Duhamel series we can show the existence of a solution to the fixed-point equation \eqref{fixed.point.full.pseudo}.
\end{proof}
\section*{Acknowledgement} The authors gratefully acknowledge the support of the grant CRC
1060 ``The Mathematics of Emergent Effects" (Project-ID 211504053) of the University of Bonn funded
through the Deutsche Forschungsgemeinschaft (DFG, German Research Foundation). E. Demattè is supported by the Bonn International Graduate School of Mathematics (BIGS) at the Hausdorff Center for Mathematics. J. W. Jang is supported by the National Research Foundation of Korea (NRF) grants RS-2023-00210484, RS-2023-00219980, and 2021R1A6A1A10042944, and also by Posco Holdings Inc. and Samsung Electronics Co., Ltd. E. Demattè and J. J. L. Vel\'azquez are also funded by DFG under Germany's Excellence Strategy-EXC-2047/1-390685813.

\section*{Statements and Declarations}
The authors have no relevant financial or non-financial interests to disclose. Data sharing not applicable to this article as no datasets were generated or analysed during the current study.
\bibliographystyle{amsplain3links}
\bibliography{bibliography.bib}{}

\providecommand{\bysame}{\leavevmode\hbox to3em{\hrulefill}\thinspace}
\providecommand{\href}[2]{#2}
\begin{thebibliography}{10}
\expandafter\ifx\csname urlpdf\endcsname\relax
  \def\urlpdf#1{\url{#1}}\fi
\expandafter\ifx\csname arxiv\endcsname\relax
  \def\arxiv#1{\burlalt{arXiv:#1}{http://arxiv.org/abs/#1}}\fi
\expandafter\ifx\csname doi\endcsname\relax
  \def\doi#1{\burlalt{doi:#1}{http://dx.doi.org/#1}}\fi
\expandafter\ifx\csname href\endcsname\relax
  \def\href#1#2{#2}\fi
\expandafter\ifx\csname burlalt\endcsname\relax
  \def\burlalt#1#2{\href{#2}{#1}}\fi

\bibitem{NouriArkeryd}
L.~Arkeryd and A.~Nouri, \emph{Discrete velocity {B}oltzmann equations in the
  plane: stationary solutions}, Anal. PDE \textbf{16} (2023), no.~8,
  1869--1884, \doi{10.2140/apde.2023.16.1869}.

\bibitem{B2}
C.~Bardos, R.~Caflisch, and B.~Nicolaenko, \emph{Different aspects of the
  {M}ilne problem (based on energy estimates)}, Proceedings of the conference
  on mathematical methods applied to kinetic equations ({P}aris, 1985),
  vol.~16, 1987, pp.~561--585, \doi{10.1080/00411458708204306}.

\bibitem{B1}
C.~Bardos, F.~Golse, and B.~Perthame, \emph{The radiative transfer equations:
  existence of solutions and diffusion approximation under accretivity
  assumptions---a survey}, Proceedings of the conference on mathematical
  methods applied to kinetic equations ({P}aris, 1985), vol.~16, 1987,
  pp.~637--652, \doi{10.1080/00411458708204308}.

\bibitem{B3}
\bysame, \emph{The {R}osseland approximation for the radiative transfer
  equations}, Comm. Pure. Appl. Math. \textbf{40} (1987), no.~6, 691--721,
  \doi{https://doi.org/10.1002/cpa.3160400603}.

\bibitem{Golse-Bardos}
C.~Bardos, F.~Golse, B.~Perthame, and R.~Sentis, \emph{The nonaccretive
  radiative transfer equations: existence of solutions and {R}osseland
  approximation}, J. Funct. Anal. \textbf{77} (1988), no.~2, 434--460,
  \doi{10.1016/0022-1236(88)90096-1}.

\bibitem{MR3324151}
G.~Basile, A.~Nota, F.~Pezzotti, and M.~Pulvirenti, \emph{Derivation of the
  {F}ick's law for the {L}orentz model in a low density regime}, Comm. Math.
  Phys. \textbf{336} (2015), no.~3, 1607--1636,
  \doi{10.1007/s00220-015-2306-z}.

\bibitem{MR533346}
A.~Bensoussan, J.-L. Lions, and G.~C. Papanicolaou, \emph{Boundary layers and
  homogenization of transport processes}, Publ. Res. Inst. Math. Sci.
  \textbf{15} (1979), no.~1, 53--157, \doi{10.2977/prims/1195188427}.

\bibitem{BBS}
C.~Boldrighini, L.~A. Bunimovich, and Ya.~G. Sina\u{\i}, \emph{On the
  {B}oltzmann equation for the {L}orentz gas}, J. Statist. Phys. \textbf{32}
  (1983), no.~3, 477--501, \doi{10.1007/BF01008951}.

\bibitem{MR1234995}
C.~Boutin and J.-L. Auriault, \emph{Rayleigh scattering in elastic composite
  materials}, Internat. J. Engrg. Sci. \textbf{31} (1993), no.~12, 1669--1689,
  \doi{10.1016/0020-7225(93)90082-6}.

\bibitem{MR2759829}
H.~Brezis, \emph{Functional analysis, {S}obolev spaces and partial differential
  equations}, Universitext, Springer, New York, 2011.

\bibitem{clouet2009milne}
J.~F. Clou{\"e}t and R.~Sentis, \emph{Milne problem for non-grey radiative
  transfer}, Kinetic and Related Models \textbf{2} (2009), no.~2, 345--362.

\bibitem{Compton}
K.~T. Compton, \emph{{LXXIII}. {S}ome properties of resonance radiation and
  excited atoms}, The London, Edinburgh, and Dublin Philosophical Magazine and
  Journal of Science \textbf{45} (1923), no.~268, 750--760.

\bibitem{MR3412332}
R.~M.~S. da~Gama, \emph{Existence, uniqueness and construction of the solution
  of the energy transfer problem in a rigid and non-convex blackbody with
  temperature-dependent thermal conductivity}, Z. Angew. Math. Phys.
  \textbf{66} (2015), no.~5, 2921--2939, \doi{10.1007/s00033-015-0549-3}.

\bibitem{Dematte2023}
E.~Demattè, \emph{On a kinetic equation describing the behavior of a gas
  interacting mainly with radiation}, Journal of Statistical Physics
  \textbf{190} (2023), no.~7, 124,
  \doi{https://doi.org/10.1007/s10955-023-03128-0 10.1007/s10955-023-03128-0}.

\bibitem{DeVorePetrova}
R.~DeVore and G.~Petrova, \emph{The averaging lemma}, J. Amer. Math. Soc.
  \textbf{14} (2001), no.~2, 279--296, \doi{10.1090/S0894-0347-00-00359-3}.

\bibitem{DIPERNA1991271}
R.~J. Diperna, P.-L. Lions, and Y.~Meyer, \emph{${L}^p$ regularity of velocity
  averages}, Annales de l'Institut Henri Poincaré C, Analyse non linéaire
  \textbf{8} (1991), no.~3, 271--287,
  \doi{https://doi.org/10.1016/S0294-1449(16)30264-5}.

\bibitem{MR2478911}
P.-\'{E}. Druet, \emph{Weak solutions to a stationary heat equation with
  nonlocal radiation boundary condition and right-hand side in {$L^p\ (p\geq
  1)$}}, Math. Methods Appl. Sci. \textbf{32} (2009), no.~2, 135--166,
  \doi{10.1002/mma.1029}.

\bibitem{evans}
L.~C. Evans, \emph{Partial {D}ifferential {E}quations}, Graduate studies in
  mathematics, American Mathematical Society, 2010.

\bibitem{federer}
H.~Federer, \emph{Geometric {M}easure {T}heory}, Grundlehren der mathematischen
  Wissenschaften, Springer Berlin Heidelberg, 1969.

\bibitem{G}
G.~Gallavotti, \emph{Statistical mechanics}, Texts and Monographs in Physics,
  Springer-Verlag, Berlin, 1999, \doi{10.1007/978-3-662-03952-6}, A short
  treatise.

\bibitem{ghattassi2022diffusive}
M.~Ghattassi, X.~Huo, and N.~Masmoudi, \emph{On the diffusive limits of
  radiative heat transfer system {I}: well-prepared initial and boundary
  conditions}, SIAM Journal on Mathematical Analysis \textbf{54} (2022), no.~5,
  5335--5387.

\bibitem{ghattassi2023diffusive1}
\bysame, \emph{Diffusive limits of the steady state radiative heat transfer
  system: Boundary layers}, Journal de Math{\'e}matiques Pures et
  Appliqu{\'e}es \textbf{175} (2023), 181--215.

\bibitem{ghattassi2023diffusive2}
\bysame, \emph{Diffusive limits of the steady state radiative heat transfer
  system: Curvature effects}, arXiv preprint arXiv:2305.17661 (2023).

\bibitem{ghattassi2023stability}
\bysame, \emph{Stability of the nonlinear {M}ilne {P}roblem for radiative heat
  transfer system}, Archive for Rational Mechanics and Analysis \textbf{247}
  (2023), no.~5, 102.

\bibitem{MR3797032}
M.~Ghattassi, J.~R. Roche, and D.~Schmitt, \emph{Existence and uniqueness of a
  transient state for the coupled radiative-conductive heat transfer problem},
  Comput. Math. Appl. \textbf{75} (2018), no.~11, 3918--3928,
  \doi{10.1016/j.camwa.2018.03.002}.

\bibitem{golse1987milne}
F.~Golse, \emph{The {M}ilne problem for the radiative transfer equations (with
  frequency dependence)}, Transactions of the American Mathematical Society
  \textbf{303} (1987), no.~1, 125--143.

\bibitem{GOLSE1988110}
F.~Golse, P.-L. Lions, B.~Perthame, and R.~Sentis, \emph{Regularity of the
  moments of the solution of a {T}ransport {E}quation}, Journal of Functional
  Analysis \textbf{76} (1988), no.~1, 110--125,
  \doi{https://doi.org/10.1016/0022-1236(88)90051-1}.

\bibitem{golse2008radiative}
F.~Golse and F.~Salvarani, \emph{Radiative transfer equations and {R}osseland
  approximation in gray matter}, Waves and Stability in Continuous Media, World
  Scientific, 2008, pp.~321--326.

\bibitem{MR3686005}
Y.~Guo and L.~Wu, \emph{Geometric correction in diffusive limit of neutron
  transport equation in 2{D} convex domains}, Arch. Ration. Mech. Anal.
  \textbf{226} (2017), no.~1, 321--403, \doi{10.1007/s00205-017-1135-y}.

\bibitem{Holstein}
T.~Holstein, \emph{Imprisonment of resonance radiation in gases}, Physical
  Review \textbf{72} (1947), no.~12, 1212.

\bibitem{refId0}
P.-E. Jabin and B.~Perthame, \emph{Regularity in kinetic formulations via
  averaging lemmas}, ESAIM: COCV \textbf{8} (2002), 761--774,
  \doi{10.1051/cocv:2002033}.

\bibitem{2109.10071}
J.~W. Jang and J.~J.~L. Vel{\'a}zquez, \emph{L{TE} and non-{LTE} solutions in
  gases interacting with radiation}, J. Stat. Phys. \textbf{186} (2022), no.~3,
  Paper No. 47, 62, \doi{10.1007/s10955-022-02888-5}.

\bibitem{zbMATH07686011}
\bysame, \emph{Kinetic models for semiflexible polymers in a half-plane}, Pure
  Appl. Anal. \textbf{5} (2023), no.~1, 145--212 (English),
  \doi{10.2140/paa.2023.5.145}.

\bibitem{jang2023temperature}
\bysame, \emph{On the temperature distribution of a body heated by radiation},
  2023, \arxiv{2211.11289}.

\bibitem{PhysRev.42.823}
C.~Kenty, \emph{On {R}adiation {D}iffusion and the {R}apidity of {E}scape of
  {R}esonance {R}adiation from a {G}as}, Phys. Rev. \textbf{42} (1932),
  823--842, \doi{10.1103/PhysRev.42.823}.

\bibitem{MR1608072}
M.~T. Laitinen and T.~Tiihonen, \emph{Integro-differential equation modelling
  heat transfer in conducting, radiating and semitransparent materials}, Math.
  Methods Appl. Sci. \textbf{21} (1998), no.~5, 375--392,
  \doi{10.1002/(SICI)1099-1476(19980325)21:5<375::AID-MMA953>3.0.CO;2-U}.

\bibitem{MR1866555}
\bysame, \emph{Conductive-radiative heat transfer in grey materials}, Quart.
  Appl. Math. \textbf{59} (2001), no.~4, 737--768, \doi{10.1090/qam/1866555}.

\bibitem{LS}
J.~L. Lebowitz and H.~Spohn, \emph{Transport properties of the {L}orentz gas:
  {F}ourier's law}, J. Statist. Phys. \textbf{19} (1978), no.~6, 633--654,
  \doi{10.1007/BF01011774}.

\bibitem{mercier1987application}
B.~Mercier, \emph{Application of accretive operators theory to the radiative
  transfer equations}, SIAM Journal on Mathematical Analysis \textbf{18}
  (1987), no.~2, 393--408.

\bibitem{mihalas2013foundations}
D.~Mihalas and B.~W. Mihalas, \emph{Foundations of radiation hydrodynamics},
  Courier Corporation, 2013.

\bibitem{Milne}
E.~A. Milne, \emph{The diffusion of imprisoned radiation through a gas}, J.
  London Math. Soc. \textbf{1} (1926), no.~1, 40--51,
  \doi{10.1112/jlms/s1-1.1.40}.

\bibitem{MR3356368}
A.~Nota, \emph{Diffusive limit for the random {L}orentz gas}, From particle
  systems to partial differential equations. {II}, Springer Proc. Math. Stat.,
  vol. 129, Springer, Cham, 2015, pp.~273--292,
  \doi{10.1007/978-3-319-16637-7\_10}.

\bibitem{Nouri2}
A.~Nouri, \emph{Stationary states of a gas in a radiation field from a kinetic
  point of view}, Ann. Fac. Sci. Toulouse Math. (6) \textbf{10} (2001), no.~2,
  361--390.

\bibitem{Nouri1}
\bysame, \emph{The evolution of a gas in a radiation field from a kinetic point
  of view}, J. Statist. Phys. \textbf{106} (2002), no.~3-4, 589--622,
  \doi{10.1023/A:1013758305917}.

\bibitem{oxenius}
J.~Oxenius, \emph{Kinetic theory of particles and photons}, Springer Series in
  Electrophysics, vol.~20, Springer-Verlag, Berlin, 1986,
  \doi{10.1007/978-3-642-70728-5}, Theoretical foundations of non-LTE plasma
  spectroscopy.

\bibitem{MR2076780}
M.~M. Porzio and \'{O}. L\'{o}pez-Pouso, \emph{Application of accretive
  operators theory to evolutive combined conduction, convection and radiation},
  Rev. Mat. Iberoamericana \textbf{20} (2004), no.~1, 257--275,
  \doi{10.4171/RMI/388}.

\bibitem{RSM}
A.~Rossani, G.~Spiga, and R.~Monaco, \emph{Kinetic approach for two-level atoms
  interacting with monochromatic photons}, Mechanics Research Communications
  \textbf{24} (1997), no.~3, 237--242.

\bibitem{rutten1995radiative}
R.~J. Rutten, \emph{Radiative transfer in stellar atmospheres}, Utrecht
  University lecture notes, 8th edition, 2003,
  \href{https://robrutten.nl/rrweb/rjr-pubs/2003rtsa.book.....R.pdf}{https://robrutten.nl/rrweb/rjr-pubs/2003rtsa.book.....R.pdf}.

\bibitem{MR0906921}
R.~Sentis, \emph{Half space problems for frequency dependent transport
  equations. {A}pplication to the {R}osseland approximation of the radiative
  transfer equations}, Proceedings of the conference on mathematical methods
  applied to kinetic equations ({P}aris, 1985), vol.~16, 1987, pp.~653--697,
  \doi{10.1080/00411458708204309}.

\bibitem{spiegel1957smoothing}
E.~A. Spiegel, \emph{The smoothing of temperature fluctuations by radiative
  transfer}, The Astrophysical Journal \textbf{126} (1957), 202.

\bibitem{Spohn1978}
H.~Spohn, \emph{The {L}orentz process converges to a random flight process},
  Communications in Mathematical Physics \textbf{60} (1978), no.~3, 277–290,
  \doi{https://doi.org/10.1007/BF01612893 10.1007/BF01612893}.

\bibitem{MR1471600}
T.~Tiihonen, \emph{A nonlocal problem arising from heat radiation on non-convex
  surfaces}, European J. Appl. Math. \textbf{8} (1997), no.~4, 403--416,
  \doi{10.1017/S0956792597003185}.

\bibitem{Westdickenberg}
M.~Westdickenberg, \emph{Some new velocity averaging results}, SIAM Journal on
  Mathematical Analysis \textbf{33} (2002), no.~5, 1007--1032,
  \arxiv{https://doi.org/10.1137/S0036141000380760},
  \doi{10.1137/S0036141000380760}.

\bibitem{MR3324149}
L.~Wu and Y.~Guo, \emph{Geometric correction for diffusive expansion of steady
  neutron transport equation}, Comm. Math. Phys. \textbf{336} (2015), no.~3,
  1473--1553, \doi{10.1007/s00220-015-2315-y}.

\end{thebibliography}

 \end{document}